\documentclass[]{amsart}
\usepackage{amsmath,amssymb,amsfonts}
\usepackage[initials,nobysame]{amsrefs}
\usepackage{hyperref}
\usepackage{cleveref}   
\usepackage{subcaption}
\usepackage{float}
\usepackage{graphicx,xcolor,comment} % Required for inserting images
\usepackage[top=2.5cm, bottom=3cm, left=2cm, right=2cm]{geometry}

\title{Existence of closed non-planar $p$-elasticae}
\author[F.~Gruen]{Florian Gruen}
\address[F.~Gruen]{Department of Mathematics, Graduate School of Science, Kyoto University, Kitashirakawa Oiwake-cho, Sakyo-ku, Kyoto 606-8502, Japan}
\email{gruen.florian.32r@st.kyoto-u.ac.jp}

\keywords{$p$-elastica, classification, boundary value problem, closed curve, torus-knot}
\subjclass[2020]{49Q10 and 53A04}

\numberwithin{equation}{section}
\newtheorem{theorem}{Theorem}[section]
\newtheorem{corollary}[theorem]{Corollary}

\newtheorem{lemma}[theorem]{Lemma}
\newtheorem{proposition}[theorem]{Proposition}
\theoremstyle{definition}
\newtheorem{definition}[theorem]{Definition}

\newtheorem{remark}[theorem]{Remark}

\DeclareMathOperator{\argmin}{argmin}

\DeclareMathOperator{\R}{\mathbb{R}}
\DeclareMathOperator{\N}{\mathbb{N}}
\DeclareMathOperator{\Q}{\mathbb{Q}}

\DeclareMathOperator{\sn}{sn}

\DeclareMathOperator{\Z}{\mathbb{Z}}

\newcommand{\eps}{\varepsilon}

\crefname{lemma}{Lemma}{Lemmas}
\Crefname{lemma}{Lemma}{Lemmas}
\crefname{theorem}{Theorem}{Theorems}
\Crefname{theorem}{Theorem}{Theorems}

\begin{document}

\begin{abstract}
  We prove existence of closed non-planar $p$-elasticae for general exponents $p\in (1,\infty)$. In particular, we show that for any $p \in (1,\infty)$, there exists a countable  family of non-planar $p$-elasticae, which are realized as torus knots. This generalizes  well-known results of Langer--Singer for the quadratic case $p=2$ to general exponents $p\in (1,\infty)$.
\end{abstract}

\maketitle

\section{Introduction}

The elastica is one of the oldest problems in the Calculus of Variations, originating in works of Euler and Bernoulli. 
Physically motivated by the modeling of an elastic rod \cites{Levien_berkely, elastica_influence}, it continues to be studied from pure and applied perspectives, revitalized with the pioneering works of Langer--Singer  (e.g.\ \cites{langer-singer_classification, langersinger_minmax, LS_Lagrangian_aspects, singer_lecturenotes}) to modern treatments \cites{miura_phase_transitions, elastica_survey, miura_LiYau, miuraUniqueness, kawohl_original, Bucur-Henrot, henrot2, bevilacquaVariationalAnalysisInextensible2022, dondl_confined, bellettini2026concentrationeffectsgammalimitelastica}. 
This work deals with a recent generalization of the classical squared bending energy, the $p$-bending energy, which nowadays attracts much interest in the research community \cites{acerbi, Japan_p, p-elastica_in_sphere, kawohl_general_p, miura_pinned_p, watanabe_flatcore, miuraclassification,  Pozzetta1, masnou_levellines, ambrosio_image, p-obstacle, ARROYO2003339, miuraStabilityFlatcorePinned2025,infinite_elastica_classification,infinity_elastica_manifold, miura_jlms}.

In particular, for $p\in (1,\infty)$, the \textit{$p$-bending energy} of an immersed curve $\gamma \in W^{2,p}(0,1;\R^n)$ is defined as 
\begin{equation*}
  \mathcal{B}_p[\gamma] := \int_\gamma |\kappa|^p ds,
\end{equation*}
where $\kappa$ is the curvature vector of $\gamma$. Note that $\mathcal{B}_p$ is invariant under reparameterization and Euclidean isometries.

The so-called \textit{$p$-elasticae} arise as critical points of $\mathcal{B}_p$ under a fixed-length constraint $\mathcal{L}[\gamma]=\int_\gamma ds = L$, giving rise to a Lagrange multiplier $\lambda$, i.e.\ more precisely we have the following definitions. 

\begin{definition}
  For $p\in (1,\infty)$, the immersed curve $\gamma \in W^{2,p}(0,1;\R^n)$ is called a \textit{$p$-elastica} if there exists $\lambda \in \R$ such that 
  \begin{equation*}
    \frac{d}{d\eps} \left( \mathcal{B}_p[\gamma+\eps \eta] + \lambda \mathcal{L}[\gamma + \eps \eta] \right) \bigg|_{\eps=0} = 0 \qquad \forall \eta \in C_c^\infty(0,1;\R^n).
  \end{equation*}
\end{definition}

\begin{definition}
  For $p\in (1,\infty)$, the immersed curve $\gamma \in W^{2,p}(\R/\Z;\R^n)$ is called a \textit{closed $p$-elastica} if there exists $\lambda \in \R$ such that 
  \begin{equation}
  \label{eq: first variation closed}
    \frac{d}{d\eps} \left( \mathcal{B}_p[\gamma+\eps \eta] + \lambda \mathcal{L}[\gamma + \eps \eta] \right) \bigg|_{\eps=0} = 0 \qquad \forall \eta \in C^\infty(\R/\Z;\R^n).
  \end{equation}
\end{definition}

Despite existence of minimizers (which follows from a straightforward direct method argument) to the $p$-bending energy under the fixed-length constraint with respect to various boundary conditions, it is a nontrivial question to ask whether global or local minimizers, and generally critical points, are planar or non-planar.
The case of pinned boundary conditions has been thoroughly analyzed in \cites{miura_pinned_p,miuraStabilityFlatcorePinned2025, miura_jlms}, but the case of general clamped boundary conditions remains widely open. 
Closed $p$-elasticae can be interpreted as a special case of clamped boundary conditions, see Remark~\ref{rmk: closed implies closed in C1} and Lemma~\ref{lemma: closedness conditions}. 

The main result of this work is a first proof of the existence of closed non-planar $p$-elasticae, numerical examples for varying $p$ are shown in Figure~\ref{fig: torus different p} and Figure~\ref{fig: cross section different p}.

\begin{theorem}
  \label{thm: main result}
  Let $p\in (1,\infty)$. Then there exists a countable family of closed non-planar $p$-elasticae in $\R^3$. 
  Each curve is analytic, embedded and lies on an embedded torus of revolution. 
  Moreover, each $(n,m)$-torus knot with $m>2n>0$ (where $m$, $n$ are coprime) is realized by some closed non-planar $p$-elastica.
\end{theorem}

For the quadratic case ($p=2$), existence and stability results for closed non-planar $p$-elasticae have been obtained by Langer and Singer \cites{langer-singer_classification,langersinger_minmax,singer_lecturenotes} in the 1980s. 
In summary, closed non-planar $2$-elasticae exist as embedded curves lying on tori of revolution, but are all unstable critical points and the only spatially stable closed $2$-elastica is the once-covered circle.

For general $p\in (1,\infty)$, Miura and Yoshizawa showed in \cite{miuraclassification} that closed planar $p$-elasticae are either a circle or a $p$-figure-eight, possibly multiply covered. 
Moreover, they also investigated stability in \cite{miura_General_Rigidity}. 
Their results rely on a novel characterization of planar $p$-elasticae in terms of $p$-elliptic functions, see also \cites{takeuchi,watanabe_flatcore}.

\begin{figure}[ht]
    \centering
    \begin{subfigure}{0.32\textwidth}
        \centering
        \includegraphics[trim={80px 100px 80px 100px}, clip, width=\linewidth]{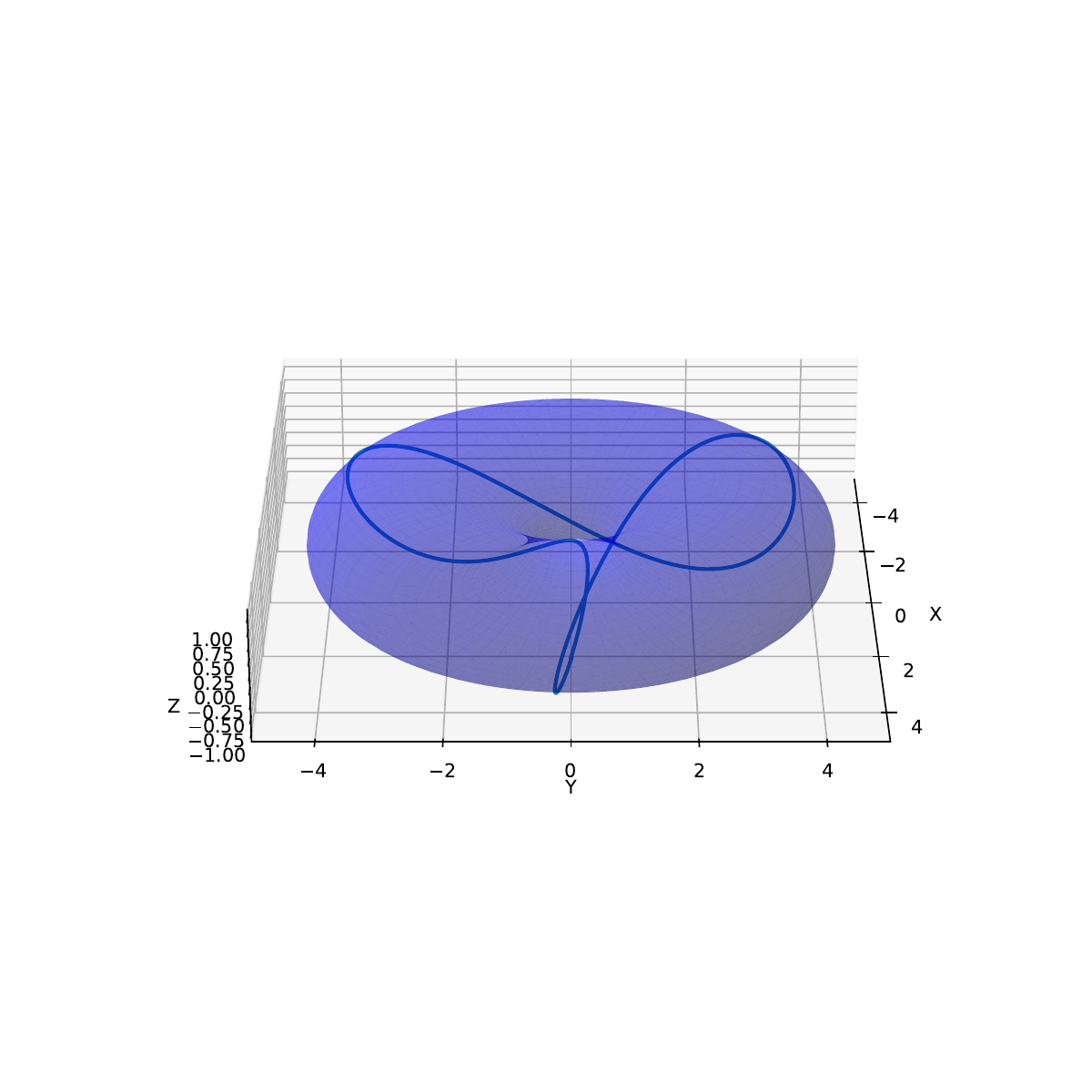}
        \caption{$p=1.5$}
    \end{subfigure}
    \hfill
    \begin{subfigure}{0.32\textwidth}
        \centering
        \includegraphics[trim={80px 100px 80px 100px}, clip, width=\linewidth]{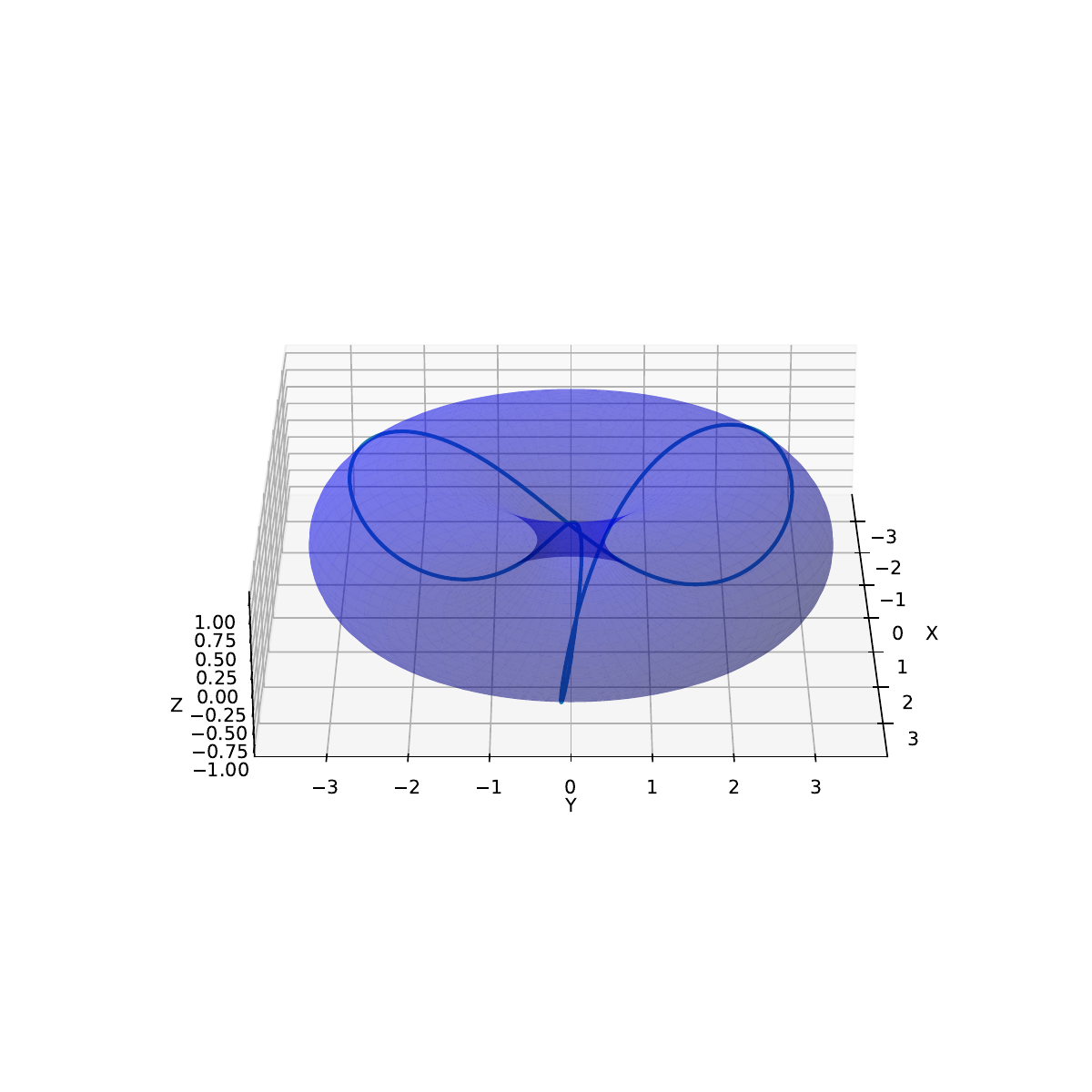}
        \caption{$p=2$}
    \end{subfigure}
    \hfill
    \begin{subfigure}{0.32\textwidth}
        \centering
        \includegraphics[trim={80px 100px 80px 100px}, clip, width=\linewidth]{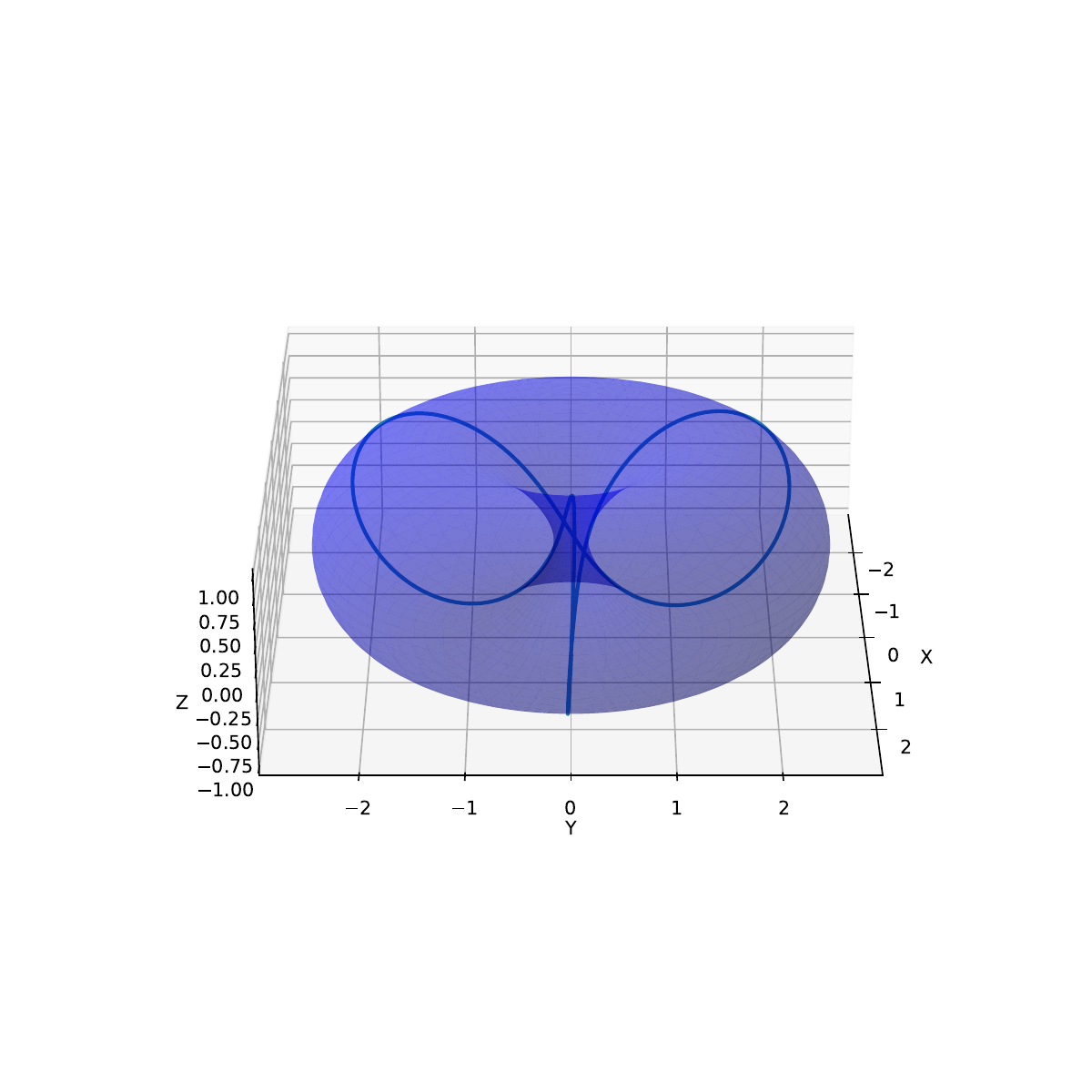}
        \caption{$p=4$}
    \end{subfigure}
    \caption{The $(1,3)$-torus knot $p$-elasticae for different $p$: For $p\to 1$, the torus knot becomes flat, for $p\to \infty$ ``upright''. Such ``flattening/uprising'' behavior for $p\to 1$ and $p\to \infty$ respectively has also been observed in \cite[Section~5]{ourpaper} in the context of leafed $p$-elasticae with self-intersections.}
    \label{fig: torus different p}
\end{figure}

\begin{figure}[ht]
    \centering
    \hspace{1cm}
    \begin{subfigure}{0.24\textwidth}
        \centering
        \includegraphics[trim={30px 30px 30px 155px}, clip, width=\linewidth]{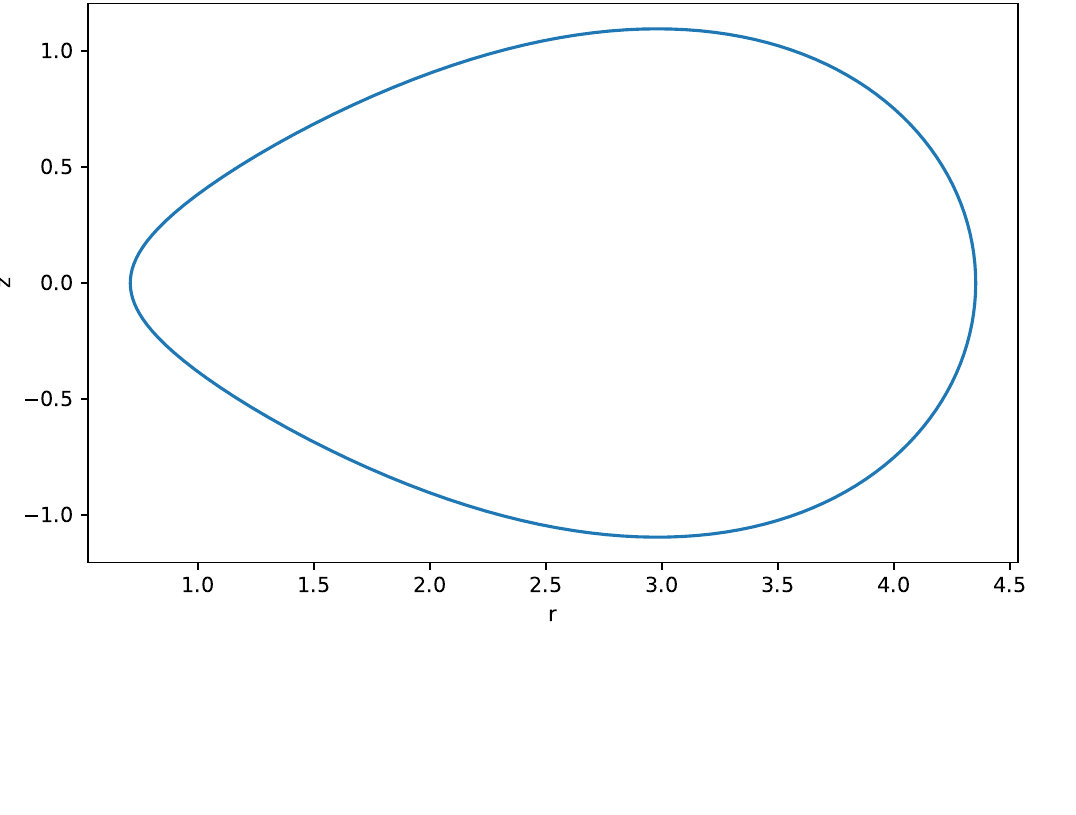}
        \caption{$p=1.5$}
    \end{subfigure}
    \hfill
    \begin{subfigure}{0.24\textwidth}
        \centering
        \includegraphics[trim={30px 30px 30px 125px}, clip, width=\linewidth]{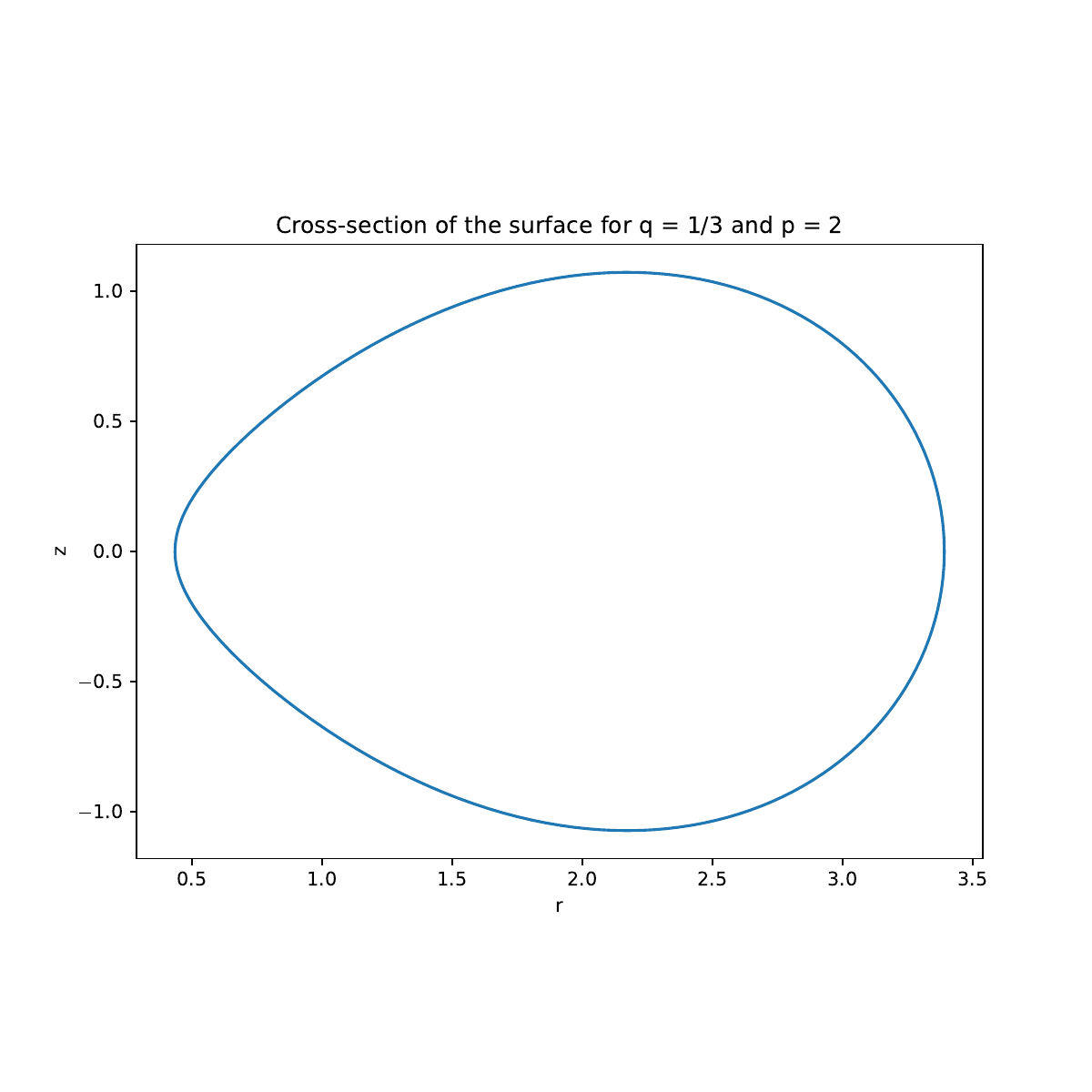}
        \caption{$p=2$}
    \end{subfigure}
    \hfill
    \begin{subfigure}{0.24\textwidth}
        \centering
        \includegraphics[trim={30px 30px 30px 95px}, clip, width=\linewidth]{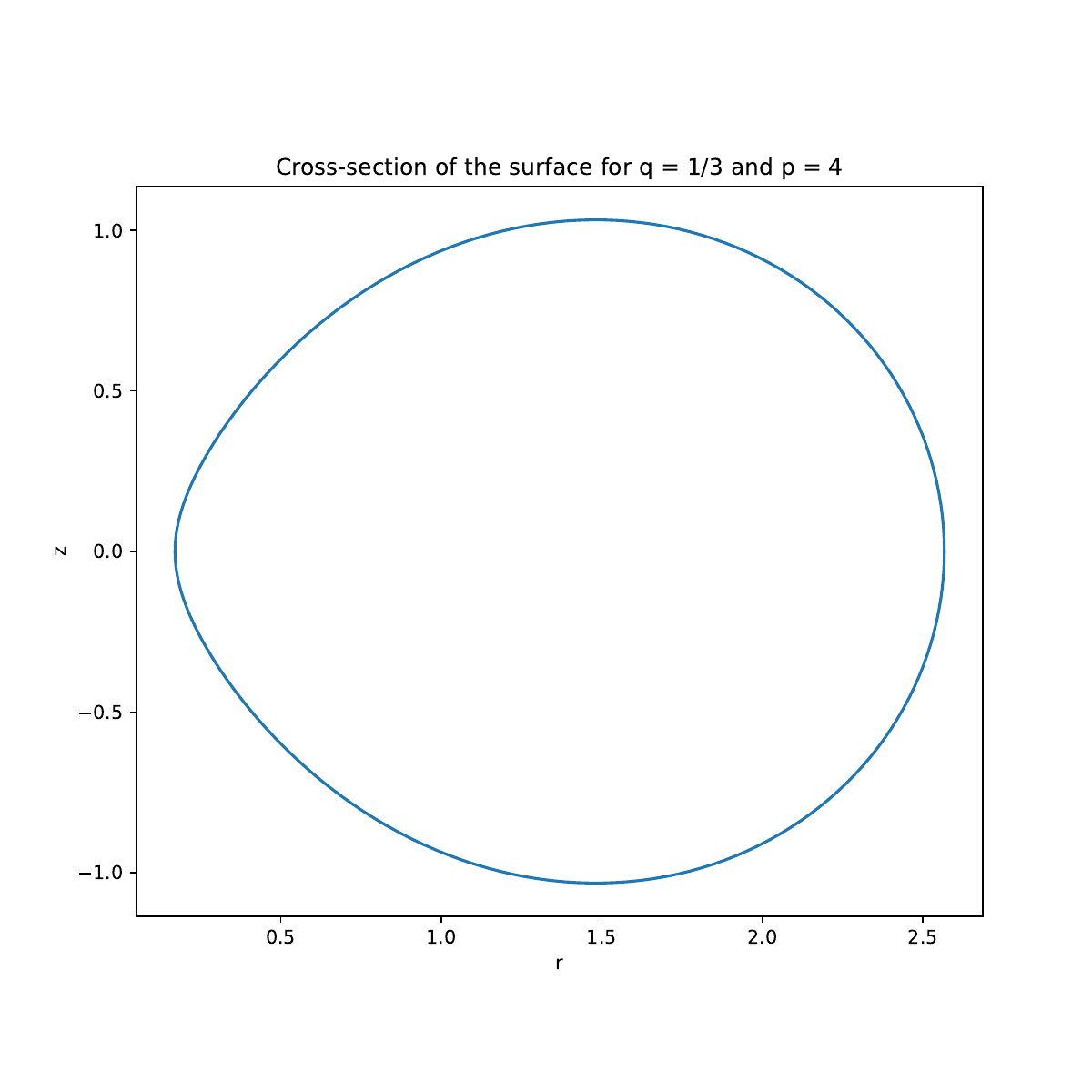}
        \caption{$p=4$}
    \end{subfigure}
    \hspace{1cm}
    \caption{The cross sections of the $(1,3)$-torus knot $p$-elasticae for different $p$: For $p\to 1$, the cross section becomes flat, for $p\to \infty$ circular.}
    \label{fig: cross section different p}
\end{figure}

\subsection{Idea of the proof}
We describe briefly the idea of the proof, building on results from \cite{ourpaper}. 
Recall that by computing the first variation \cites{miuraclassification, ourpaper}, any arclength reparameterized $p$-elastica 
$\gamma$
(hereafter $\gamma$ always denotes an arclength parameterized curve) satisfies the fourth order ODE
\begin{equation}
  \label{eq: first variation}
  p (|\gamma''|^{p-2}\gamma'')'' - ((1-2p)|\gamma''|^p\gamma' + \lambda \gamma')' = 0
\end{equation}
in the distributional sense, and if $\gamma$ is sufficiently smooth, pointwise. 
 Moreover, from \cite[Theorem~1.2, 1.3]{ourpaper} any non-planar $p$-elastica in $\R^n$ is either
\begin{itemize}
  \item smooth (even real-analytic) and three-dimensional (i.e.\ contained in a subspace of dimension three) or,
  \item a flat-core $p$-elastica, which is partially planar and only occurs when $p>2$.
\end{itemize}

Since flat-core $p$-elasticae do not yield closed $p$-elasticae, any closed non-planar $p$-elastica (if it exists) must be smooth, three dimensional and ``completely non-planar'' (i.e.\ no flat parts).
In that case, \eqref{eq: first variation} (or more precisely its restriction to the three-dimensional subspace) reduces to two ODEs for the scalar curvature $k:=|\kappa|$ and scalar torsion $\tau:= \frac{\det(\gamma',\gamma'',\gamma''')}{|\gamma''|^2}$.

\begin{theorem}\cite[Theorem~1.4]{ourpaper}
\label{thm:EL for k and tau}
    Let $p\in (1,\infty)$ and $\gamma:[0,L]\to\R^n$ be an analytic non-planar $p$-elastica, parameterized by arclength.
    Then $\gamma$ is a three-dimensional curve with $k,|\tau|>0$ on $[0,L]$, satisfying 
    \begin{equation}
    \label{eq:EL for k and tau}
    \begin{split}
        p (k^{p-1})'' + (p-1) k^{p+1} -p k^{p-1} \tau^2 -\lambda k &= 0, \\
         k^{2p-2} \tau &= C,
    \end{split}
    \end{equation}
    for some real constant $C\neq 0$.
\end{theorem}

We note that for $k:[0,L] \to (0,\infty)$ and $\tau:[0,L]\to \R$ satisfying \eqref{eq:EL for k and tau} with $\lambda$ and $C\neq0$ fixed, integrating the Frenet--Serret equations with some initial conditions $(T(0),N(0),B(0))$ gives a non-planar (as $k$, $C$ and so $\tau$ are all non-zero) $p$-elastica $\gamma_{\lambda,C^2}$. 
In other words, the problem reduces to studying the solutions of \eqref{eq:EL for k and tau}, that is, for which parameters $\lambda$ and $C$ is the resulting $p$-elastica $\gamma_{\lambda,C}$ closed 
and first of all, is it even possible to obtain closed curves for $C\neq 0$?
Since the sign of $C$ only changes the chirality of $\gamma_{\lambda,C^2}$, we assume from now on $C>0$.

In the seminal work \cite{langer-singer_classification},  solutions to \eqref{eq:EL for k and tau} with $p=2$ have explicit representation formulae in terms of classical Jacobi elliptic functions. The parameter space of solutions can be expressed in terms of the elliptic moduli from the explicit solution, but now, since explicit formulae for general $p\in (1,\infty)$ and $C\neq 0$ are currently unknown, we work directly in the $\lambda$-$C^2$ parameter space $S$, see Figure~\ref{fig: parameterspace}. 
This leads to a substantially more implicit analysis. A key contribution of the present work is to show that, through a careful combination of delicate estimates and asymptotic approximations, the overall strategy of \cite{langer-singer_classification} can nevertheless be carried out, even without relying on any explicit representation formula.

To be precise, after some preliminaries in Section~\ref{section:Preliminaries}, we set up the cylindrical coordinate system $(r,\theta, z)$ in Section~\ref{section Killing fields} and define the closedness quantities $\Delta z(\lambda,C^2)$ and $\Delta \theta(\lambda,C^2)$. 
In particular, a $p$-elastica $\gamma_{\lambda,C^2}$ is closed if and only if $\Delta z(\lambda,C^2)=0$ and $\Delta \theta(\lambda,C^2) \in 2\pi \Q$.
In Section~\ref{section: Delta z} we show that the zero level-set $\{\Delta z=0\}$ contains a connected subset $\Gamma$ connecting the planar $p$-figure-eight to the circular elasticae in the parameter space $S$. Lastly, in Section~\ref{section: Delta theta}, we show that the limits of $\Delta \theta$ at the two endpoints of $\Gamma$ are $-\pi$ and $0$ respectively, and that all intermediate values are attained.

\begin{figure}[ht]
  \centering
  \includegraphics*[width=0.7\textwidth]{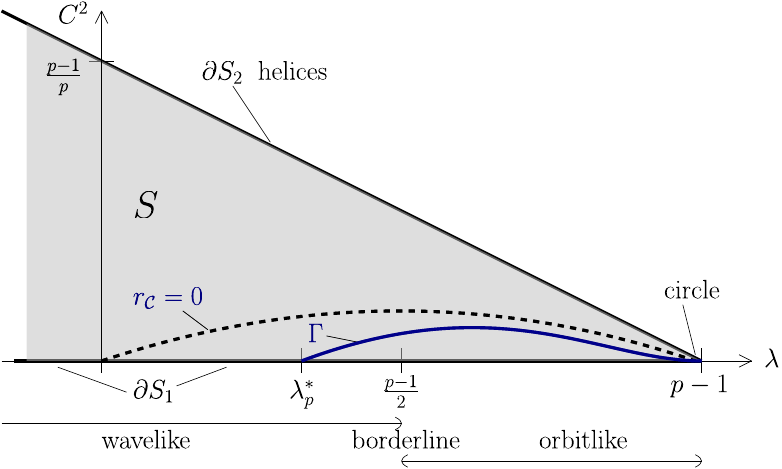}
  \caption{The complete $\lambda$-$C^2$ parameter space $S$: Planar $p$-elasticae are situated on the $\lambda$-axis, ranging from the circle at $(p-1,0)$ over orbitlike $p$-elasticae to the borderline/flatcore $p$-elastica at $(\frac{p-1}{2},0)$. 
  For smaller $\lambda$, the $p$-elasticae are wavelike and the straight line appears ``in the limit at negative infinity''. 
  The line $C^2=\tfrac{p-1-\lambda}{p}$ contains helices, ranging again from the circle to the straight line. The interior contains non-planar, non-constant $p$-elasticae.}
  \label{fig: parameterspace}
\end{figure}

\subsection{Acknowledgments}
The author thanks Tatsuya Miura for helpful advice and discussions.

\section{Preliminaries}
\label{section:Preliminaries}

We first remark that closed $p$-elastica can be interpreted as a special case of clamped boundary conditions, which enables us to use Theorem~\ref{thm:EL for k and tau} and work directly with the ODE.

\begin{remark}
\label{rmk: closed implies closed in C1}
From a closed $p$-elastica $\gamma \in W^{2,p}(\R/\Z;\R^n)$, we identify its open curve counterpart, i.e.\  $\tilde \gamma \in W^{2,p}(0,1;\R^n)$ is a $p$-elastica with $\gamma(0)=\gamma(1)$ and $\gamma'(0)=\gamma'(1)$.
We refer to Lemma~\ref{lemma: closedness conditions} for the reverse implication.
Furthermore, from \cite[Theorem~1.2, 1.3]{ourpaper}, it follows that every closed non-planar $p$-elastica is smooth and \eqref{eq:EL for k and tau} holds pointwise. This is not the case for planar closed $p$-elasticae, since the $p$-figure-eight curve might lose regularity of its higher-order derivatives at inflection points, see \cite{miuraclassification}.
\end{remark}

Substituting the second equation in the first in \eqref{eq:EL for k and tau}, we have for a non-planar $p$-elastica
\begin{equation}
      \label{eq:EL for k}
        p (k(s)^{p-1})'' + (p-1) k(s)^{p+1} -pC^2 k(s)^{3-3p} -\lambda k(s) = 0,
\end{equation}
and upon using the substitution $w=k^{p-1}$, 
\begin{equation}
    \label{eq:EL pointwise for w}
        w''(s) + \frac{p-1}{p} w(s)^{\frac{p+1}{p-1}} - C^2 w(s)^{-3} - \frac{\lambda}{p} w(s)^{\frac{1}{p-1}}=0.
\end{equation}

We start by showing the natural scaling invariance and assume hereafter that $\|k\|_{L^\infty} = \|w\|_{L^\infty}=1$.

\begin{lemma}
  Let $w$ be a solution to \eqref{eq:EL pointwise for w}. Then for $A>0$, the function $\tilde w := A w(A^{\frac{1}{p-1}}\cdot)$ is also a solution to \eqref{eq:EL pointwise for w} with $\tilde{\lambda} = A^{\frac{p}{p-1}}\lambda$ and $\tilde{C}^2 = A^{\frac{4p-2}{p-1}}C^2 $.
\end{lemma}

\begin{proof}
  We have 
  \begin{equation*}
  \begin{split}
   &\tilde w''(s) + \frac{p-1}{p} \tilde w(s)^{\frac{p+1}{p-1}} -\tilde C^2 \tilde w(s)^{-3} - \frac{\tilde \lambda}{p} \tilde w(s)^{\frac{1}{p-1}} \\ 
   &\quad = AA^{\frac{2}{p-1}} w''(A^{\frac{1}{p-1}}s) +  \frac{p-1}{p} A^{\frac{p+1}{p-1}} w(A^{\frac{1}{p-1}}s)^{\frac{p+1}{p-1}} -  A^{\frac{4p-2}{p-1}} A^{-3}w(A^{\frac{1}{p-1}}s)^{-3} - A^{\frac{p}{p-1}} \frac{\lambda}{p} A^{\frac{1}{p-1}}w(A^{\frac{1}{p-1}}s)^{\frac{1}{p-1}} \\
   &\quad =A^{\frac{p+1}{p-1}} \left( w''(A^{\frac{1}{p-1}}s) + \frac{p-1}{p} w(A^{\frac{1}{p-1}}s)^{\frac{p+1}{p-1}} - C^2 w(A^{\frac{1}{p-1}}s)^{-3} - \frac{\lambda}{p} w(A^{\frac{1}{p-1}}s)^{\frac{1}{p-1}} \right)  = 0,
  \end{split}
  \end{equation*}
  as we wanted to show.
\end{proof}

If $p=2$, the solution to \eqref{eq:EL for k} is given explicitly as
\begin{equation*}
  k^2(s) = 1-\frac{q^2}{w^2} \sn^2 \left(\frac{1}{2w}s,q\right),
\end{equation*}
where $2\lambda = \frac{3w^2-q^2-1}{w^2}$ and $4C^2 = \frac{(1-w^2)(w^2-q^2)}{w^4}$ relate the parameters $q$ and $w$ to $\lambda$ and $C^2$. 
In \cite{langer-singer_classification}, the analysis is carried out using $q$ (named $p$ in \cite{langer-singer_classification}, renamed here to $q$ to avoid confusion with the exponent $p$) and $w$ in the parameter space $0\leq q \leq w\leq 1$.
However, to our knowledge such an explicit solution form is not known for general $p\in (1,\infty)$ and $C^2>0$ and thus we work implicitly with the ODE, depending on the parameters $\lambda$ and $C^2$.

The admissible parameter space in the $\lambda$-$C^2$ plane is also a triangle, but of infinite size, see Figure~\ref{fig: parameterspace}. For $\|w\|_{L^\infty}=1$ it is given by $\overline S$ where
\begin{equation*}
\begin{split}
  S:&= \{(\lambda, C^2) \in \R^2: 0<C^2 < \tfrac{p-1-\lambda}{p} \},
  \\ \partial S &= \partial S_1 \cup \partial S_2 \text{ with }
  \partial S_1= (-\infty,p-1] \times \{0\} \text{ and }
  \partial S_2 = \{(\lambda, C^2) \in \R^2: 0\leq C^2 = \tfrac{p-1-\lambda}{p} \}. 
\end{split}
\end{equation*}
The reason for the boundary curve being $C^2=\tfrac{p-1-\lambda}{p}$ will become apparent in the proof of Proposition~\ref{prop: existence}. 
We work mainly in $S^0 = S \cap \{(x,y):x\geq 0\}$, with boundary $\partial S^0 = \partial S^0_1 \cup \partial S^0_2 \cup \partial S^0_3$, where
\begin{equation*}
  \begin{split}
    \partial S_1^0 &= [0,p-1] \times \{0\} ,\\
    \partial S_2^0 &= \{(\lambda, C^2) \in \R_{\geq 0} \times \R_{\geq 0}: C^2 = \tfrac{p-1-\lambda}{p} \},\\
    \partial S_3^0 &= \{0\} \times [0,\tfrac{p-1}{p}].\\
  \end{split}
\end{equation*}
 Furthermore, let $S^\rho = S^0 \setminus B_\rho((\frac{p-1}{2},0))$, where $\rho>0$ is a small parameter, which will be specified later (see Figure~\ref{fig: S rho domain}). 
 In particular, $\partial S^\rho = \partial S_1^\rho \cup \partial S_2^\rho \cup \partial S_3^\rho \cup \partial S_4^\rho$, where 
\begin{equation*}
  \begin{split}
    \partial S_1^\rho &= \left( [0,p-1] \times \{0\} \right) \setminus \left( \left( \tfrac{p-1}{2}-\rho, \tfrac{p-1}{2}+\rho \right) \times \{0\} \right),\\
    \partial S_2^\rho &= \{(\lambda, C^2) \in \R_{\geq 0} \times \R_{\geq 0}: C^2 = \tfrac{p-1-\lambda}{p} \}, \\
    \partial S_3^\rho &= \{0\} \times [0,\tfrac{p-1}{p}],\\
    \partial S_4^\rho &= \partial B_\rho((\tfrac{p-1}{2},0)) \cap \{(x,y):y>0\}.
  \end{split}
\end{equation*}

\begin{figure}[ht]
  \includegraphics*[width=0.7\textwidth]{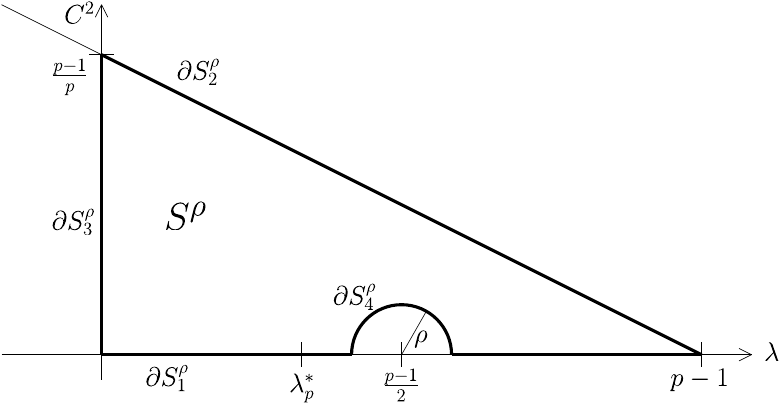}
  \caption{The domain $S^\rho$}
  \label{fig: S rho domain}
\end{figure}

First, we prove existence and uniqueness of solutions to \eqref{eq:EL pointwise for w} for a parameter pair $(\lambda, C^2)\in \overline S \setminus (\frac{p-1}{2},0)$, i.e.\ we also include planar curves (where formally $C^2=0$).

\begin{proposition}
\label{prop: existence}
    Let $p\in (1,\infty)$. 
    \begin{itemize}
        \item If $(\lambda,C^2) \in \partial S_2$, then \eqref{eq:EL pointwise for w} has exactly one strictly positive solution $w=w_{\lambda,C^2}:\R \to \R$ with $w(0)=\|w\|_{L^\infty}=1$, namely the constant solution $w \equiv1$.
        \item If $(\lambda,C^2) \in S$, then \eqref{eq:EL pointwise for w} has exactly one strictly positive solution $w=w_{\lambda,C^2}:\R \to \R$ with $w(0)=\|w\|_{L^\infty}=1$, which is even, smooth and periodic with period $2P=2P(\lambda,C^2) \in (0,\infty)$. 
        \item If $(\lambda,C^2)=(\lambda,0)$ with $\frac{p-1}{2}<\lambda<p-1$, then \eqref{eq:EL pointwise for w} has exactly one strictly positive solution $w=w_{\lambda,C^2}:\R \to \R$ with $w(0)=\|w\|_{L^\infty}=1$, which is even, smooth and periodic with period $2P=2P(\lambda,0) \in (0,\infty)$. 
        \item If $(\lambda,C^2)=(\lambda,0)$ with $\lambda<\frac{p-1}{2}$, then there exists $P=P(\lambda,0)$ and a unique nonnegative continuous function $w=w_{\lambda,C^2}:\R \to \R$ satisfying the following properties:
        \begin{itemize}
            \item $w$ is even;
            \item $w$ is periodic with period $2P=2P(\lambda,0) \in (0,\infty)$;
            \item $w(0)=\|w\|_{L^\infty}=1$;
            \item   $w$ is smooth and satisfies \eqref{eq:EL pointwise for w} in $\R \setminus \{(2j-1)P:j\in \Z\}$;
            \item $w((2j-1)P)=0$ for every $j\in \Z$.
        \end{itemize}
    \end{itemize}
\end{proposition}

\begin{proof}
  Multiplying \eqref{eq:EL pointwise for w} by $2w'$ and integrating once more gives eventually
    \begin{equation}
    \label{eq:EL for w first order}
         w'(s)^2 + \frac{(p-1)^2}{p^2}  w(s)^{\frac{2p} {p-1}} - 2\lambda \frac{p-1}{p^2}w(s)^{\frac{p}{p-1}} + C^2 w(s)^{-2} =E, 
    \end{equation}
  with $E\in \R$ a constant of integration. 
  Thus we write $H_{\lambda,C^2}(w,w')=E$, with the Hamiltonian $H_{\lambda,C^2}:(0,\infty) \times \R \to \R$ being given by
  \begin{equation}
    H_{\lambda,C^2}(x,y) = y^2 + V(x) =
    y^2 + \frac{(p-1)^2}{p^2}  x^{\frac{2p} {p-1}} - 2\lambda \frac{p-1}{p^2} x^{\frac{p}{p-1}} + C^2 x^{-2}.
  \end{equation}
  We note that $H_{\lambda,C^2}(x,y)\to \infty$ as $|x|,|y| \to \infty$ and $\lim_{|(x,y)|\to 0} H_{\lambda,C^2}(x,y)= \begin{cases}
    0 &\text{ if } C=0, \\ \infty &\text{ if } C>0.
  \end{cases}$

 \noindent The critical points of $H_{\lambda,C^2}$ are given by
  \begin{equation*}
    0=\nabla H_{\lambda,C^2}(x,y) = (V'(x),2y) = 2\left( \frac{p-1}{p} x^{\frac{p+1}{p-1}} - C^2 x^{-3} - \frac{\lambda}{p} x^{\frac{1}{p-1}}, y \right).
  \end{equation*} 

  \underline{Case $C^2>0$:}
  The first term in $x$ vanishes only at one point $x_0$ by Lemma~\ref{lemma: unique root} and thus by the structure of $H_{\lambda,C^2}$ the only critical point $(x_0,0)$ is a global minimum.

  Suppose that  $(\lambda,C^2) \in \partial S_2$. Then $C^2 = \frac{p-1-\lambda}{p}$ and so $1$ is a root of $V'$, which is unique by Lemma~\ref{lemma: unique root}. 
  Thus $(1,0)$ is the unique minimum of $H_{\lambda,C^2}$, and $\min V = \min H_{\lambda,C^2}$. 
  Hence $(w,w')=(1,0)$ is the only solution to \eqref{eq:EL for w first order} and also \eqref{eq:EL pointwise for w} with $\|w\|_{L^\infty} = 1$. 
  Any other solution $w$ would have $H_{\lambda,C^2}(w,w')>\min H_{\lambda,C^2}$ and thus at points where $w'=0$, we have $H_{\lambda,C^2}(w,0)=V(w)>\min H_{\lambda,C^2}$ by Lemma~\ref{lemma: number of solutions}, which implies that $\|w\|_{L^\infty}>\argmin V = 1$.

  Suppose that  $(\lambda,C^2)\in S$, i.e.\
   \begin{equation*}
    0 <   \frac{p-1-\lambda}{p} - C^2  \qquad \text{and} \qquad C^2>0.
  \end{equation*}
  By Lemma~\ref{lemma: number of solutions}, this implies that the unique minimum $x_0=\argmin V(x)$ is strictly less than $1$. 
  We set 
  \begin{equation*}
     E = E(\lambda,C^2) := \frac{(p-1)^2}{p^2}-2\lambda \frac{p-1}{p^2}+C^2,  
  \end{equation*}
  then $E = H_{\lambda,C^2}(1,0) > H_{\lambda,C^2}(x_0,0) =\min H_{\lambda,C^2}$.
  Thus the level set $H_{\lambda,C^2}^{-1}(E)\cap ((0,\infty) \times \R)$ gives a positive periodic solution to \eqref{eq:EL for w first order} and thereby \eqref{eq:EL pointwise for w}. 
  Denote by $0<w_m(E)<w_M(E)=1$ the two values such that $V(w_m(E)) = V(w_M(E)) = E$, i.e. the minimum and maximum amplitude of the solution $w$ where $w'=0$. 
  The fact that there are exactly two such points where $x\mapsto V(x)-E$ vanishes follows also from Lemma~\ref{lemma: number of solutions}. 
  Moreover, by the same lemma, any other level set has $w_M \neq 1$ (since $V'>0$), hence unique existence. 
  
  From \eqref{eq:EL for w first order}, we get $\frac{dw}{ds} = \pm \sqrt{ E-V(w)}$ and therefore one (half-)period of $w$ from the minimum $w_m$ to the maximum $1$ is given as 
  \begin{equation*}
    \begin{split}
      P(\lambda,C^2)&= \int_{w_m}^1 \frac{ds}{dw} dw = \int_{w_m}^1 \frac{dw}{\frac{dw}{ds}}   = \int_{w_m}^{1} \frac{dw}{\sqrt{ E-V(w)}} = \int_{w_m}^{1} \frac{dw}{\sqrt{ E-\frac{(p-1)^2}{p^2}  w^{\frac{2p} {p-1}} + 2\lambda \frac{p-1}{p^2} w^{\frac{p}{p-1}} - C^2 w^{-2}}}.
    \end{split}
  \end{equation*} 
  Smoothness follows from a bootstrap argument in \eqref{eq:EL pointwise for w}, as $w\geq w_m>0$.

  \underline{Case $C=0$:}
  This case is the planar setting from \cite{miuraclassification} with maximal curvature $k_M=w_M=1$, attained at $s=0$. 
 For $\lambda>\frac{p-1}{2}$, the $p$-elastica is orbitlike, i.e.\ its curvature is strictly positive (with $w_m>0$ being the minimum of $w$) and so \eqref{eq:EL pointwise for w} is equivalent to \cite[Eq~(4.2)]{miuraclassification}. Then \eqref{eq:EL for w first order} holds for $s \in \R$ and from $\frac{dw}{ds} = \pm \sqrt{ E-V(w)}$ we obtain as before
  \begin{equation*}
      P(\lambda,0) = \int_{w_m}^1 \frac{ds}{dw}dw = \int_{w_m}^{1} \frac{dw}{\sqrt{ E-V(w)}} = \int_{w_m}^1 \frac{dw}{\sqrt{ \frac{(p-1)^2}{p^2} - 2\lambda\frac{p-1}{p^2}-\frac{(p-1)^2}{p^2}  w^{\frac{2p} {p-1}} + 2\lambda \frac{p-1}{p^2} w^{\frac{p}{p-1}} }}.
  \end{equation*}

  For $\lambda<\frac{p-1}{2}$, the $p$-elastica is wavelike, i.e.\ its curvature changes sign. Let $(-P,P)$ denote the maximal interval where $k$ (and so $w$) are strictly positive, i.e.\ \eqref{eq:EL pointwise for w} and \eqref{eq:EL for w first order} hold pointwise and so again
  \begin{equation*}
      P(\lambda,0) = \int_{0}^1 \frac{ds}{dw}dw = \int_{0}^{1} \frac{dw}{\sqrt{ E-V(w)}} = \int_{0}^1 \frac{dw}{\sqrt{ \frac{(p-1)^2}{p^2} - 2\lambda\frac{p-1}{p^2}-\frac{(p-1)^2}{p^2}  w^{\frac{2p} {p-1}} + 2\lambda \frac{p-1}{p^2} w^{\frac{p}{p-1}} }}.
  \end{equation*}
  Extending the solution on $(-P,P)$ by $2P$-periodicity yields the desired  function $w:\R \to \R$.
\end{proof}

\begin{remark}
  \label{rmk: k prime}
  We also note that on $[-P(\lambda,C^2),0]$,
  \begin{equation*}
  \begin{split}
    k'(s)=k'&=(w^{\frac{1}{p-1}})' = \frac{1}{p-1}w^{\frac{2-p}{p-1}}w' = \frac{1}{p-1}w^{\frac{2-p}{p-1}} \sqrt{E- \frac{(p-1)^2}{p^2} w^{\frac{2p}{p-1}} + 2\lambda \frac{p-1}{p^2} w^{\frac{p}{p-1}} - C^2 w^{-2}} \\
    &=\frac{1}{p-1}\sqrt{E w^{2\frac{2-p}{p-1}} -\frac{(p-1)^2}{p^2} w^{\frac{4}{p-1}} + 2\lambda \frac{p-1}{p^2} w^{\frac{4-p}{p-1}} - C^2 w^{\frac{6-4p}{p-1}}  } \\
    &=\frac{1}{p-1}\sqrt{E k^{4-2p} -\frac{(p-1)^2}{p^2} k^{4} + 2\lambda \frac{p-1}{p^2} k^{4-p} - C^2 k^{6-4p}  }.
  \end{split}
  \end{equation*}
  In particular, this implies that $w$ and $k$ are monotone between their minima and maxima.
\end{remark}

Throughout the paper, we analyze the functions (which are related by change of variable)
\begin{align*}
  f(k) &= \left(\frac{(p-1)^2}{p^2}-2\lambda \frac{p-1}{p^2}+C^2\right) - \frac{(p-1)^2}{p^2}  k^{2p} + 2\lambda \frac{p-1}{p^2} k^{p} - C^2 k^{2-2p},\\
  g(w) &= f(w^{\frac{1}{p-1}})= \left(\frac{(p-1)^2}{p^2}-2\lambda \frac{p-1}{p^2}+C^2\right) -\frac{(p-1)^2}{p^2}  w^{\frac{2p}{p-1}} + 2\lambda \frac{p-1}{p^2} w^{\frac{p}{p-1}} - C^2 w^{-2} ,\\
  h(u) &= f(u^\frac{1}{p}) = \left(\frac{(p-1)^2}{p^2}-2\lambda \frac{p-1}{p^2}+C^2\right) - \frac{(p-1)^2}{p^2}  u^{2} + 2\lambda \frac{p-1}{p^2} u - C^2 u^{\frac{2}{p}-2}, \\
  \phi(u) &= u^{\frac{2}{p}}h(u) = u^{\frac{2}{p}} \left(\frac{(p-1)^2}{p^2}-2\lambda \frac{p-1}{p^2}+C^2\right) - \frac{(p-1)^2}{p^2}  u^{2+\frac{2}{p}} + 2\lambda \frac{p-1}{p^2} u^{1+\frac{2}{p}} - C^2 u^{\frac{4}{p}-2}.
\end{align*}

\subsection{Estimates for the minimal curvature}
Recall that for a solution $k$ to \eqref{eq:EL for k} the maximal curvature is $k_M=1$, but the minimal curvature is more delicate. We define $k_m$ and $k_M=1$ to be the roots to $f(k)$.
Note that this gives $w_m:=k_m^{p-1}$ (the minimum amplitude of the solution $w$ in Proposition~\ref{prop: existence}) and $1$ as well as $u_m:=k_m^p$ and $1$ as the roots to $g(w)$ and $h(u)$ respectively.

\begin{remark}
  In the case $p=2$, thanks to the polynomial structure of \eqref{eq:EL for w first order}, we explicitly compute $u_m=w_m^2=k_m^2=\lambda-\frac{1}{2}+\sqrt{(\lambda-\frac{1}{2})^2 + 4C^2}$, but this is not possible for general $p\neq 2$.
\end{remark}

In the general case, we have to resort to asymptotic estimates, which differ depending on the region in $S^0$.
For points close to $\partial S_1$, we have the following estimates.
\begin{proposition}
  \label{prop: estimates for um} 
  Let $(\lambda,C^2) \in S^0$, i.e.\ $C^2>0$.
  \begin{enumerate}
    \item[(i)] If $0\leq \lambda < \frac{p-1}{2}$ and $\gamma_p = \min\left(\frac{4p}{2p-2}, \frac{6p-4}{2p-2}\right)$, then there exist $B_{\lambda,p}, \eps_{\lambda,p} \in (0,\infty)$ such that for all $C^2<\eps_{\lambda,p}$
    \begin{equation*}
      \left(\frac{(p-1)(p-1-2\lambda)}{p^2} \right)^{\frac{p}{2-2p}} C^{\frac{2p}{2p-2}} -B_{\lambda,p} C^{\gamma_p} < u_m < \left(\frac{(p-1)(p-1-2\lambda)}{p^2} \right)^{\frac{p}{2-2p}} C^{\frac{2p}{2p-2}}.
    \end{equation*} 
    Moreover, for $0<\delta\ll 1$ fixed, if $\lambda \in [0,\frac{p-1}{2}-\delta]$, then $B_{\lambda,p} \in [m_{\delta,p}, M_{\delta,p}]$ for some constants $0<m_{\delta,p}, M_{\delta,p}$.

    \item[(ii)] If $\frac{p-1}{2}<\lambda < p-1$, then there exists $B_{\lambda,p}>0$, continuous in $\lambda$ with $\lim_{\lambda\to p-1}B_{\lambda,p}=\frac{5p}{2(p-1)}$, and $\eps_{\lambda,p} \in (0,\infty)$ such that for all $C^2<\eps_{\lambda,p}$
    \begin{equation*}
      \frac{2\lambda}{p-1}-1  < u_m < \frac{2\lambda}{p-1}-1 + B_{\lambda,p} C^{2}.
    \end{equation*}
    Also for $0<\delta\ll 1$ fixed, if $\lambda \in [\frac{p-1}{2}+\delta,p-1)$, then $B_{\lambda,p} \in [m_{\delta,p}, M_{\delta,p}]$ for some constants $0<m_{\delta,p}, M_{\delta,p}$.
  \end{enumerate}

\end{proposition}

\begin{proof}
  Recall that $u_m=k_m^p$ is the only solution in $(0,1)$ to 
  \begin{equation*}
    h(u) = \left(\frac{(p-1)^2}{p^2}-2\lambda \frac{p-1}{p^2}+C^2\right) - \frac{(p-1)^2}{p^2}  u^{2} + 2\lambda \frac{p-1}{p^2} u - C^2 u^{\frac{2}{p}-2},
  \end{equation*}
  by Lemma~\ref{lemma: number of solutions}, and approaches $-\infty$ as $u\to 0$.\\

  \underline{Case $0\leq \lambda < \frac{p-1}{2}$:} Take the ansatz $u=AC^{\frac{2p}{2p-2}}-BC^{\gamma_p}$ with $ \gamma_p = \min \left( \frac{4p}{2p-2}, \frac{6p-4}{2p-2} \right) >2, \frac{2p}{2p-2}$.
  By Remark~\ref{rmk:first order Taylor},
  \begin{equation*}
    \begin{split}
      h(u) &= \left(\frac{(p-1)^2}{p^2}-2\lambda \frac{p-1}{p^2}+C^2\right) - \frac{(p-1)^2}{p^2}  (AC^{\frac{2p}{2p-2}}-BC^{\gamma_p})^{2} \\
      &\quad + 2\lambda \frac{p-1}{p^2} (AC^{\frac{2p}{2p-2}}-BC^{\gamma_p}) - C^2 (AC^{\frac{2p}{2p-2}}-BC^{\gamma_p})^{\frac{2-2p}{p}} \\
      &= \left(\frac{(p-1)^2}{p^2}-2\lambda \frac{p-1}{p^2}+C^2\right) + 2\lambda \frac{p-1}{p^2} A C^{\frac{2p}{2p-2}} \\ &\quad - C^2 \left( A^{\frac{2-2p}{p}}C^{-2} - \frac{2-2p}{p} A^{\frac{2-3p}{p}}B C^{\gamma_p + \frac{4-6p}{2p-2}} +O(C^{2\gamma_p + \frac{4-8p}{2p-2}}) \right) + O(C^{\gamma_p}) \\
      &= \left(\frac{(p-1)^2}{p^2}-2\lambda \frac{p-1}{p^2}- A^{\frac{2-2p}{p}}\right) + C^2 + 2\lambda \frac{p-1}{p^2} A C^{\frac{2p}{2p-2}} - \frac{2p-2}{p} A^{\frac{2-3p}{p}}B C^{\gamma_p - \frac{2p}{2p-2}} + O(C^{\min(\frac{4p}{2p-2},4)}), \\
    \end{split}
  \end{equation*}
  with the constant in $O(C^{\min(\frac{4p}{2p-2},4)})$ depending only on $A,B,\lambda, p$.
Setting $A=A_{\lambda,p}=\left(\frac{(p-1)(p-1-2\lambda)}{p^2} \right)^{\frac{p}{2-2p}}>0$ (since $\lambda < \frac{p-1}{2}$) gives
  \begin{equation*}
    \begin{split}
      h(u) & = C^2 + 2\lambda \frac{p-1}{p^2} A C^{\frac{2p}{2p-2}} - \frac{2p-2}{p} A^{\frac{2-3p}{p}}B C^{\gamma_p - \frac{2p}{2p-2}} + O(C^{\min(\frac{4p}{2p-2},4)}).
    \end{split}
  \end{equation*}
Take now $B=0$, thus
  \begin{equation*}
    h(u)  = C^2 + 2\lambda \frac{p-1}{p^2} A C^{\frac{2p}{2p-2}}  + O(C^{\min(\frac{4p}{2p-2},4)}),
  \end{equation*}
  which is strictly positive for $C^2<\eps_{\lambda,p}$ sufficiently small. Hence $u_m< AC^{\frac{2p}{2p-2}}$. 

  It remains to show the lower bound.
  If $p< 2$, then $2<\frac{2p}{2p-2}$, the leading order term is $C^2$ and as $4<\frac{4p}{2p-2}$,  
  \begin{equation*}
    h(u) = C^2 \left(1 - \frac{2p-2}{p} A^{\frac{2-3p}{p}}B \right) + O(C^{\min(\frac{2p}{2p-2},4)}).
  \end{equation*}
  Taking now $B=B_{\lambda,p}=2\frac{p}{2p-2}A^{\frac{3p-2}{p}}$ gives $h(u) = -C^2 + O(C^{4})$ which is strictly negative for $C^2<\eps_{\lambda,p}$ sufficiently small, that is $u_m > AC^{\frac{2p}{2p-2}}-BC^{\gamma_p}$.

  If $p=2$, the leading order term is of size $C^2$, so for $\gamma_p = 4$, then 
  \begin{equation*}
    h(u) = C^2 \left(1+\frac{\lambda}{2} A - A^{-2} B \right)  + O(C^{4}).
    \end{equation*}
  Taking $B=A^2(2+\frac{\lambda}{2}A)$ gives $h(u) = -C^2 + O(C^4)$ and we conclude as before.
  
  If $p>2$, we have $2>\frac{2p}{2p-2}$, the leading order term is of order $C^{\frac{2p}{2p-2}}$ and $4> \frac{4p}{2p-2}>2$, thus
  \begin{equation*}
    h(u) = C^{\frac{2p}{2p-2}} \left( 2\lambda \frac{p-1}{p^2} A -\frac{2p-2}{p} A^{\frac{2-3p}{p}}B \right) + O(C^{2}),
    \end{equation*}
  Taking $B=\left(2\lambda \frac{p-1}{p^2} A +1 \right)\frac{p}{2p-2}A^{\frac{3p-2}{p}}$ gives $h(u) = -C^{\frac{2p}{2p-2}} + O(C^2)$
  and we conclude as before.
 Note that if $\lambda \leq \frac{p-1}{2}-\delta$ for a fixed $\delta>0$, then $A \in [m'_{p,\delta}, M'_{p,\delta}]$, and so also $B \in [m_{p,\delta}, M_{p,\delta}]\subset \subset (0,\infty)$.

  \underline{Case $\lambda > \frac{p-1}{2}$:} Take the ansatz $u_m = \left( \frac{2\lambda}{p-1}-1 \right) + B C^{2}$.
  We compute, using again Remark~\ref{rmk:first order Taylor}
  \begin{equation*}
  \begin{split}
  h(u) &= \left(\frac{(p-1)^2}{p^2}-2\lambda \frac{p-1}{p^2}+C^2\right) - \frac{(p-1)^2}{p^2}  \left(\left(\frac{2\lambda}{p-1}-1\right) + B C^{2} \right)^{2} \\
  &\quad + 2\lambda \frac{p-1}{p^2} \left(\left(\frac{2\lambda}{p-1}-1\right) + B C^{2} \right) - C^2 \left(\left(\frac{2\lambda}{p-1}-1\right) + B C^{2} \right)^{\frac{2-2p}{p}} \\
  &= \left(\frac{(p-1)^2}{p^2}-2\lambda \frac{p-1}{p^2}+C^2\right) \\
  &\quad- \frac{(p-1)^2}{p^2} \left(\left(\frac{2\lambda}{p-1}-1\right)^{2} + 2 \left(\frac{2\lambda}{p-1}-1\right)B C^{2} + B^2C^{4} \right) \\
  &\quad + 2\lambda \frac{p-1}{p^2} \left(\left(\frac{2\lambda}{p-1}-1\right) + B C^{2} \right) \\
  &\quad - C^2 \left( \left(\frac{2\lambda}{p-1}-1\right)^{\frac{2-2p}{p}} + \frac{2-2p}{p} \left(\frac{2\lambda}{p-1}-1\right)^{\frac{2-3p}{p}} B C^{2} + O(C^{4}) \right).
\end{split}
\end{equation*}
  The coefficient of the $C^0$ term vanishes, as
  \begin{equation*}
  \left( \frac{(p-1)^2}{p^2}-2\lambda \frac{p-1}{p^2} \right) -\frac{(p-1)^2}{p^2}\left(\frac{2\lambda}{p-1} -1\right)^2 + 2\lambda \frac{p-1}{p^2} \left(\frac{2\lambda}{p-1} -1\right) =0.
\end{equation*}
The coefficient of the $C^2$ term, $\text{coeff}(C^2)$, is
\begin{equation*}
  \begin{split}
    1+B \left[-2\frac{(p-1)^2}{p^2} \left(\frac{2\lambda}{p-1}-1 \right)  + 2\lambda\frac{p-1}{p^2} \right] - \left(\frac{2\lambda}{p-1}-1 \right)^{\frac{2-2p}{p}} = 1+B \frac{2(p-1)(p-1-\lambda)}{p^2} - \left(\frac{2\lambda}{p-1}-1 \right)^{\frac{2-2p}{p}}. 
  \end{split}
\end{equation*}
If we take $B=0$, then said coefficient is $1-\left(\frac{2\lambda}{p-1}-1 \right)^{\frac{2-2p}{p}}<0$, that is 
\begin{equation*}
  h(u) = \left( 1-\left(\frac{2\lambda}{p-1}-1 \right)^{\frac{2-2p}{p}} \right) C^2 + O(C^4),
\end{equation*}
which is strictly negative for $C^2<\eps_{\lambda,p}$ sufficiently small, that is $u_m > \left(\frac{2\lambda}{p-1}-1\right)$. 

For the upper bound, taking $B =B_{\lambda,p} = p^2\frac{\left(\frac{2\lambda}{p-1}-1 \right)^{\frac{2-2p}{p}}-\frac{\lambda+1}{p}}{2(p-1)(p-1-\lambda)}$ (which is clearly continuous in $\lambda$ on $(\frac{p-1}{2}, p-1)$) gives
\begin{equation*}
  \text{coeff}(C^2) = 1+ p^2\frac{\left(\frac{2\lambda}{p-1}-1 \right)^{\frac{2-2p}{p}}-\frac{\lambda+1}{p}}{2(p-1)(p-1-\lambda)} \frac{2(p-1)(p-1-\lambda)}{p^2} - \left(\frac{2\lambda}{p-1}-1 \right)^{\frac{2-2p}{p}} = \tfrac{p-1-\lambda}{p}>0,
\end{equation*}
and so $h(u) = \tfrac{p-1-\lambda}{p} C^2 + O(C^4)>0$ for $C^2<\eps_{\lambda,p}$ sufficiently small and so $u_m< \left(\frac{2\lambda}{p-1}-1\right) + BC^2$.

We note that $B \to \infty$ as $\lambda \downarrow \frac{p-1}{2}$, but letting $\eps:=p-1-\lambda \to 0$,
\begin{equation*}
\begin{split}
B &= p^2\frac{ \left(\frac{2\lambda}{p-1}-1 \right)^{\frac{2-2p}{p}}-\frac{\lambda+1}{p}}{2(p-1)(p-1-\lambda)} = p^2\frac{\left(1-\frac{2\eps}{p-1} \right)^{\frac{2-2p}{p}}-\left(1 - \frac{\eps}{p} \right)}{2(p-1)\eps} = p^2\frac{\left(1+\frac{4}{p}\eps  + O(\eps^2)\right)-\left(1 - \frac{\eps}{p} \right)}{2(p-1)\eps},   
\end{split}
\end{equation*}
which converges to $\frac{5p}{2(p-1)}$ as $\eps\to 0$, despite the apparent factor of $\eps$ in the denominator of $B$. Finally, if $\lambda \geq \frac{p-1}{2}+\delta$ for a fixed $\delta>0$, then $B \in [m_{p,\delta}, M_{p,\delta}]\subset \subset (0,\infty)$ as before.
The proof is finished.
\end{proof}

Around $\lambda=\frac{p-1}{2}$, we do not calculate an exact estimate, but only  use a crude upper bound (only needed in Lemma~\ref{lemma: continuity of minimal curvature}), the reason being that at a later point we work with $S^\rho$ anyway.
\begin{lemma}
  \label{lemma: estimate around (p-1)/2}
Let $(\lambda,C^2) \in S^0$ with $\left|\lambda -\frac{p-1}{2}\right| < \frac{p-1}{4} $. Then there exists $\eps_{p}>0$ such that $u_m<\max(0,\frac{2\lambda}{p-1}-1)+C^{\frac{p}{3p-2}}$ for all $C^2<\eps_{p}$.
\end{lemma}
\begin{proof}
  Let $u = \max(0,\frac{2\lambda}{p-1}-1)+C^{\frac{p}{3p-2}} \geq C^{\frac{p}{3p-2}}$, then
  \begin{equation*}
  \begin{split}
      h(u) &= \left(\frac{(p-1)^2}{p^2} -2\lambda\frac{p-1}{p^2}+ C^2\right)-\frac{(p-1)^2}{p^2} u^2 + 2\lambda\frac{p-1}{p^2} u - C^{2}u^{\frac{2-2p}{p}} \\
      &\geq \left(\frac{(p-1)^2}{p^2} -2\lambda\frac{p-1}{p^2}+ C^2\right)-\frac{(p-1)^2}{p^2} u^2 + 2\lambda\frac{p-1}{p^2} u - C^{2}C^{\frac{2-2p}{3p-2}}  \\
      &\geq C^2 - \frac{(p-1)^2}{p^2}\left(2\max\left(0,\frac{2\lambda}{p-1}-1\right)C^{\frac{p}{3p-2}} + C^{\frac{2p}{3p-2}} \right) + 2\lambda\frac{p-1}{p^2} C^{\frac{p}{3p-2}} - C^{\frac{4p-2}{3p-2}}  \\
  \end{split}
  \end{equation*}
  using $\frac{(p-1)^2}{p^2}-2\lambda \frac{p-1}{p^2}\geq 0$ if $\lambda\leq \frac{p-1}{2}$ and otherwise
  \begin{equation*}
  \left( \frac{(p-1)^2}{p^2}-2\lambda \frac{p-1}{p^2} \right) -\frac{(p-1)^2}{p^2}\left(\frac{p-1-2\lambda}{p-1}\right)^2 + 2\lambda \frac{p-1}{p^2} \left(\frac{2\lambda}{p-1} -1\right) =0.
\end{equation*}
The lowest order term is $C^{\frac{p}{3p-2}}$, with coefficient
\begin{equation*}
  2\frac{p-1}{p^2} \left(\lambda - \max(0,2\lambda - (p-1)) \right) \geq \frac{1}{2} \frac{(p-1)^2}{p^2} >0,
\end{equation*}
thus $h(u)>0$ and so $u_m<\max(0,\frac{2\lambda}{p-1}-1)+C^{\frac{p}{3p-2}}$.
\end{proof}

We directly obtain the following corollary for $k=u^{\frac{1}{p}}$.
\begin{corollary}
  \label{cor: estimates for km}
  Let $(\lambda,C^2) \in S^0$.
  \begin{enumerate}
    \item[(i)] If $0\leq \lambda < \frac{p-1}{2}$ and $\gamma_p = \min\left(\frac{4p}{2p-2}-1, \frac{6p-4}{2p-2}-1\right)$, then there exist $B_{\lambda,p}, \eps_{\lambda,p} \in (0,\infty)$ such that for all $C^2<\eps_{\lambda,p}$
    \begin{equation*}
      \left(\frac{(p-1)(p-1-2\lambda)}{p^2} \right)^{\frac{1}{2-2p}} C^{\frac{2}{2p-2}} -B_{\lambda,p} C^{\gamma_p} < k_m < \left(\frac{(p-1)(p-1-2\lambda)}{p^2} \right)^{\frac{1}{2-2p}} C^{\frac{2}{2p-2}}.
    \end{equation*} 
    Moreover, for $0<\delta\ll 1$ fixed, if $\lambda \in [0,\frac{p-1}{2}-\delta]$, then $B_{\lambda,p} \in [m_{\delta,p}, M_{\delta,p}]$ for some constants $0<m_{\delta,p}, M_{\delta,p}$.

    \item[(ii)] If $\frac{p-1}{2}<\lambda < p-1$, then there exists $B_{\lambda,p}, \eps_{\lambda,p} \in (0,\infty)$ such that for all $C^2<\eps_{\lambda,p}$
    \begin{equation*}
      \left(\frac{2\lambda}{p-1}-1 \right)^{\frac{1}{p}}  < k_m < \left( \frac{2\lambda}{p-1}-1 \right)^{\frac{1}{p}} + B_{\lambda,p} C^{2}.
    \end{equation*}
    Moreover $\lambda \mapsto B_{\lambda,p}$ is continuous and $\lim_{\lambda\to p-1}B_{\lambda,p}=  \frac{5}{p-1}$. Also for $0<\delta\ll 1$ fixed, if $\lambda \in [\frac{p-1}{2}+\delta,p-1)$ then $B_{\lambda,p} \in [m_{\delta,p}, M_{\delta,p}]$ for some constants $0<m_{\delta,p}, M_{\delta,p}$.
  \end{enumerate}
\end{corollary}

\begin{proof}
  We take the estimates from Proposition~\ref{prop: estimates for um}.

  \underline{Case $0\leq \lambda < \frac{p-1}{2}$:} For the lower bound, by Taylor (Remark~\ref{rmk:first order Taylor}),
  \begin{equation*}
    \begin{split}
      k_m &>\left( \left(\frac{(p-1)(p-1-2\lambda)}{p^2} \right)^{\frac{p}{2-2p}} C^{\frac{2p}{2p-2}} -B_{\lambda,p} C^{\gamma_p} \right)^{\frac{1}{p}} \\
       &= \left(\frac{(p-1)(p-1-2\lambda)}{p^2} \right)^{\frac{1}{2-2p}} C^{\frac{2}{2p-2}} -\frac{B_{\lambda,p}}{p} \left(\frac{(p-1)(p-1-2\lambda)}{p^2} \right)^{\frac{1}{2}} C^{\gamma_p-1} + O(C^{2\gamma_p-2-\frac{1}{p-1}})
    \end{split}
  \end{equation*}
  and the result follows by redefining the constants and absorbing the remainder into the $C^{\gamma_p-1}$ term. The upper bound is trivial.

  \underline{Case $\frac{p-1}{2} < \lambda < p-1$:} Similarly, we only verify the upper bound, i.e.
  \begin{equation*}
    \begin{split}
      k_m &< \left( \frac{2\lambda}{p-1}-1 + B_{\lambda,p} C^{2}\right)^{\frac{1}{p}} = \left( \frac{2\lambda}{p-1}-1 \right)^{\frac{1}{p}} + \frac{B_{\lambda,p}}{p} \left( \frac{2\lambda}{p-1}-1 \right)^{\frac{1-p}{p}} C^2 + O(C^4).
    \end{split}
  \end{equation*}
 Upon redefining the constants and absorbing the remainder, the proof is finished.
\end{proof}

We also derive a global lower bound for $u_m$.
\begin{lemma}
  \label{lemma: global lower bound for um}
Let $(\lambda,C^2) \in S^0$. 
Then there exists $\eps_{p}>0$ such that $u_m > \begin{cases}
    C^{\frac{2p}{2p-2}}, &\text{if } C^2<\eps_p, \\ \left( \frac{p^2 \eps_p}{(p-1)(2p-1)} \right)^{\frac{p}{2p-2}}, &\text{if } C^2\geq \eps_p.
\end{cases}$
\end{lemma} 

\begin{proof}
  Let $u = C^{\frac{2p}{2p-2}}$, we have for $C^2<\eps_p$ sufficiently small
  \begin{equation*}
    \begin{split}
      h(u) &= \left(\frac{(p-1)^2}{p^2} -2\lambda\frac{p-1}{p^2}+ C^2\right)-\frac{(p-1)^2}{p^2} C^{\frac{4p}{2p-2}} + 2\lambda\frac{p-1}{p^2} C^{\frac{2p}{2p-2}} - C^{2+ \frac{2p}{2p-2}\frac{2-2p}{p}} \\
      &= \left(\frac{(p-1)^2}{p^2} -2\lambda\frac{p-1}{p^2}-1\right) +O(C^{\min(\frac{2p}{2p-2},2)}) \leq -\frac{1}{2}\frac{2p-1}{p^2}<0,
    \end{split}
  \end{equation*}
  hence $u_m>C^{\frac{2p}{2p-2}}$ holds. 
  On the other hand, for $C^2\geq \eps_p$, letting $u=\left( \frac{p^2 \eps_p}{(p-1)(2p-1)} \right)^{\frac{p}{2p-2}}$,
  \begin{equation*}
  \begin{split}
      h(u) &= \left(\frac{(p-1)^2}{p^2} -2\lambda\frac{p-1}{p^2}+ C^2\right)-\frac{(p-1)^2}{p^2} u^2 + 2\lambda\frac{p-1}{p^2} u - C^{2}u^{\frac{2-2p}{p}} \\
      &< \frac{(p-1)^2}{p^2} + \frac{p-1}{p} - \eps_p \left( \frac{p^2 \eps_p}{(p-1)(2p-1)} \right)^{\frac{p}{2p-2} \frac{2-2p}{p}}  = 0,
  \end{split}
  \end{equation*}
  hence $u_m > \left( \frac{p^2 \eps_p}{(p-1)(2p-1)} \right)^{\frac{p}{2p-2}}$ holds.
\end{proof}

\begin{remark}
  The asymptotic estimates also match the planar theory from \cite{miuraclassification} for $C^2=0$. 
  In that case, the curvature $k$ switches from being wavelike over borderline/flatcore to orbitlike at $\lambda = \frac{p-1}{2}$, \cite[Theorem~4.1]{miuraclassification} (note that here $A_{p,\lambda}=1$). 
  Hence for $\lambda \leq \frac{p-1}{2}$, $k_m=0$ and for $\lambda > \frac{p-1}{2}$, we can easily find $u_m=k_m^p$ as the second root of the quadratic polynomial $z \mapsto P(z)= \left(\frac{(p-1)^2}{p^2}-2\lambda \frac{p-1}{p^2}\right) - \left( \frac{(p-1)^2}{p^2}z^2 - 2\lambda \frac{p-1}{p^2}z \right)$, namely $u_m=k_m^p = \frac{2\lambda}{p-1}-1$.
\end{remark}

For points close to $\partial S_2$ (including points around $(p-1,0)$), we derive a new estimate. 
 Note that now we use $w=k^{p-1}$ (since $w$ and not $k$ is needed at a later point in Proposition~\ref{prop: continuity of period}), but such an estimate can easily be adapted to $k$ or $u$, similar to Corollary~\ref{cor: estimates for km}.

\begin{proposition}
  \label{prop: wm around S2:2}
  For every $\gamma\in (1,2)$, there exists $\bar \eps_\gamma >0$ such that for all $(\lambda,C^2) \in S^0$ with $\tfrac{p-1-\lambda}{p}-  C^2 \leq \bar \eps_\gamma $ we have the estimate
  \begin{equation*}
   1 - A\left( \frac{p-1-\lambda}{p}-C^2 \right) - \left( \frac{p-1-\lambda}{p}-C^2 \right)^\gamma < w_m <    1 - A\left( \frac{p-1-\lambda}{p}-C^2 \right) + \left( \frac{p-1-\lambda}{p}-C^2 \right)^\gamma,
  \end{equation*}
  with $A= \frac{-4}{g''(1)} = 2\left[\frac{p+1}{p} - \frac{\lambda}{p(p-1)} + 3 C^2 \right]^{-1} \in [\frac{p}{2p-1},2]$.
\end{proposition}

\begin{proof}
  Let $\eps:=\tfrac{p-1-\lambda}{p}-C^2 \ll 1$ and take the ansatz $w = 1-A\eps +B \eps^\gamma$ with $\gamma \in (1,2)$ fixed in $g(w)$.
  Now using Remark~\ref{rmk: Taylor for (1+x)to the r} in the three nonlinear terms gives
\begin{equation*}
\begin{split}
  g(w) &= \frac{(p-1)^2}{p^2} - 2\lambda \frac{p-1}{p^2} +C^2 \\
  &\quad-\frac{(p-1)^2}{p^2}  \left( 1- A \eps +B \eps^\gamma \right)^{\frac{2p} {p-1}} + 2\lambda \frac{p-1}{p^2} \left( 1-A\eps +B \eps^\gamma \right)^{\frac{p}{p-1}} - C^2 \left( 1-A\eps +B \eps^\gamma \right)^{-2} \\
  &= \frac{(p-1)^2}{p^2} - 2\lambda \frac{p-1}{p^2} +C^2 + O(\eps^{\min(3,2\gamma)})\\
    &\quad - \frac{(p-1)^2}{p^2}  \left( 1 -\eps\left[\frac{2p}{p-1}A\right] + \eps^\gamma \left[\frac{2p}{p-1}B\right] + \eps^2 \left[\frac{\frac{2p}{p-1}(\frac{2p}{p-1}-1)}{2}A^2 \right] - \eps^{1+\gamma} \left[\frac{2p}{p-1}(\frac{2p}{p-1}-1) AB \right]  \right) \\
    &\quad +2\lambda \frac{p-1}{p^2}  \left( 1 -\eps\left[\frac{p}{p-1}A\right] + \eps^\gamma \left[\frac{p}{p-1}B\right] + \eps^2 \left[\frac{\frac{p}{p-1}(\frac{p}{p-1}-1)}{2}A^2 \right] - \eps^{1+\gamma} \left[\frac{p}{p-1}(\frac{p}{p-1}-1) AB \right] \right) \\
    &\quad - C^2 \left( 1 - \eps[-2A] + \eps^\gamma [-2B] + \eps^2 \left[\frac{-2(-3)}{2}A^2 \right] - \eps^{1+\gamma} [-2(-3) AB] \right) \\
    &= \frac{(p-1)^2}{p^2} - 2\lambda \frac{p-1}{p^2} +C^2 + O(\eps^{\min(2\gamma,3)})\\
   &\quad - \frac{(p-1)^2}{p^2}  \left( 1 - \eps\left[\frac{2p}{p-1}A\right] + \eps^\gamma \left[\frac{2p}{p-1}B\right] + \eps^2 \left[\frac{p(p+1)}{(p-1)^2}A^2 \right] - \eps^{1+\gamma} \left[\frac{2p(p+1)}{(p-1)^2} AB\right]  \right) \\
   &\quad +2\lambda \frac{p-1}{p^2}  \left( 1 - \eps\left[\frac{p}{p-1}A\right] + \eps^\gamma \left[\frac{p}{p-1}B \right] + \eps^2 \left[\frac{p}{2(p-1)^2}A^2 \right] - \eps^{1+\gamma} \left[\frac{p}{(p-1)^2} AB\right] \right) \\
    &\quad - C^2 \left( 1 + \eps[2A] + \eps^\gamma [-2B] + \eps^2 \left[3A^2 \right] - \eps^{1+\gamma} [6 AB] \right). \\
\end{split}
\end{equation*}
Collecting terms in powers of $\eps$, we obtain
\begin{equation*}
\begin{split}  
 g(w) 
  &=  - \eps A \left( -2\frac{p-1}{p} + \frac{2\lambda }{p} +2C^2 \right) \\
    &\quad + \eps^\gamma B \left( -2\frac{p-1}{p} + \frac{2\lambda}{p} + 2C^2 \right) \\
    & \quad + \eps^2 A^2 \left( -\frac{p+1}{p} + \frac{\lambda}{p(p-1)} - 3 C^2 \right) \\
    & \quad - \eps^{1+\gamma} AB 2 \left( -\frac{p+1}{p} + \frac{\lambda}{p(p-1)} - 3 C^2 \right) + O(\eps^{\min(3,2\gamma)})\\
    &=   \eps^2A \left(-2- A \left[-\frac{p+1}{p} + \frac{\lambda}{p(p-1)} - 3 C^2 \right] \right) \\
  &\quad + \eps^{1+\gamma}B \left(-2 - 2A\left[-\frac{p+1}{p} + \frac{\lambda}{p(p-1)} - 3 C^2 \right] \right) + O(\eps^{\min(3,2\gamma)}), \\
  \end{split}
\end{equation*}
using $\frac{p-1}{p}-\frac{\lambda}{p}-C^2 = \eps$. Taking now $A=2\left[\frac{p+1}{p} - \frac{\lambda}{p(p-1)} + 3 C^2 \right]^{-1}$ and thus $-\frac{p+1}{p} + \frac{\lambda}{p(p-1)} - 3 C^2 = \frac{-2}{A}$, we obtain
\begin{equation*}
\begin{split}  
  g(w)
  &= O(\eps^{\min(3,2\gamma)}) + 2B \eps^{1+\gamma}.  \\
\end{split}
\end{equation*}
 Taking now $B= 1$ gives $g(1-A\eps+\eps^\gamma) >0 $ and taking $B= -1$ gives $g(1-A\eps-\eps^\gamma) <0 $, as long as $\eps<\bar \eps_\gamma$ sufficiently small (since the constant in $O(\eps^{\min(3,2\gamma)})$ depends only on $p$ and $\lambda$, see Remark~\ref{rmk: Taylor for (1+x)to the r}). The fact that $A=\frac{-4}{g''(1)}$ follows from Lemma~\ref{lemma: estimates for gprime and gbis at 1}. 
 The proof is finished.
\end{proof}

\begin{corollary}
    \label{cor: approx of A from estimate at p-1}
    For every $\gamma \in (1,2)$, there exists $\bar \eps_\gamma>0$ such that for $(\lambda,C^2) \in S^0$ with $|p-1-\lambda|<\bar \eps_\gamma$ 
  \begin{equation*}
    w_m = 1-2 \left(\frac{p-1-\lambda}{p} -C^2 \right) 
    + O(|p-1-\lambda|^\gamma).
  \end{equation*}
\end{corollary}
\begin{proof}
    This follows from Proposition~\ref{prop: wm around S2:2} since the term $A$ is approximated as 
  \begin{equation*}
    \begin{split}
      A=\frac{-4}{g''(1)}&=2p(p-1)[p(p-1)+(p-1-\lambda)+3p(p-1)C^2]^{-1} = 2+O(|p-1-\lambda|),
    \end{split}
  \end{equation*}
  using also $C^2 \leq \frac{p-1-\lambda}{p}$.
\end{proof}

Combining the asymptotic estimates on the boundary we obtain the following lemma.

\begin{lemma}
  \label{lemma: continuity of minimal curvature}
  The minimal curvature $k_m$ (and $w_m$, $u_m$) are continuous functions of $(\lambda, C^2)$ inside $S^0$, extend continuously to $\partial S^0$ and are strictly positive if and only if $C^2>0$ or $\lambda>\frac{p-1}{2}$.
\end{lemma}
\begin{proof}
  Continuity and positivity strictly inside $S^0$ follow from Lemma~\ref{lemma: number of solutions} (to be precise, $f(k)$ has nonvanishing derivative at the root $k_m$ for any $(\lambda,C^2) \in S^0$, so locally the implicit function theorem applies) and the continuity of the parameters of $f$. 
  If $C^2 \to \tfrac{p-1-\lambda}{p}$, then $k_m\to1$ by Proposition~\ref{prop: wm around S2:2}, that is $k_m$ extends to $\partial S_2$ and $(p-1,0)$.
  If $C^2\to0$, the asymptotic estimates from Corollary~\ref{cor: estimates for km} and Lemma~\ref{lemma: estimate around (p-1)/2} give that $k_m\to 0$ if $\lambda < \frac{p-1}{2}$ and $k_m\to \left(\frac{2\lambda}{p-1}-1 \right)^{\frac{1}{p}}$ if $\lambda>\frac{p-1}{2}$. 
\end{proof}

Moreover, we obtain a uniform Lipschitz estimate in  $S^0$.
\begin{lemma}
  \label{lemma: uniform Lip estimate}
  There exists $L_p \in \R$ such that any solution $w$ to \eqref{eq:EL pointwise for w} with $(\lambda,C^2) \in S^0$ satisfies $|w(s_1)-w(s_2)|\leq L_p|s_1-s_2|$ for all $s_1,s_2$ in $\R$.
\end{lemma}
\begin{proof}
  It suffices to show that the derivative $w'$ is uniformly bounded, regardless of parameters $\lambda$ and $C^2$. 
  From \eqref{eq:EL for w first order}, it remains to bound the term $C^2w_m^{-2}$. 
  By Lemma~\ref{lemma: global lower bound for um}, we have $w_m=k_m^{p-1} = u_m^{\frac{p-1}{p}}> C$ and so $C^2w_m^{-2} < 1$ as long as $C^2 < \eps_p$.
  If $C^2 \geq \eps_p$, the estimate $C^2 w_m^{-2} \leq \frac{p-1}{p} \left( \frac{p^2 \eps_p}{(p-1)(2p-1)} \right)^{\frac{-2p}{2p-2}}$ holds.
  \end{proof}

\subsection{Continuity in the parameter space}
We now turn towards the continuity of the solution and its period with respect to variations in the parameters $\lambda$ and $C^2$. Note that here we restrict from $S^0$ to $S^\rho$, since the period diverges at $(\lambda,C^2)=(\frac{p-1}{2},0)$, the borderline elastica, whenever $p\leq2$. Instead of treating the cases $p\leq2$ and $p>2$ separately, for simplicity we restrict to $S^\rho$ in both cases. 
We start with the following lemma.

\begin{lemma}
  \label{lemma: convergence of s-periods}
  Let $(\lambda_n,C^2_n) \to (\lambda, C^2) \in S^0 \cup \partial S_{1}^0$ such that $(\lambda,C^2) \neq (\frac{p-1}{2},0)$ and $\lambda<p-1$. 
  Moreover, let $g_n:=g_{\lambda_n,C^2_n}$ and $g:=g_{\lambda,C^2}$. Then 
  \begin{equation*}
    S_n: [0,1] \to \R: w\mapsto \begin{cases}
      \int_{w_{m,n}}^w \frac{1}{\sqrt{g_n(x)}}dx &\text{if }w_{m,n}<w\leq1, \\ 0  &\text{if }0\leq w\leq w_{m,n},
    \end{cases}
  \end{equation*}
  converges pointwise to 
  \begin{equation*}
    S:[0,1] \to \R: w \mapsto \begin{cases}
      \int_{w_{m}}^w \frac{1}{\sqrt{g(x)}}dx &\text{if }w_{m}<w\leq 1, \\ 0  &\text{if }0\leq w\leq w_{m},
    \end{cases}
  \end{equation*}
  for any $w \in [0,1]$.
\end{lemma}
\begin{proof}
  First, note that the minimum $w_m$ of $w$ (and so also $w_{m,n}$ for large $n$) are strictly less than $1$, i.e.\ $S_n$ and $S$ are not identically zero.
  Recall that
  \begin{equation*}
    g(x)=g_{\lambda,C^2}(x) = \left(\frac{(p-1)^2}{p^2}-2\lambda \frac{p-1}{p^2} +C^2 \right) -\frac{(p-1)^2}{p^2}  x^{\frac{2p} {p-1}} + 2\lambda \frac{p-1}{p^2} x^{\frac{p}{p-1}} - C^2 x^{-2}
  \end{equation*}
  with
  \begin{equation}
    \label{eq: derivatives for g period}
    \begin{split}
    g_{\lambda,C^2}'(x) &= - 2\frac{p-1}{p} x^{\frac{p+1}{p-1}} + \frac{2 \lambda}{p} x^{\frac{1}{p-1}} + 2C^2 x^{-3}, \\
    g_{\lambda,C^2}''(x) &= - 2\frac{p+1}{p} x^{\frac{2}{p-1}} + \frac{2 \lambda}{p(p-1)} x^{\frac{2-p}{p-1}} - 6C^2 x^{-4} .
    \end{split}
  \end{equation}
  By the change of variable $t=\frac{x-w_{m,n}}{w-w_{m,n}}$, $dt = \frac{dx}{w-w_{m,n}}$ we have 
  \begin{equation*}
    S_n(w) = \int_{0}^1 \frac{w-w_{m,n}}{\sqrt{g_{n}((w-w_{m,n})t+w_{m,n})}}dt.
  \end{equation*}

  Let $(\lambda_n,C_n^2)\to (\lambda,C^2)$. By the continuity of $g$ in its parameters and its argument and Lemma~\ref{lemma: continuity of minimal curvature} we have for every $t\in (0,1)$,
  \begin{equation*}
    \frac{w-w_{m,n}}{\sqrt{g_{n}((w-w_{m,n})t+w_{m,n})}} \to \frac{w-w_{m}}{\sqrt{g((w-w_{m})t+w_{m})}}.
  \end{equation*}
  By Vitali's convergence theorem it suffices now to show that the sequence of integrands is uniformly integrable in $L^1$. 
  This follows by de La Vall\'ee-Poussin's theorem \cite[Theorem~2.29]{leoniCov}, once we show that they are uniformly bounded in $L^q(0,1)$ for some $q \in (1,2)$. 
  From
  \begin{equation*}
    \int_{0}^1 \left( \frac{w-w_{m,n}}{\sqrt{g_{n}((w-w_{m,n})t+w_{m,n})}} \right)^q dt = \int_{w_{m,n}}^w \frac{(w-w_{m,n})^{q-1}}{g_n(x)^{q/2}}dx \leq \int_{w_{m,n}}^1 \frac{1}{g_n(x)^{q/2}}dx =:Q_n,
  \end{equation*} 
  it remains to show that $Q_n \leq M =M(\lambda,C^2,q)$.

  \underline{Case $C^2>0$ or $\lambda>\frac{p-1}{2}$:} We have $0<w_m<1$. By Lemma~\ref{lemma: continuity of minimal curvature} and \eqref{eq: derivatives for g period},  $g_n''$ is uniformly bounded in an open neighborhood $U$ of $[w_m,1]$. 
  Thus by Arzela--Ascoli, $g_n$ and $g_n'$ converge locally uniformly to $g$ in $U$.
  Moreover, $g'(w_m),-g'(1)>0$ by Lemma~\ref{lemma: estimates for gprime and gbis at 1} and Lemma~\ref{lemma: estimate for gprime at wm}. 
  Also, we note that there exists only one nontrivial solution to $g'(w)=0$ (if $C^2=0$, we calculate it to be $w=\left(\frac{\lambda}{p-1}\right)^\frac{1}{p}$, if $C^2>0$ use Lemma~\ref{lemma: unique root}) with $0<w_m<w<1$. 
  Thus there exists $\delta=\delta(\lambda,C^2)>0$ and $\eps=\eps(\lambda,C^2)>0$ such that 
  \begin{equation*}
    \inf_{(w_m,w_m+\delta)} g'(w) , \inf_{(1-\delta,1)} -g'(w), \inf_{(w_m+\delta, 1-\delta)} g(w) >\eps.
  \end{equation*}
  Thanks to the uniform convergence, we also have for large $n$
  \begin{equation*}
    \inf_{(w_{m,n},w_{m,n}+\delta)} g_n'(w) , \inf_{(1-\delta,1)} -g'_n(w), \inf_{(w_{m,n}+\delta, 1-\delta)} g_n(w) > \frac{\eps}{2}.
  \end{equation*}
  Hence, by a zero order Taylor expansion with Lagrange remainder for the endpoint terms,
  \begin{equation*}
    Q_n \leq  \int_{w_{m,n}}^{w_{m,n}+\delta} \frac{1}{g_n'(\xi_x)^{q/2} (x-w_{m,n})^{q/2}} \, dx + \int_{w_{m,n}+\delta}^{1-\delta} \frac{1}{\eps^{q/2}} \, dx + \int_{1-\delta}^{1} \frac{1}{(-g_n'(\eta_x))^{q/2} (1-x)^{q/2}} dx \leq 3\frac{(1-\frac{q}{2})^{-1}}{\eps^{q/2}},
  \end{equation*}
  where $\xi_x \in [w_{m,n},w_{m,n}+\delta]$ and $\eta_x \in [1-\delta,1]$.
  
  \underline{Case $C^2=0$ and $\lambda<\frac{p-1}{2}$:} Now $w_{m,n}\to w_m= 0$, i.e.\ the uniform convergence (at least around $w_m$) is lost and we have to split up the integral depending on $n$. 
  From Proposition~\ref{prop: estimates for um}, we have $w_{m,n} = \frac{p}{\sqrt{(p-1)(p-1-2\lambda_n)}} C_n + O(C_n^{\gamma_p})$, with $\gamma_p>1$ and the constant in $O(C_n^{\gamma_p})$ depending only on $p$ and $\lambda$. 
  Taking $a_n = 2\frac{p}{\sqrt{(p-1)(p-1-2\lambda_n)}} C_n>w_{m,n}$ gives
  \begin{equation*}
    g_n(a_n) = \frac{(p-1)^2}{p^2} - 2\lambda_n \frac{p-1}{p^2} +  o(C_n) - \frac{1}{4} \left(\frac{(p-1)^2}{p^2} - 2\lambda_n \frac{p-1}{p^2}\right) > \eps_{p,\lambda},
  \end{equation*}
  for some $\eps_{p,\lambda}>0$.
  As in the previous case (e.g.\ use the uniform convergence argument in an open neighborhood around $1$), there exists $\delta>0$ and $\eps<\eps_{p,\lambda}$ such that $g_n(1-\delta), -g_n'(1)>\eps$. 
  Lastly, we note that for any $w \in [w_{m,n}, a_n]$, the inequality $g_n'(w) \geq g_n'(a_n)$ holds, ($g_n''(w)\ll 0$, as the $C_n^2C_n^{-4}=C_n^{-2}$ term dominates). 
  Thereby, using that $g_n$ is unimodal (Lemma~\ref{lemma: number of solutions}) for the middle term,
  \begin{equation*}
    Q_n \leq  \int_{w_{m,n}}^{a_n} \frac{1}{g_n'(\xi_x)^{q/2} (x-w_{m,n})^{q/2}} \, dx + \int_{a_n}^{1-\delta} \frac{1}{\eps^{q/2}} \, dx + \int_{1-\delta}^{1} \frac{1}{(-g_n'(\eta_x))^{q/2} (1-x)^{q/2}} \,dx \leq \frac{(1-\frac{q}{2})^{-1}}{g_n'(a_n)^{q/2}} + \frac{2(1-\frac{q}{2})^{-1}}{\eps^{q/2}} .
  \end{equation*}
  The uniform boundedness of $Q_n$ follows by noting that $g_n'(a_n) \approx g_n'(C_n) \approx C_n^{-1} \to \infty$.
\end{proof}

Recall that the period $P:S \to [0,\infty]$ is defined as 
\begin{equation*}
  P=P({\lambda,C^2}) = \int_{w_m({\lambda,C^2})}^1 \frac{dw}{\sqrt{\frac{(p-1)^2}{p^2}-2\lambda\frac{p-1}{p^2}+C^2-\frac{(p-1)^2}{p^2}w^{\frac{2p}{p-1}} + 2\lambda \frac{p-1}{p^2} w^{\frac{p}{p-1}} -C^2w^{-2}}}.
\end{equation*}

We have the following continuity up to the boundary with respect to the parameters.
\begin{proposition}
    \label{prop: continuity of period}
  Let $(\lambda_n,C^2_n) \to (\lambda, C^2) \in \overline {S^0}$ such that $(\lambda,C^2) \neq (\frac{p-1}{2},0)$.
  \begin{itemize}
    \item If $(\lambda,C^2) \in S^0 $, then  $P(\lambda_n,C^2_n) \to P(\lambda,C^2) \in (0,\infty)$;
    \item If $(\lambda,C^2) \in \partial S_1^0$ with $\lambda\neq \frac{p-1}{2}$, then  $P(\lambda_n,C^2_n) \to P(\lambda,0)$;
    \item If $(\lambda,C^2) \in \partial S^0_2 $ with $\lambda<p-1$, then
      $P_{\lambda,C^2}:=\lim_n P(\lambda_n,C^2_n)$ exists as limit from inside $S^0$ and $P_{\lambda,C^2}<\pi$;
    \item If $(\lambda,C^2) =(p-1,0)$, then  $P(\lambda_n,C^2_n) \to \pi$ (as limit from inside $S^0$).
    \end{itemize}
\end{proposition}
\begin{proof}
  For the first two points take $w=1$ in Lemma~\ref{lemma: convergence of s-periods}. For the last two points, $w_{m,n}\to w_m=1$, so the strategy of Lemma~\ref{lemma: convergence of s-periods} does not apply.
  Denoting $\bar w_n = \frac{w_{m,n}+1}{2}$, we approximate $g_{n}$ by the quadratic polynomial
 \begin{equation*}
  G_n(w) = B_n (w-w_{m,n})(1-w) := 4 \frac{g_n(\frac{w_{m,n}+1}{2})}{(1-w_{m,n})^2} (w-w_{m,n})(1-w),
 \end{equation*}
 which has the exact same roots as $g_{n}$ and $G_n(\bar w_n) = g_{n}(\bar w_n)$.
 Rewrite 
 \begin{equation*}
  P_n = P(\lambda_n,C_n^2) = \int_{w_{m,n}}^1 \frac{1}{\sqrt{g_{n}(w)}}dw = \int_{w_{m,n}}^1 \frac{1}{\sqrt{G_n(w)}}dw + \int_{w_{m,n}}^1 \frac{1}{\sqrt{g_{n}(w)}} - \frac{1}{\sqrt{G_n(w)}} dw =: I_1 + I_2.
 \end{equation*}
 
 First, we compute $I_1$ as
 \begin{equation*}
  I_1 = \int_{w_{m,n}}^1 \frac{1}{\sqrt{G_n(w)}}dw = \int_{w_{m,n}}^1 \frac{1}{\sqrt{B_n(w-w_{m,n})(1-w)}}dw = \frac{\pi}{\sqrt{B_n}}.
 \end{equation*}
 Moreover, there exists $\xi_n \in [\frac{w_{m,n}+1}{2},1]$ such that
 \begin{equation*}
  B_n = \frac{g_n(\frac{w_{m,n}+1}{2})}{\left(\frac{1-w_{m,n}}{2}\right)^2} 
  = \frac{g_n(1)+g_n'(1)\frac{w_{m,n}-1}{2} + \frac{1}{2}g_n''(\xi_n) \left(\frac{w_{m,n}-1}{2} \right)^2}{\left(\frac{1-w_{m,n}}{2}\right)^2} 
  = \frac{-g'_n(1)}{\frac{1-w_{m,n}}{2}} + \frac{1}{2}g_n''(\xi_n).
 \end{equation*}
 From Proposition~\ref{prop: wm around S2:2} and Lemma~\ref{lemma: estimates for gprime and gbis at 1}, there exists $\gamma \in (1,2)$ and $\bar \eps_\gamma>0$ such that for $\left|\frac{p-1-\lambda_n}{p}-C_n^2\right|<\bar \eps_\gamma$, 
 \begin{equation*}
  \begin{split}
    B_n &= \frac{4\left(\frac{p-1-\lambda_n}{p}-C_n^2\right)}{-\frac{4}{g''(1)} \left(\frac{p-1-\lambda_n}{p}-C_n^2\right) + O\left(\left|\frac{p-1-\lambda_n}{p}-C_n^2\right|^\gamma\right) } + \frac{1}{2}g_n''(\xi_n) = -g_n''(1) + \frac{1}{2}g_n''(\xi_n) + O\left(\left|\tfrac{p-1-\lambda_n}{p}-C_n^2\right|^{\gamma-1}\right).
  \end{split}
 \end{equation*}
 From continuity of $g''_n$ around $w=1$ and Lemma~\ref{lemma: estimates for gprime and gbis at 1}, we obtain 
 \begin{equation*}
  g_n''(\xi_n) \to g''(1) = -2\frac{p+1}{p}+\frac{2\lambda}{p(p-1)}-6C^2\quad  \text{ as $w_{m,n} \to 1$ \quad  (i.e.\ $n\to \infty$)}. 
 \end{equation*}
 Hence passing to the limit as $n\to \infty$, we have
  $B_n \to B:=-\frac{1}{2}g''(1)|_{C^2=\tfrac{p-1-\lambda}{p}} = 1 + \frac{3p-2}{p(p-1)}(p-1-\lambda)$,
  which at the point $(\lambda,C^2)=(p-1,0)$ gives $B=1$, that is $I_1\to \pi$, otherwise the limit is strictly less than $\pi$.

 Secondly, we show that the error term $I_2$ vanishes. Letting $Q_n(w)=\frac{g_n(w)}{(w-w_{m,n})(1-w)}$,
 \begin{equation*}
  |I_2|  \leq  \int_{w_{m,n}}^1 \frac{1}{\sqrt{(w-w_{m,n})(1-w)}} \left| \frac{1}{\sqrt{Q_n(w)}} - \frac{1}{\sqrt{B_n}}\right|dw \leq \pi \sup_{w \in (w_{m,n},1)} \left| \frac{1}{\sqrt{Q_n(w)}} - \frac{1}{\sqrt{B_n}}\right|,
 \end{equation*}
 it suffices to show that the supremum of $T_n(w):= \left| Q_n(w) - B_n \right|$ on $(w_{m,n},1)$ vanishes in the limit.
 By a first order Lagrange interpolation for $g_n(w)$, there exists $\zeta_{w,n} \in (w_{m,n},1)$ such that
 \begin{equation}
     Q_n(w) = \frac{g_n(w)}{(w-w_{m,n})(1-w)} = \frac{g_n''(\zeta_{n,w})(w-w_{m,n})(w-1)}{2(w-w_{m,n})(1-w)} = -\frac{g_n''(\zeta_{n,w})}{2}
 \end{equation}
 and thus
 \begin{equation}
 \begin{split}
     \sup_{w \in (w_{m,n},1)}|Q_n(w)-B_n| &= \sup_{w \in (w_{m,n},1)} \left| -\frac{g_n''(\zeta_{n,w})}{2}-\left(-g_n''(1) + \frac{1}{2}g_n''(\xi_n) + O\left(\left|\tfrac{p-1-\lambda_n}{p}-C_n^2\right|^{\gamma-1}\right)  \right)\right| \\
     &\leq \sup_{w \in (w_{m,n},1)} \frac{1}{2}\left|g_n''(1)- g_n''(\zeta_{n,w})\right| + \left|g_n''(1)- g_n''(\xi_{n}) \right| + O\left(\left|\tfrac{p-1-\lambda_n}{p}-C_n^2\right|^{\gamma-1}\right) , 
 \end{split}
 \end{equation}
 which vanishes as $n\to \infty$, as $w_{m,n} \to 1$. This finishes the proof.
\end{proof}

From the proof we directly obtain the following corollary, its proof is safely omitted.
\begin{corollary}
  \label{cor: finiteness of period}
The period $P(\lambda, C^2)$ is finite for every $(\lambda, C^2) \in \overline{S^\rho}$.
\end{corollary}

\begin{remark}
  \label{rmk: period of flatcore}
  We note that for the case $p>2$, the period stays finite even at $(\frac{p-1}{2},0)$, as borderline $p$-elasticae turn into $p$-flatcore elasticae. However, to avoid distinguishing two different cases, we work always in $S^\rho$. 
  For $(\lambda,C^2) \in \partial S^0_2$, we have $k\equiv 1$ so $P(\lambda, C^2)$ is to be understood as the limit from inside.
\end{remark}

Next we show that not only the period, but also the solutions themselves are continuous with respect to the parameters $\lambda$ and $C^2$.
\begin{proposition}
    \label{prop: pointwise convergence of w_n}
    Let $(\lambda_n,C^2_n) \to (\lambda, C^2) \in \overline{S^0}$ such that $(\lambda,C^2) \neq (\frac{p-1}{2},0)$.
  Then, for the solutions $w_n=w_{\lambda_n,C^2_n}$ and $w=w_{\lambda,C^2}$ to \eqref{eq:EL pointwise for w} (Proposition~\ref{prop: existence}), we have $w_n \to w$ pointwise, that is for every $s\in \R$, $w_n(s) \to w(s)$.
\end{proposition}

\begin{proof}
  \underline{Case $(\lambda,C^2) \in \partial S_2 \cup {(p-1,0)}$:} As $w_{\lambda,C^2}\equiv 1$, by Lemma~\ref{lemma: continuity of minimal curvature}, we have $w_{m,n} \to 1$, thereby 
  \begin{equation*}
    \|w_{\lambda_n,C_n^2} - w_{\lambda,C^2}\|_{L^\infty(\R)} = \|w_{\lambda_n,C_n^2}-1\|_{L^\infty(\R)} \leq 1-w_{m,n} \to 0.
  \end{equation*}

  \underline{Case $(\lambda,C^2) \in S \cup \partial S_1$:} 
  Using Lemma~\ref{lemma: convergence of s-periods}, the solutions $w_n:[0,P_n]\to [0,1]$ and $w:[0,P]\to [0,1]$ to \eqref{eq:EL pointwise for w} are given by
  \begin{equation*}
    w_n(s) = 
    \begin{cases}
      S^{-1}_n(P_n-s) &\text{if } s<P_n, \\ 
      w_{m,n} &\text{if } s=P_n.
    \end{cases} \qquad   
    \text{ and } \qquad 
    w(s) = \begin{cases}
      S^{-1}(P-s) &\text{if } s<P,\\
      w_{m} &\text{if } s=P.
    \end{cases} 
  \end{equation*}
     Since $S_n$ and $S$ are strictly increasing functions whenever they are positive, the inverses exist and are continuous. In particular, $S_n^{-1}$ and $S^{-1}$ are also strictly increasing on $(0,P_n)$ and $(0,P)$ respectively.

    First, we show that for every $s\in [0,P]$, we have $w_n(s) \to w(s)$. 
    For $s=0$, clearly $w_n(0)=1 = w(0)$. 
    For $s=P$, Lemma~\ref{lemma: continuity of minimal curvature} and Lemma~\ref{lemma: uniform Lip estimate} gives
    \begin{equation*}
      |w_n(P)-w(P)| \leq |w_n(P_n)-w(P)| + |w_n(P_n)-w_n(P)| \leq |w_{m,n}-w_m| + L|P_n-P| \to 0 .      
    \end{equation*}
    Consider now $s\in (0,P)$ fixed and take $\eps>0$ arbitrary. Then, by continuity of $S^{-1}$, there exists $\delta \in (0,\min(s,P-s))$ such that 
    \begin{equation*}
      |S^{-1}(P-(s+\delta)) - S^{-1}(P-s)|,\, |S^{-1}(P-(s-\delta)) - S^{-1}(P-s)|< \eps. 
    \end{equation*}
    As $S^{-1}$ is strictly increasing, this is equivalent to
    \begin{equation}
      \label{eq: proof of pw convergence1}
      S^{-1}(P-(s-\delta)) - \eps < S^{-1}(P-s) < S^{-1}(P-(s+\delta)) + \eps.
    \end{equation}
    From Lemma~\ref{lemma: convergence of s-periods}, 
    \begin{equation*}
    \begin{split}
      S_n(S^{-1}(P-(s+\delta))) \to S(S^{-1}(P-(s+\delta))) = P-(s+\delta),  \\
      S_n(S^{-1}(P-(s-\delta))) \to S(S^{-1}(P-(s-\delta))) = P-(s-\delta),
    \end{split}
    \end{equation*}
    that is for $n$ large enough 
    \begin{equation*}
    \begin{split}
      |S_n(S^{-1}(P-(s+\delta))) - (P_n-(s+\delta))| \leq |S_n(S^{-1}(P-(s+\delta))) - (P-(s+\delta))| + |P_n-P| < \delta, \\
      |S_n(S^{-1}(P-(s-\delta))) - (P_n-(s-\delta))| \leq |S_n(S^{-1}(P-(s-\delta))) - (P-(s-\delta))| + |P_n-P| < \delta.
    \end{split}
    \end{equation*}
    Thereby
    \begin{equation*}
    S_n(S^{-1}(P-(s+\delta))) < P_n-s < S_n(S^{-1}(P-(s-\delta))).
    \end{equation*}
    Upon taking $n$ larger if necessary, as $P_n \to P$, we have $s \in (0,P_n)$ and so $P_n-s \in (0,P_n)$ as well. 
    By strict monotonicity of $S_n^{-1}$,
    \begin{equation*}
      S^{-1}(P-(s+\delta)) < S_n^{-1}(P_n-s) < S^{-1}(P-(s-\delta)).
    \end{equation*}
    Using now \eqref{eq: proof of pw convergence1},
    \begin{equation*}
      S^{-1}(P-s) - \eps < S^{-1}(P-(s+\delta)) < S_n^{-1}(P_n-s) < S^{-1}(P-(s-\delta)) < S^{-1}(P-s) + \eps.
    \end{equation*}
    Thus $|w_n(s) -w(s)| = |S^{-1}_n(P_n-s) -S^{-1}(P-s)|< \eps$, and the result follows as $s$ and $\eps$ were arbitrary.

    Finally, for $s'>P$ fixed, there exists $m\in \N$ such that $s'=mP+s$ with $s\in [0,P)$. Thus 
    \begin{equation*}
      \begin{split}
        |w_n(s') - w(s')| &= |w_n(mP+s) - w(mP+s)|  
        \leq |w_n(mP+s)-w_n(mP_n+s)| + |w_n(mP_n+s)-w(mP+s)|. 
      \end{split}
    \end{equation*}
    The first term is bounded by $Lm|P_n-P|$ (Lemma~\ref{lemma: uniform Lip estimate}), which converges to zero. 
    If $m$ is even, then by periodicity $|w_n(mP_n+s)-w(mP+s)| = |w_n(s)-w(s)|$, which vanishes in the limit. 
    If $m$ odd, since $w$ is even (Proposition~\ref{prop: existence})
    \begin{equation*}
    |w_n(mP_n+s)-w(mP+s)| = |w_n(P_n-s)-w(P-s)| \leq |w_n(P_n-s)-w_n(P-s)| + |w_n(P-s)-w(P-s)|.  
    \end{equation*}
    The first term is bounded by $L|P_n-P|\to 0$ (Lemma~\ref{lemma: uniform Lip estimate}) and, as $P-s \in [0,P)$, the second term converges to zero in the limit as well.
     Noticing that $w(-s)=w(s)$ gives pointwise convergence for every $s\in \R$, the proof is finished.
\end{proof}

We note the following remark, strengthening the convergence locally.
\begin{remark}
  \label{rmk: local uniform convergence of w_n}
  From Lemma~\ref{lemma: uniform Lip estimate}, the sequence $w_n$ is equicontinuous, thus by Arzela-Ascoli $w_n \to w$ locally uniformly, i.e.\ $\|w_n-w\|_{L^\infty(K)}\to 0$ for every $K\subset \subset \R$. 
  However, 
  as the periods $P_n$ and $P$ can differ, global $L^\infty(\R)$ convergence may not hold. 
  The exact same convergence results hold for $k_n \to k$ and $u_n\to u$.
\end{remark}

\section{Killing fields and cylindrical coordinates}
\label{section Killing fields}
We set up a cylindrical coordinate system for non-planar ($C^2>0$) curves, following the ideas in \cite{langer-singer_classification}.
\begin{definition}
  A \textit{Killing vector field on $\R^3$} is a $C^1$ vector field $X:\R^3 \to \R^3$ such that there exist $a,\omega \in \R^3$ with $X(x) = a + \omega \times x$ for every $x$. 
\end{definition}

\begin{definition}
  A \textit{Killing vector field along $\gamma$} is the restriction of a Killing vector field in $\R^3$ to $\gamma$. 
\end{definition}

We first find two Killing vector fields along non-planar $p$-elasticae.
\begin{lemma}
\label{lemma: two Killing vector fields}
  Let $k$ be a solution to \eqref{eq:EL for k} for $C^2> 0$ on $[0,L]$. By \eqref{eq:EL for k and tau}, set $\tau =C k^{2-2p}$. Let $\gamma:[0,L]\to \R^3$ be an associated curve with tangent $T$, normal $N$ and binormal $B$. 
  Then the vector fields 
  \begin{equation*}
    J_0 = T ((p-1)k^p - \lambda) + N (p(p-1)k^{p-2}k') + B(pk^{p-1}\tau)\qquad \text{and} \qquad I = k^{p-1}B 
  \end{equation*}
  are Killing vector fields along $\gamma$.
  Moreover, $|J_0|^2 = (p-1-\lambda)^2 + p^2C^2$.
\end{lemma}
\begin{proof}
  Since $C,k,\tau$ are all strictly positive, it follows from \cite[Theorem~1.4]{ourpaper}, that $k,\tau$ and $\gamma$ are all smooth and furthermore $T, N, B$ and so $J_0$ and $I$ are well defined. 
  It remains to check the conditions from Proposition~\ref{prop: characterization of Killing field}. We compute, using $\tau=Ck^{2-2p}$ and the $p$-elastica equation \eqref{eq:EL for k},
  \begin{equation*}
  \begin{split}
    J_0' &= [(p-1)pk^{p-1}k']T + [p(p-1)(p-2)k^{p-3}k'^2+p(p-1)k^{p-2}k'']N+[pC(1-p)k^{-p}k']B \\
    &\quad [(p-1)k^p-\lambda]kN+[p(p-1)k^{p-2}k'][-kT+\tau B]+[pCk^{1-p}][-\tau N] \\
    &= N[p(p-1)(p-2)k^{p-3}k'^2 + p(p-1)k^{p-2}k'' + (p-1)k^{p+1} -\lambda k -pC^2k^{3-3p}] \\
    &= N[p(k^{p-1})''+(p-1)k^{p+1}- pC^2k^{3-3p}-\lambda k] =0,
  \end{split}
  \end{equation*}
  and so $J_0$ is actually constant and trivially Killing.
  For $I$ we have
  \begin{equation*}
  \begin{split}
  I' &= (p-1)k^{p-2}k'B - k^{p-1}\tau N = (p-1)k^{p-2}k'B - Ck^{1-p} N, \\
  I'' &=  [(p-1)(p-2)k^{p-3}k'^2+(p-1)k^{p-2}k'']B-(p-1)k^{p-2}k'\tau N-C(1-p)k^{-p}k'N-Ck^{1-p}(-kT+\tau B),\\
  &\quad = [(p-1)(p-2)k^{p-3}k'^2+(p-1)k^{p-2}k''- C^2k^{3-3p}]B + Ck^{2-p}T \\
  I''' &= [(p-1)(p-2)((p-3)k^{p-4}k'^3 + k^{p-3}3k'k'') + (p-1)k^{p-2}k''' - (3-3p)C^2k^{2-3p}k']B + (\dots)N+(\dots)T.
  \end{split}
  \end{equation*}
  We observe that it remains to check only the last condition, since $I'\perp T$ and $I'' \perp N$. We compute
  \begin{equation*}
  \begin{split}
    &\langle I'''-\frac{k'}{k}I''+k^2 I',B \rangle \\ &\quad= [(p-1)(p-2)((p-3)k^{p-4}k'^3 + k^{p-3}3k'k'') + (p-1)k^{p-2}k''' - (3-3p)C^2k^{2-3p}k'] \\&\qquad-\frac{k'}{k}[(p-1)(p-2)k^{p-3}k'^2+(p-1)k^{p-2}k''-C^2k^{3-3p}]  + k^2 (p-1)k^{p-2}k' \\
    &\quad = (p-1)\big( (p-2)((p-4)k^{p-4}k'^3 + k^{p-3}3k'k'') + k^{p-2}k''' -k^{p-3}k'k''  + k^{p}k' \big) + (3p-2)C^2k^{2-3p}k' .
  \end{split}
  \end{equation*}
  From dividing the elastica equation \eqref{eq:EL for k} by $k$ and differentiating, we get
  \begin{equation*}
    \begin{split}
      0&=p(p-1)k^{p-3}k''' + p(p-1)(3p-7)k^{p-4}k'k''+ p(p-1)(p-2)(p-4)k^{p-5}k'^3 
      \\ &\quad + p(p-1)k^{p-1}k'- pC^2(2-3p)k^{1-3p}k',
    \end{split}
  \end{equation*}
  which upon multiplying by $\frac{k}{p}$ gives exactly the same terms in $\langle I'''-\frac{k'}{k}I''+k^2I',B \rangle$.
  Hence $J_0$ and $I$ are both Killing vector fields.

To compute $|J_0|$, we compute it at the maximum of $k$ (where in our case $k_M=1$), that is
\begin{equation*}
  |J_0|^2 = ((p-1)k_M^p-\lambda)^2 + p^2 \frac{C^2}{k_M^{2p-2}}=(p-1-\lambda)^2 + p^2 C^2,
\end{equation*}
finishing the proof.
\end{proof}

We now construct a rotational vector field around the axis $\frac{J_0}{|J_0|}$.

\begin{lemma}
The vector field
\begin{equation*}
  J_1 = J_0 - \frac{1}{p C} |J_0|^2 k^{p-1} B,
\end{equation*}
with 
\begin{equation}
  \label{eq: normsq of J_1}
  |J_1|^2 = \left\langle J_0 - \frac{1}{p C} |J_0|^2 k^{p-1} B, J_0 - \frac{1}{p C} |J_0|^2 k^{p-1} B \right\rangle  = \frac{|J_0|^2 }{p^2C^2} (|J_0|^2 k^{2p-2}-p^2C^2).
\end{equation} 
is a pure rotation along the axis $\frac{J_0}{|J_0|}$.
\end{lemma}
\begin{proof}
Clearly $J_1$ is a Killing vector field as a linear combination of the vector fields in Lemma~\ref{lemma: two Killing vector fields} and hence $J_1(s) =a+\omega \times \gamma(s)$ for some $a,\omega \in \R^3$. 
If $k$ is not constant, so is $J_1$ by \eqref{eq: normsq of J_1}. 
If $k$ is constant, $\gamma$ is a helix, i.e.\ $B$ and therefore also $J_1$ are not constant. 
To conclude, $\omega \neq 0$.
Since
\begin{equation*}
    \langle J_0, \omega \times \gamma(s) \rangle + \langle J_0, a \rangle = \langle J_0, J_1 \rangle = |J_0|^2 - \frac{1}{pC} |J_0|^2 k^{p-1} p k^{p-1} \tau = |J_0|^2 - \frac{1}{pC} |J_0|^2 k^{p-1} p k^{p-1} C k^{2-2p} = 0,
\end{equation*}
the quantity $\langle J_0,a \rangle$ and thus also $\langle \gamma(s), \omega \times J_0 \rangle$ are constant. 
However, as $\gamma(s)$ is non-planar, $\omega \parallel J_0$ and $\langle J_0,a \rangle=0$. 
In other words, $a = - \omega \times \gamma_0$ for some $\gamma_0 \in \R^3$ and thereby $J_1(s) = \omega \times (\gamma(s) -\gamma_0)$, a pure rotation along the axis $\omega$ parallel to the constant vector field $J_0$.
\end{proof}

We conclude that the vector fields
\begin{equation*}
  \begin{split}
    \frac{\partial}{\partial z} = \frac{J_0}{|J_0|}, \qquad \frac{\partial}{\partial \theta} = QJ_1, \qquad
    \frac{\partial}{\partial r} = \frac{J_0 \times B}{|J_0 \times B|},
  \end{split}
\end{equation*}
form a cylindrical coordinate system.
Note that $Q$ is a normalization factor, which is to be determined.
We also stress that these cylindrical coordinates are only defined under the assumption $C^2>0$ and become singular when $C=0$.

\begin{proposition}
  \label{prop: derivatives for cylindrical coordinates}
  Let $\gamma(s) = (r(s), \theta(s), z(s))$. Then $Q=\frac{p^2C}{|J_0|^3}$ and moreover
  \begin{equation*}
  \begin{split}
    r_s &= \frac{p (k^{2p-2})'}{2\sqrt{|J_0|^2 k^{2p-2}- p^2 C^2}}, \\
    \theta_s &= \frac{C |J_0| ((p-1)k^p - \lambda)}{k^{2p-2}|J_0|^2-p^2C^2}, \\
    z_s &= \frac{(p-1)k^p-\lambda}{|J_0|}.
  \end{split}
  \end{equation*}
\end{proposition}

\begin{proof}
  \underline{Step 1:}
  Write $T=r_s \frac{\partial}{\partial r} + \theta_s \frac{\partial}{\partial \theta} + z_s \frac{\partial}{\partial z}$. 
  We directly have 
  \begin{equation*}
    z_s = \left\langle T, \frac{J_0}{|J_0|} \right\rangle = \frac{(p-1)k^p-\lambda}{|J_0|}.
  \end{equation*}
  Since 
  \begin{equation*}
    J_0 \times B = -N ((p-1)k^p - \lambda) + T (p(p-1)k^{p-2}k') = -N \left((p-1)k^p - \lambda \right) + T \left( p(k^{p-1})' \right),
  \end{equation*}
  and from \eqref{eq:EL for w first order},
  \begin{equation*}
  \begin{split}
 & ((p-1)k^p - \lambda)^2 + p^2(k^{p-1})'^2 \\
      &\quad = (p-1)^2 k^{2p} - 2(p-1)\lambda k^p + \lambda^2 + \left( (p-1)^2 -2\lambda (p-1) + p^2 C^2 - (p-1)^2 k^{2p} + 2\lambda(p-1)k^p - p^2C^2 k^{2-2p}  \right) \\
      &\quad = (p-1-\lambda)^2 + p^2C^2 - p^2C^2 k^{2-2p},
  \end{split}
  \end{equation*}
  we have 
  \begin{equation*}
    \begin{split}
      r_s &= \left\langle T, \frac{\partial}{\partial r} \right\rangle = \left\langle T, \frac{J_0 \times B}{|J_0 \times B|} \right\rangle = \frac{p(k^{p-1})'}{\sqrt{((p-1)k^p - \lambda)^2 + p^2(k^{p-1})'^2}} \\
      &= \frac{p(p-1)k^{p-2}k'}{\sqrt{|J_0|^2 -\frac{p^2C^2}{k^{2p-2}}}} = \frac{p(p-1)k^{p-2}k'k^{p-1}}{\sqrt{|J_0|^2 k^{2p-2}-p^2C^2}} = \frac{p(k^{2p-2})'}{2\sqrt{|J_0|^2 k^{2p-2}-p^2C^2}}.
    \end{split}
  \end{equation*}

  \underline{Step 2:}
  We compute the normalization factor $Q$. Take $s_M$ such that $k(s)$ and $r(s)$ attain their maxima at $s_M$. Then $\left|\frac{\partial}{\partial \theta}(s_M) \right| = r(s_M) =: R_\mathcal{C}$. 
  Note that at $s_M$ we have for the vertical component of the unit tangent vector $T$, 
  \begin{equation*}
    \left\langle T(s_M), \frac{\partial}{\partial z}(s_M) \right\rangle =  \left\langle T(s_M), \frac{J_0(s_M)}{|J_0(s_M)|} \right\rangle = \frac{(p-1)k(s_M)^p - \lambda}{|J_0|} = \frac{(p-1) - \lambda}{|J_0|},
  \end{equation*}
  by Pythagoras, the horizontal component is
  \begin{equation*}
    \left\langle T(s_M), \frac{\frac{\partial}{\partial \theta}(s_M)}{\left|\frac{\partial}{\partial \theta}(s_M) \right|} \right\rangle = \frac{pC}{|J_0|}.
  \end{equation*}
  This horizontal component corresponds to a tangent vector $v(s_M)$ to the circle $\mathcal{C}$ centered at $(0,0,z(s_M))$ with radius $r(s_M)=R_\mathcal{C}$. 
  To compute $R_\mathcal{C}$ we take the reciprocal of the curvature $k_\mathcal{C}$ to $\mathcal{C}$. Let $\tilde s$ be the arclength parameter of $\mathcal{C}$ (which generally differs from $s$) and note that $\frac{d\tilde s}{ds}(s_M) = |v(s_M)| = \frac{pC}{|J_0|}$. Then
  \begin{equation*}
    k_{\mathcal{C}} = \left| \frac{d}{d\tilde{s}}\bigg|_{\tilde s (s_M)} \left(\frac{\frac{\partial}{\partial \theta}}{\left|\frac{\partial}{\partial \theta} \right|} \right) \right| = \left| \frac{d}{ds} \bigg|_{s_M} \left( \frac{\frac{\partial}{\partial \theta}}{\left|\frac{\partial}{\partial \theta} \right|} \right) \right| \cdot \left| \frac{d s}{ d\tilde s}(s_M) \right| = \frac{|J_0|}{pC} \left(\frac{J_1}{|J_1|} \right)'(s_M).
  \end{equation*}
  At the maximum of $k=k_M=1$, we have $k'=0$ and so
  \begin{equation*}
    J_1' = -\frac{1}{pC} |J_0|^2((p-1)k^{p-2}k'B+k^{p-1}(-\tau)N) = \frac{|J_0|^2}{p} N,
  \end{equation*}
  as well as $\langle J_1,J_1' \rangle=0$ and hence 
  \begin{equation*}
    \left(\frac{J_1}{|J_1|} \right)'(s_M) = \frac{J_1'|J_1|^2-J_1 \langle J_1, J_1' \rangle}{|J_1|^3} (s_M) = \frac{|J_1'(s_M)|}{|J_1(s_M)|},
  \end{equation*}
  which finally gives 
  \begin{equation*}
    \frac{1}{R_\mathcal{C}}= k_\mathcal{C} = \frac{|J_0|}{pC} \frac{|J_0|^2}{p |J_1(s_M)|}= \frac{|J_0|^2}{p (p-1-\lambda)} = \frac{(p-1-\lambda)^2 + p^2C^2 }{p(p-1-\lambda)}.
  \end{equation*}
 The expression for $Q$ follows directly from
  \begin{equation*}
     Q|J_1(s_M)| =\left|\frac{\partial}{\partial \theta}(s_M)\right| = R_\mathcal{C} = \frac{p^2 C}{|J_0|^3} |J_1(s_M)| .
  \end{equation*}
  
  \underline{Step 3:} We compute directly
  \begin{equation*}
    \begin{split}
      \theta_s &= \frac{1}{\left|\frac{\partial}{\partial \theta}\right|^2} \left\langle T, \frac{\partial}{\partial \theta}\right\rangle = \frac{1}{Q |J_1|^2} ((p-1)k^p-\lambda) \\
      &= \frac{p^2C^2 }{|J_0|^2  (|J_0|^2 k^{2p-2}-p^2C^2)} \frac{|J_0|^3}{p^2 C}((p-1) k^p-\lambda) = \frac{C |J_0|((p-1)k^p-\lambda)}{(|J_0|^2 k^{2p-2}-p^2C^2)},
    \end{split}
  \end{equation*}
  which finishes the proof.
\end{proof}

Note that the critical points of $r$ and $k$ coincide, a minimum (maximum) of $r$ corresponds to a minimum (maximum) of $k$ and they have the same period. 
Moreover, the critical points of $z$ and $\theta$ coincide and $z$ and $\theta$ differ from the periodic function $k$ with period $P$ by a linear function. 
The elastica lies between two concentric cylinders with outer radius $R_\mathcal{C}$ and inner radius $r_\mathcal{C}$, which is given by
\begin{equation}
\label{eq: r_C definition}
   r_\mathcal{C} = r_\mathcal{C}(\lambda,C^2) = R_\mathcal{C} - \int_{s_m}^{s_M} r_s\, ds = \frac{p\sqrt{|J_0|^2k_m^{2p-2} -p^2C^2}}{|J_0|^2} .
\end{equation}

We now characterize the set where $r_\mathcal{C}=0$, i.e.\ the  case where the $p$-elastica passes through the $z$-axis.
\begin{lemma}
  If $(\lambda,C^2)\in S$, then $r_\mathcal{C}(\lambda,C^2)=0$ if and only if $\lambda\geq 0$ and
  \label{lemma: condition for rm=0}
  \begin{equation*}
   C^2 = \frac{(p-1-\lambda)^2}{p^2} \frac{1}{1-\left(\frac{\lambda}{p-1}\right)^{\frac{2p-2}{p}}} \left(\frac{\lambda}{p-1}\right)^\frac{2p-2}{p}.
  \end{equation*}
  The set $\{r_\mathcal{C}=0\}$ is a smooth curve in $S^0 \subset S$ connecting $(0,0)$ to $(p-1,0)$ and $(p-1)k_m^p-\lambda=0$ along that curve.
\end{lemma}
\begin{proof}
  For $(\lambda,C^2) \in S$, the function $r_\mathcal{C}$ vanishes if and only if $k_m^{2p-2}=\frac{p^2 C^2}{(p-1-\lambda)^2 + p^2 C^2}$, i.e.\ (since $f$ only has one root strictly less than $1$ at $k_m$) if and only if $f((p^2C^2 |J_0|^{-2})^{\frac{1}{2p-2}})=0$. 
  In that case, we have
  \begin{equation}
    \label{eq: f evaluated}
    \begin{split}
      0&=f\left(\left(\frac{p^2 C^2}{(p-1-\lambda)^2 + p^2 C^2}\right)^{\frac{1}{2p-2}}\right) \\
      &= \frac{(p-1)^2}{p^2}-2\lambda \frac{p-1}{p^2}+C^2 - \frac{(p-1)^2}{p^2} \left(\frac{p^2 C^2}{(p-1-\lambda)^2 + p^2 C^2}\right)^{\frac{2p}{2p-2}} \\ 
      &\quad + 2\lambda \frac{p-1}{p^2} \left(\frac{p^2 C^2}{(p-1-\lambda)^2 + p^2 C^2}\right)^{\frac{p}{2p-2}} - C^2 \left(\frac{p^2 C^2}{(p-1-\lambda)^2 + p^2 C^2}\right)^{\frac{2-2p}{2p-2}} \\
      &= - \frac{(p-1)^2}{p^2} \left(\frac{p^2 C^2}{(p-1-\lambda)^2 + p^2 C^2}\right)^{\frac{2p}{2p-2}}  + 2\lambda \frac{p-1}{p^2}\left(\frac{p^2 C^2}{(p-1-\lambda)^2 + p^2 C^2}\right)^{\frac{p}{2p-2}} - \frac{\lambda^2}{p^2} \\
      &= -\frac{1}{p^2}\left( (p-1) \left(\frac{p^2C^2}{(p-1-\lambda)^2 + p^2C^2} \right)^{\frac{p}{2p-2}}-\lambda \right)^2.
    \end{split}
  \end{equation}
That is, any $(\lambda,C^2) \in S$ satisfies $r_\mathcal{C}(\lambda, C^2)=0$ if and only if it also satisfies
\begin{equation*}
  0=(p-1)\left(\frac{p^2 C^2}{(p-1-\lambda)^2 + p^2C^2} \right)^{\frac{p}{2p-2}} - \lambda = (p-1)k_m^p-\lambda,
\end{equation*}
i.e.\ $\lambda$ cannot be negative.
Solving for $C^2$ gives the explicit curve.
\end{proof}

We are now able to quantify Lemma~\ref{lemma: global lower bound for um}.
\begin{corollary}
  \label{cor: global lower bound for km}
  For every $(\lambda,C^2)\in S$ we have $k_m(\lambda,C^2) \geq \left(\frac{p^2 C^2}{(p-1-\lambda)^2 + p^2 C^2}\right)^{\frac{1}{2p-2}}$ with equality if and only if $(\lambda,C^2)\in \{r_\mathcal{C}=0\}$.
\end{corollary}
\begin{proof}
  Suppose that $k_m<\left(\frac{p^2 C^2}{(p-1-\lambda)^2 + p^2 C^2}\right)^{\frac{1}{2p-2}}$, then by Lemma~\ref{lemma: number of solutions} we have $f\left((\frac{p^2 C^2}{(p-1-\lambda)^2 + p^2 C^2})^{\frac{1}{2p-2}} \right)>0$, which contradicts \eqref{eq: f evaluated}.
\end{proof}

Let $\Delta z$ and $\Delta \theta$  denote the change in $z$ and $\theta$ respectively through one period of $k$, starting from its minimum $k_m$, that is 
\begin{equation*}
  \Delta z := \int_{-P}^P z_s(s) ds = 2\int_{-P}^0 z_s(s) ds \qquad \text{and} \qquad \Delta \theta := \int_{-P}^P \theta_s(s) ds = 2\int_{-P}^0 \theta_s(s)ds. 
\end{equation*}
We note that these functions are a priori only defined for $(\lambda, C^2)\in S$ where $C^2>0$. In the next section we extend them suitably to the boundary of $S^\rho$. 

We finally show that for a $p$-elastica $\gamma_{\lambda,C^2}$ to be closed, $\Delta z(\lambda,C^2)=0$ and $\Delta \theta(\lambda,C^2) \in 2\pi \mathbb{Q}$ at the same time. 
To find and characterize $(\lambda, C^2) \in S$ that satisfy these two conditions is the main part of the subsequent sections.

\begin{lemma}
    \label{lemma: closedness conditions} 
    The following are equivalent:
    \begin{enumerate}
        \item[(i)] The curve $\gamma_{\lambda,C^2}\in W^{2,p}(\R/\Z;\R^3)$ is a closed non-planar $p$-elastica.
        \item[(ii)] The open curve $\tilde \gamma_{\lambda,C^2}\in W^{2,p}(0,1;\R^3)$ is a  non-planar $p$-elastica with $\tilde \gamma_{\lambda,C^2}(0)=\tilde \gamma_{\lambda,C^2}(1)$ and $\tilde \gamma_{\lambda,C^2}'(0)=\tilde \gamma_{\lambda,C^2}'(1)$.
        \item[(iii)] The parameter couple $(\lambda,C^2) \in S$ satisfies $\Delta z(\lambda,C^2)=0$ and $\Delta \theta(\lambda,C^2) \in 2\pi \Q$.
    \end{enumerate}
\end{lemma}

\begin{proof}
The implication $(i) \implies (ii)$ is trivial, as noted in Remark~\ref{rmk: closed implies closed in C1}.
To show $(ii)\implies (iii)$, from $\tilde \gamma_{\lambda,C^2}(0)=\tilde \gamma_{\lambda,C^2}(L)$, we have $k(0)=k(L)$ as
\begin{equation*}
    \frac{p\sqrt{|J_0|^2k(0)^{2p-2} -p^2C^2}}{|J_0|^2}=r(0) = r(L) = \frac{p\sqrt{|J_0|^2k(L)^{2p-2} -p^2C^2}}{|J_0|^2}.
\end{equation*}
Since $k$ is monotone while going from $k_m$ to $k_M$ (which we assume to be $1$, after a suitable normalization), see Remark~\ref{rmk: k prime}, $k(0)=k(L)$ if and only if the length $L$ is a positive integer multiple $m$ of the curvature period $2P(\lambda,C^2)$.
   This periodicity of $k$ (thus of $\theta_s$, $z_s$, $r_s$ and $r$) implies for the cylindrical coordinates,
    \begin{equation*}
    \begin{split}
         m \Delta z(\lambda,C^2) &= z(2mP(\lambda,C^2))-z(0) = 0, \\
         m \Delta \theta(\lambda,C^2) &= \theta(2mP(\lambda,C^2))-\theta(0) = 0 \mod 2\pi,\\
    \end{split}
    \end{equation*}
   and hence $\Delta z(\lambda,C^2)=0$ and $\Delta \theta(\lambda,C^2) \in 2\pi \mathbb{Q}$.

    Lastly we show $(iii) \implies (i)$. 
    Indeed, $\Delta z(\lambda,C^2)=0$ and $\Delta \theta(\lambda,C^2) = 2\pi \frac{n}{m} \in 2\pi \mathbb{Q}$ implies for the corresponding arclength-parameterized open curve (constructed by the Frenet--Serret frame), $\tilde \gamma_{\lambda,C^2}(0) = \tilde \gamma_{\lambda,C^2}(L)$, where $L=2mP(\lambda,C^2)$. 
    Concerning the condition for $\tilde \gamma'_{\lambda,C^2}$, note that in cylindrical coordinates
    \begin{equation*}
       \tilde \gamma_{\lambda,C^2}'(s) = r_s(s) \frac{\partial}{\partial r}(s) + \theta_s(s) \frac{\partial}{\partial \theta}(s) + z_s(s) \frac{\partial}{\partial z}(s).
    \end{equation*}
    Since all the terms on the RHS are $L$-periodic, $\tilde \gamma'_{\lambda,C^2}(0) = \tilde \gamma_{\lambda,C^2}'(L)$. Moreover,
    \begin{equation*}
    \begin{split}
        \tilde \gamma_{\lambda,C^2}''(s)= \left(r_{ss}(s)-r(s)\theta_s(s)^2\right)\frac{\partial}{\partial r}(s)+\left(2\frac{r_s(s)}{r(s)}\theta_s(s)+\theta_{ss}(s)\right)\frac{\partial}{\partial \theta}(s) +z_{ss}(s)\frac{\partial}{\partial z}(s),
    \end{split}
    \end{equation*}
    and since $r_{ss}$, $\theta_{ss}$ and $z_{ss}$ are also $L$-periodic as derivatives of periodic functions, $\tilde \gamma''_{\lambda,C^2}(0)= \tilde \gamma''_{\lambda,C^2}(L)$.
    In particular, as $\{\Delta z=0\}$ and $\{r_{\mathcal{C}}=0\}$ only intersect at $(p-1,0)$ (see Proposition~\ref{prop: global growth of Delta z=0}), we have $s\mapsto r(s)$ being bounded away from zero. 
    From \eqref{eq: first variation}, 
    \begin{equation*}
       (|\tilde \gamma''|^{p-2}\tilde \gamma'')'(s)= A+\frac{1-2p}{p}|\tilde \gamma''(s)|^p \tilde\gamma'(s) + \frac{\lambda}{p} \tilde\gamma'(s)
    \end{equation*}
    for some $A\in \R^3$ and so after moving terms (cf.\ \cite[Proposition~2.7]{ourpaper})
    \begin{equation}
    \label{eq: expression for gamma'''}
    \begin{split}
    \tilde\gamma_{\lambda,C^2}'''(s) = |\tilde\gamma_{\lambda,C^2}''(s)|^{2-p}
    \left(I-\frac{p-2}{p-1}\frac{\tilde\gamma_{\lambda,C^2}''(s)\otimes \tilde\gamma_{\lambda,C^2}''(s)} {|\tilde\gamma_{\lambda,C^2}''(s)|^2} \right) \cdot \left( A+\left(\frac{1-2p}{p}|\tilde\gamma_{\lambda,C^2}''(s)|^p+\frac{\lambda}{p} \right) \tilde\gamma_{\lambda,C^2}'(s) \right),
    \end{split}
    \end{equation}
    i.e.\ $\tilde\gamma_{\lambda,C^2}'''(0)=\tilde\gamma_{\lambda,C^2}'''(L)$. 
    Differentiating \eqref{eq: expression for gamma'''} repeatedly (using the fact that $|\tilde\gamma_{\lambda,C^2}''(s)| =k >0$ on $[0,L]$), gives $\tilde\gamma_{\lambda,C^2}^{(k)}$ in terms of lower order derivatives, that is $\tilde\gamma_{\lambda,C^2}^{(k)}(0)=\tilde\gamma_{\lambda,C^2}^{(k)}(L)$ for every $k \in \N$ follows inductively.
    To conclude, the open curve $\tilde \gamma_{\lambda,C^2}$ is smoothly closed, i.e.\ $\gamma_{\lambda,C^2} \in C^\infty(\R/\Z;\R^3)$ and \eqref{eq: first variation closed} holds, since the boundary terms arising from integration by parts cancel.
\end{proof}

\section{The function \texorpdfstring{$\Delta z$}{Delta z}}
\label{section: Delta z}

The main object of this section is the function $\Delta z: S \cup \partial S_2 \to \R$ (which so far is only defined for $C^2>0$) with
\begin{equation}
  \label{eq: Delta z original}
   \Delta z(\lambda,C^2) = 2\int_{-P(\lambda,C^2)}^0 z_s(s) ds = 2\int_{-P(\lambda,C^2)}^0 \frac{(p-1)k_{\lambda,C^2}(s)^p-\lambda}{\sqrt{(p-1-\lambda)^2+p^2 C^2}} ds.
\end{equation}

\begin{lemma}
  \label{lemma: continuity of Delta z inside}
  The function $\Delta z$ is continuous in $S^0 \cup \partial S^0_2 \setminus (p-1,0)$.
\end{lemma}
\begin{proof}
  Let $(\lambda_n,C_n^2)\to(\lambda,C^2)\in S^0 \cup \partial S^0_2$ where $\lambda<p-1$, and denote their corresponding periods by $P_n$ and $P$. Then 
  \begin{equation*}
    \begin{split}
      |\Delta z(\lambda_n,C_n^2)-\Delta z(\lambda,C^2)| &= 2\left|\int_{-P_n}^0 (z_n)_s(s) ds - \int_{-P}^0 z_s(s)ds \right| \\
      &= 2\left| \int_{-P}^0(z_n)_s(s)ds - z_s(s)ds + \int_{-P_n}^{-P}(z_n)_s(s)ds \right| \\
      &\leq 2\|(z_n)_s-z_s\|_{L^1(0,P)} + 2|P_n-P|\|(z_n)_s\|_{L^\infty}.
    \end{split}
  \end{equation*}
  Since $(z_n)_s = \frac{(p-1)k_n^p-\lambda_n}{\sqrt{(p-1-\lambda_n)^2 + p^2 C_n^2}} \to \frac{(p-1)k^p-\lambda}{\sqrt{(p-1-\lambda)^2 + p^2 C^2}} = z_s$ in $L^1(K)$ for any compact set $K$ by Remark~\ref{rmk: local uniform convergence of w_n} (the denominator does not vanish thanks to $C^2>0$), the first term converges to zero. The second term converges also to zero by Proposition~\ref{prop: continuity of period} and uniform boundedness of $(z_n)_s$.
\end{proof}

By the change of variables $k=k(s)$, $dk=k'(s)ds$ and Remark~\ref{rmk: k prime}, we have for $(\lambda,C^2) \in S$,
\begin{equation*}
  \Delta z(\lambda,C^2) = 2(p-1) \int_{k_m}^{1} \frac{(p-1)k^p - \lambda}{|J_0|} \frac{1}{\sqrt{E k^{4-2p} -\frac{(p-1)^2}{p^2} k^{4} + 2\lambda \frac{p-1}{p^2} k^{4-p} - C^2 k^{6-4p}  }} dk,
\end{equation*}
where $E=\frac{(p-1)^2}{p^2}-2\lambda \frac{p-1}{p^2} + C^2$.
Letting $u=k^p$, $du = pk^{p-1}dk = pu^{\frac{p-1}{p}}dk$,
  \begin{equation}
  \label{eq: Delta z in u}
   \Delta z(\lambda,C^2)=\frac{2(p-1)}{p|J_0|} \int_{u_m}^{1} \frac{(p-1)u - \lambda}{\sqrt{  E u^{\frac{2}{p}} -\frac{(p-1)^2}{p^2} u^{\frac{2}{p}+2} + 2\lambda \frac{p-1}{p^2} u^{\frac{2}{p}+1} - C^2 u^{\frac{4}{p}-2}  }} du =: \frac{2(p-1)}{p|J_0|} \int_{u_m}^{1} \frac{(p-1)u - \lambda}{\sqrt{\phi(u)}} du.
  \end{equation}

\begin{remark}
  We note that for $(\lambda,C^2) \in S$ fixed, we have $\phi(u)=u^{\frac{2}{p}} h(u)$. Since $\phi'$ does not vanish at its two positive zeros ($u_m$ and $1$) from Lemma~\ref{lemma: estimate for gprime at wm} and Lemma~\ref{lemma: estimates for gprime and gbis at 1}, we estimate $\phi(u) \approx \phi'(u_m)(u-u_m)$ and $\phi(u) \approx \phi'(1)(u-1)$ around the singular points. 
  Roughly speaking, the integrand behaves like $\frac{1}{\sqrt{u}}$ around the endpoints, i.e.\ it is integrable and so $\Delta z(\lambda,C^2)$ is finite inside $S$. 
\end{remark}

We also directly can characterize the value of $\Delta z$ on parts of the boundary.
\begin{lemma}
  \label{lemma: positivity on diagonal}
  Let $(\lambda,C^2) \in \partial S_2^\rho$ with $\lambda < p-1$ or $(\lambda,C^2) \in \partial S_3^\rho$. Then $\Delta z(\lambda,C^2)>0$ .
\end{lemma}
\begin{proof}
  The case $(\lambda,C^2) \in \partial S_3^\rho$ is trivial, for $(\lambda,C^2) \in \partial S_2^\rho$ with $\lambda < p-1$, from \eqref{eq: Delta z original}
  \begin{equation}
      \Delta z(\lambda,C^2) = \int_{-P(\lambda,C^2)}^0 \frac{(p-1)k(s)^p -\lambda}{\sqrt{(p-1-\lambda)^2 +p^2C^2}} \in (0,\infty),
  \end{equation}
  since the denominator is bounded and $(p-1)k(s)^p \geq (p-1)k_m^p \to p-1 >\lambda$.
\end{proof}

\subsection{Extension to \texorpdfstring{$\overline{S^\rho}$}{S rho}}

Note that since the cylindrical coordinates are only defined for $C^2>0$, i.e.\ $\Delta z$ is so far not defined on $\partial S_1$. We show that we can extend it continuously to  $\partial S^\rho_1$, that is a continuous extension exists to $\overline{S^\rho}$.

\begin{lemma}
  \label{lemma: limit os Delta z at S_1}
  Let $(\lambda_n,C^2_n) \to (\lambda,0) \in \partial S_1^\rho$ with $\lambda<p-1$, then
    \begin{equation*}
       \Delta z(\lambda_n,C_n^2) \to \frac{2(p-1)}{p|J_0|} \int_{u_{m}}^1 \frac{(p-1)u-\lambda}{u^{\frac{1}{p}}\sqrt{\left( \frac{(p-1)^2}{p^2}-2\lambda \frac{p-1}{p^2} \right)-\frac{(p-1)^2}{p^2} u^2 + 2\lambda\frac{p-1}{p^2}u}}du =: \Delta z(\lambda, 0).
    \end{equation*}
\end{lemma}

\begin{proof}
  We note that $P_n=P(\lambda_n,C^2_n)$ and $k_n=k(\lambda_n,C^2_n)$ have well-defined limits $P$ and $k$ by Proposition~\ref{prop: continuity of period} and ~\ref{prop: pointwise convergence of w_n}. Since also $\lambda<p-1$, we have 
\begin{equation*}
  \lim_n (z_n)_s= \lim_n \frac{(p-1)k_n^p-\lambda_n}{\sqrt{(p-1-\lambda_n)^2+p^2C_n^2}} \to \frac{(p-1)k^p-\lambda}{\sqrt{(p-1-\lambda)^2}}=: z_s
\end{equation*}
  in $L^1(K)$ for any compact $K$ by Remark~\ref{rmk: local uniform convergence of w_n}. Thereby defining 
  \begin{equation*}
    \Delta z(\lambda,0) := 2\int_{-P}^0 \frac{(p-1)k(s)^p-\lambda}{\sqrt{(p-1-\lambda)^2}}\, ds, 
  \end{equation*}
  we have 
  \begin{equation*}
    \begin{split}
      |\Delta z(\lambda_n,C_n^2)-\Delta z(\lambda,0)| &= \left|2\int_{-P_n}^0 (z_n)_s \, ds - 2\int_{-P}^0 z_s \, ds \right| \\
      &= \left| 2\int_{-P}^0(z_n)_s - z_s \, ds + 2\int_{-P_n}^{-P}(z_n)_s \, ds\right| \\
      &\leq2 \|(z_n)_s - z_s\|_{L^1(0,P)} + 2|P_n-P|\|(z_n)_s\|_{L^\infty(\R)}.
    \end{split}
  \end{equation*}
  
  Since $P$ is finite, see Corollary~\ref{cor: finiteness of period}, the first term converges. Note that $(z_n)_s$ is uniformly bounded as $\lambda<p-1$ (that is the denominator is bounded below) and so the second term also vanishes in the limit, using again Proposition~\ref{prop: continuity of period}.
  The integral formula follows by the change of variable $u=k(s)^p$. 
\end{proof}

\begin{remark}
  We have deliberately used $S^\rho$, to avoid dealing with the special case $\lambda=\frac{p-1}{2}$, where the period might become infinite. In fact, formally  
  \begin{equation*}
   \Delta z \left( \tfrac{p-1}{2}, 0 \right) = \frac{4}{p} \int_{0}^1 \frac{(p-1)u-\frac{p-1}{2}}{u^{\frac{1}{p}}\sqrt{ \frac{(p-1)^2}{p^2}(u- u^2)}}du = \frac{4}{p-1} \int_{0}^1 \frac{(p-1)u-\frac{p-1}{2}}{u^{\frac{1}{p}+\frac{1}{2}} \sqrt{1-u}}du.
  \end{equation*}
  The singularity at $u=1$ is always integrable, but the singularity at $u=0$ only if $p>2$, see also Remark~\ref{rmk: period of flatcore}.
\end{remark}

We now evaluate the limiting behavior of $\Delta z$ around the endpoint $(p-1,0)$.

\begin{proposition}
  \label{prop: limit of Delta z at p-1}
  Let $S \ni (\lambda,C^2) \to (p-1,0)$. Then $\Delta z(\lambda,C^2) \to 0$. 
  Furthermore, for any fixed $\sigma \in (1,2)$ there exists $\delta=\delta(p,\sigma)$ such that for all $\lambda \in (p-1-\delta,p-1)$ and for all $C^2 \in [(p-1-\lambda)^\sigma, \frac{p-1-\lambda}{p}]$, we have $\Delta z(\lambda,C^2)>0$.
\end{proposition}

\begin{proof}
Recall that for $(\lambda,C^2) \in S$,
  \begin{equation*}
    \Delta z(\lambda,C^2)  =\frac{2(p-1)}{p|J_{0}|}\int_{u_{m}}^1 \frac{(p-1)u-\lambda}{\sqrt{\phi(u)}}du, 
  \end{equation*}
  where 
  \begin{equation*}
    \phi(u) = u^{\frac{2}{p}} h(u) = \left(\frac{(p-1)^2}{p^2}-2\lambda\frac{p-1}{p^2}+C^2 \right) u^{\frac{2}{p}} -\frac{(p-1)^2}{p^2} u^{\frac{2}{p}+2} + 2\lambda \frac{p-1}{p^2} u^{\frac{2}{p}+1} - C^2 u^{\frac{4}{p}-2}.
  \end{equation*}
  Fix $\gamma\in (1+\frac{\sigma}{2},2)$. 
  By Corollary~\ref{cor: approx of A from estimate at p-1}, there exists $\eps_\gamma>0$ such that for $|\tfrac{p-1-\lambda}{p}-C^2|<\eps_\gamma$,
  \begin{equation}
    \label{eq: estimate for um}
    \begin{split}
      u_m &= w_m^{\frac{p}{p-1}} = \left[1-2\left(\tfrac{p-1-\lambda}{p}-C^2 \right)+ O(|p-1-\lambda|^\gamma)\right]^{\frac{p}{p-1}} = 1 - \frac{2p}{p-1}\left(\tfrac{p-1-\lambda}{p}-C^2 \right) + O(|p-1-\lambda|^\gamma),
    \end{split}
  \end{equation}
  and thus Lemma~\ref{lemma: estimates for gprime and gbis at 1} gives
  \begin{equation}
    \begin{split}
      \label{eq: derivatives of phi}
      \phi'(u) &= \frac{2}{p} u^{\frac{2}{p}-1} h(u) + u^{\frac{2}{p}} h'(u) = \frac{2}{p} u^{\frac{2}{p}-1} h(u) + \frac{p-1}{p} u^{\frac{1}{p}} g'(u^{\frac{p-1}{p}}),\\
      \phi'(1) &= h'(1) = \frac{p-1}{p} g'(1) = -2\frac{p-1}{p} \left(\tfrac{p-1-\lambda}{p}-C^2 \right), \\
      \phi''(u) &= \frac{2}{p}\left(\frac{2}{p}-1\right) u^{\frac{2}{p}-2} h(u) + 2 \frac{2}{p}u^{\frac{2}{p}-1} h'(u) + u^\frac{2}{p} h''(u),\\
      \phi''(1) &= \frac{4}{p} h'(1) + h''(1) = \frac{(p-1)^2}{p^2} g''(1) + \frac{3(p-1)}{p^2} g'(1) = -2\frac{(p-1)^2}{p^2}  + O(|p-1-\lambda|).
    \end{split}
  \end{equation}
  Note also that $\|\phi'''\|_{L^\infty(\frac{1}{2},1)}\leq L_p$ for some constant $L_p>0$, i.e.\ $\phi''$ is Lipschitz continuous on $[\frac{1}{2},1]$.
  We note that $\phi$ has zeros at $u_m$ and $1$ and so we approximate it by the quadratic polynomial
  \begin{equation*}
    \Phi (u) := B(u-u_m)(1-u) = 4\frac{\phi(\frac{u_{m}+1}{2})}{(1-u_{m})^2}(u-u_{m})(1-u),
  \end{equation*}
  which has the same roots and the factor $B$ is chosen such that $\Phi(\frac{u_{m}+1}{2})=\phi(\frac{u_{m}+1}{2})$. 
  This is very similar to the argument in Proposition~\ref{prop: continuity of period}. 
  By Lagrange interpolation we have for some $\xi_u \in (u_m,1)$ the uniform estimate
  \begin{equation}
      Q(u):= \frac{\phi(u)}{(1-u)(u-u_m)} = -\frac{1}{2} \phi''(\xi_u) = -\frac{1}{2}\phi''(1) - \frac{1}{2}(\phi''(\xi_u)-\phi''(1)) = \frac{(p-1)^2}{p^2} + O(|p-1-\lambda|),
  \end{equation}
  using $|\phi''(\xi_u)-\phi''(1)|\leq \|\phi'''\|_{L^\infty(\frac{1}{2},1)}|1-u_m| = O(|p-1-\lambda|)$.
  Therefore $B=Q(\frac{u_m+1}{2})=\frac{(p-1)^2}{p^2} + O(|p-1-\lambda|)$ is strictly positive.
  We split the integral into
  \begin{equation*}
    \Delta z(\lambda,C^2) = \frac{2(p-1)}{p} \int_{u_{m}}^1 \frac{(p-1)u-\lambda}{|J_{0}|\sqrt{\Phi(u)}}du + \frac{2(p-1)}{p}\int_{u_{m}}^1 \frac{(p-1)u-\lambda}{|J_{0}|} \left(\frac{1}{\sqrt{\phi(u)}} -\frac{1}{\sqrt{\Phi(u)}} \right)du =: I_1+I_2.
  \end{equation*}

 \noindent For the error term $I_2$, since $(p-1)u_m-\lambda = O(|p-1-\lambda|)$ by \eqref{eq: estimate for um} we have
  \begin{equation*}
    \begin{split}
      |I_2| &= \frac{2(p-1)}{p} \left| \int_{u_m}^1 \frac{(p-1)u-\lambda}{|J_0|} \left(\frac{1}{\sqrt{\phi(u)}} -\frac{1}{\sqrt{\Phi(u)}} \right)du \right| \\
      &\leq \frac{2(p-1)}{p} \sup_{u\in (u_m,1)}\left| \frac{(p-1)u-\lambda}{|J_0|} \right| \int_{u_m}^1  \frac{1}{\sqrt{(u-u_m)(1-u)}} \left|\frac{1}{\sqrt{Q(u)}} -\frac{1}{\sqrt{B}} \right|du  \\
      &\leq \frac{2(p-1)}{p} \frac{O(|p-1-\lambda|)}{|J_0|} \sup_{u\in(u_m,1)} \left|\frac{1}{\sqrt{Q(u)}} -\frac{1}{\sqrt{B}} \right| \, \int_{u_m}^1  \frac{1}{\sqrt{(u-u_m)(1-u)}} du  \\
      &\leq  \frac{2\pi(p-1)}{p} \sup_{u\in(u_m,1)} \left|\frac{1}{\sqrt{Q(u)}} -\frac{1}{\sqrt{B}} \right| = \frac{2\pi(p-1)}{p} \sup_{u\in(u_m,1)} \left|\frac{Q(u)-B}{\sqrt{Q(u)B} \left(\sqrt{Q(u)}+\sqrt{B}\right)} \right| = O(|p-1-\lambda|). 
    \end{split}
  \end{equation*}

 \noindent For $I_1$ a closed form exists, i.e.\ by the substitution $u=\frac{1+u_m}{2} + \frac{1-u_{m}}{2}\cos \theta$, $du = \frac{u_{m}-1}{2} \sin \theta d\theta$ for  $\theta \in [0,\pi]$,
  \begin{equation*}
    \begin{split}
      I_1 &= \frac{2(p-1)}{p} \int_{u_{m}}^1 \frac{(p-1)u-\lambda}{|J_{0}|\sqrt{B(u-u_m)(1-u)}}du \\
      &= \frac{2(p-1)}{p} \frac{1}{|J_0|\sqrt{B}} \int_0^\pi (p-1)\left(\frac{1+u_m}{2} + \frac{1-u_m}{2}\cos \theta\right) - \lambda\, d\theta \\
      &=\frac{2(p-1)}{p} \frac{\pi}{|J_0|\sqrt{B}}\left(\frac{(p-1)(1+u_m)}{2} -\lambda\right).
    \end{split}
  \end{equation*}
  
  \noindent Thus by \eqref{eq: estimate for um},
  \begin{equation*}
    \begin{split}
      I_1 &=\frac{2(p-1)}{p} \frac{\pi}{|J_0|\sqrt{B}}\left( \frac{(p-1)\left(2 - \frac{2p}{p-1}\left(\tfrac{p-1-\lambda}{p}-C^2 \right) + O(|p-1-\lambda |^\gamma)\right) }{2}-\lambda \right) \\
        &= \frac{2\pi(p-1)}{p} \frac{pC^2 + O(|p-1-\lambda|^\gamma) } { \sqrt{(p-1-\lambda)^2+p^2C^2}  \sqrt{\frac{(p-1)^2}{p^2}+O(|p-1-\lambda|) } } = 2\pi  \frac{pC^2 + O(|p-1-\lambda|^\gamma) } { \sqrt{(p-1-\lambda)^2+p^2C^2} } + O(|p-1-\lambda|), 
    \end{split}
  \end{equation*}
    and combining the estimates for $I_1$ and $I_2$,
    \begin{equation*}
        \Delta z(\lambda,C^2) = 2 \pi \frac{pC^2 + O(|p-1-\lambda|^\gamma)}{\sqrt{(p-1-\lambda)^2 + p^2C^2}} + O(|p-1-\lambda|).
    \end{equation*}
    As $\lambda \to p-1$,
      \begin{equation*}
          |\Delta z(\lambda,C^2)| \leq 2\pi C + \left( 2\pi p +1 \right) O(|p-1-\lambda|^{\gamma-1}) \leq O(|p-1-\lambda|^{\min(\frac{1}{2},\gamma-1)}) \to 0.
      \end{equation*}
    In addition, if $C^2 \in [(p-1-\lambda)^\sigma, \frac{p-1-\lambda}{p})$,
    \begin{equation*}
        \Delta z(\lambda,C^2) \geq 2\pi p \frac{C^2}{\sqrt{2p^2C^2}} + O(|p-1-\lambda|^{\gamma-1}) \geq \sqrt{2} \pi  (p-1-\lambda)^{\sigma/2} + O(|p-1-\lambda|^{\gamma-1}),
    \end{equation*}
    which is strictly positive, provided that $|p-1-\lambda|<\delta(p,\sigma)$ is sufficiently small. The case $C^2 = \frac{p-1-\lambda}{p}$ follows directly from Lemma~\ref{lemma: positivity on diagonal}.
\end{proof}

Combining Lemma~\ref{lemma: continuity of Delta z inside}, Lemma~\ref{lemma: limit os Delta z at S_1} and Proposition~\ref{prop: limit of Delta z at p-1}, we obtain the following.
\begin{proposition}
  \label{prop: existence of continuous extension}
  The function $\Delta z$ on $S$ has a continuous extension (also denoted by $\Delta z$) to $\overline{S^\rho}$ for any $\rho>0$.
\end{proposition}

\subsection{The set \texorpdfstring{$\Delta z =0$}{Delta z =0}}

We now seek conditions on $C^2$ and $\lambda$ that make $\Delta z(\lambda,C^2)$ vanish. 
First, we characterize the value of $\Delta z$ on the boundary $\partial S^\rho_1$.

\begin{lemma}
  \label{lemma: Delta z on C=0} The following holds.
  \begin{itemize}
    \item If $\lambda \in (\frac{p-1}{2},p-1)$, then $\Delta z(\lambda, 0) < 0$. 
    \item There exists $\lambda^*_p \in (0,\frac{p-1}{2})$ such that $\Delta z(\lambda, 0) \geq 0$ for $\lambda \in [0,\lambda^*_p)$, $\Delta z(\lambda^*_p, 0) = 0$ and $\Delta z(\lambda, 0) \leq 0$ for $\lambda \in (\lambda^*_p,\frac{p-1}{2})$.
  \end{itemize}
\end{lemma}
\color{black}

\begin{proof}
  This is about estimating the integral in Lemma~\ref{lemma: limit os Delta z at S_1}, namely
  \begin{equation*}
    I(\lambda, 0) = \int_{u_m}^1 \frac{(p-1)u-\lambda}{\sqrt{E-\frac{(p-1)^2}{p^2} u^2 + 2\lambda\frac{p-1}{p^2}u}u^{\frac{1}{p}}}du = \frac{p}{p-1} \int_{u_m}^1 \frac{(p-1)u-\lambda}{\sqrt{(1-u)(u-(\frac{2\lambda}{p-1}-1))}u^{\frac{1}{p}}}du,
  \end{equation*}
  where $u_m = \frac{2\lambda}{p-1}-1$ if $\lambda>\frac{p-1}{2}$ and $0$ otherwise, see Lemma~\ref{lemma: continuity of minimal curvature}.

  \underline{Case $\lambda \in (\frac{p-1}{2},p-1)$:} We note that the integrand is negative for $u\in (u_m, \frac{u_m+1}{2})$  and positive on $(\frac{u_m+1}{2},1)$. Since the numerator and the quadratic polynomial in the square root are symmetric, multiplication by $u^{-\frac{1}{p}}$, which is strictly decreasing, implies the result. Since the singularities at the endpoints are of order $\frac{1}{\sqrt{u}}$, the integral is finite.

  \underline{Case $\lambda \in [0,\frac{p-1}{2})$:}
  We compute $\frac{d}{d\lambda}I(\lambda,0)$ for $\lambda \in [0,\frac{p-1}{2})$. Here $u_m=0$ regardless of $\lambda$ and we move directly the derivative inside, that is
  \begin{equation*}
    \begin{split}
      \frac{d}{d\lambda} I(\lambda,0) &= \frac{p}{p-1} \int_{0}^1 \frac{d}{d\lambda} \frac{(p-1)u-\lambda}{\sqrt{(1-u)(u-(\frac{2\lambda}{p-1}-1))}}du \\
      &= \frac{p}{p-1} \int_0^1 \frac{-\sqrt{(1-u)(u-(\frac{2\lambda}{p-1}-1))} +\frac{(1-u)((p-1)u-\lambda)}{\sqrt{(1-u)(u-(\frac{2\lambda}{p-1}-1))}}\frac{1}{p-1} } {(1-u)(u-(\frac{2\lambda}{p-1}-1))u^{\frac{1}{p}}}<0,
    \end{split}
  \end{equation*}
  by the elementary chain of equivalences 
  \begin{equation*}
    \begin{split}
      (1-u)\left(u-\left(\frac{2\lambda}{p-1}-1\right)\right) >\tfrac{1}{p-1}(1-u)((p-1)u-\lambda)
     \Longleftrightarrow u-\frac{2\lambda}{p-1} +1 > u - \frac{\lambda}{p-1} 
      \Longleftrightarrow 1 > \frac{\lambda}{p-1}.
    \end{split}
  \end{equation*}
  Let us also verify that $I(\lambda,0)<0$ for $\lambda=\frac{p-1}{2}-\eps$ with $0<\eps\ll 1$. In that case, we have
  \begin{equation*}
    \begin{split}
      \frac{p-1}{p} I(\lambda, 0) &=  \int_{0}^1 \frac{(p-1)u-\lambda}{\sqrt{(1-u)(u-(\frac{2\lambda}{p-1}-1))}}u^{-\frac{1}{p}}du =I_1+I_2+I_3+I_4 \\
      &= \left( \int_{0}^\eps + \int_{\eps}^{\frac{1}{2}-\frac{\eps}{p-1}} + \int_{\frac{1}{2}-\frac{\eps}{p-1}}^{1-\eps(\frac{2}{p-1}+1)} + \int_{1-\eps(\frac{2}{p-1}+1)}^1 \right) \frac{(p-1)u-\frac{p-1}{2}+\eps}{\sqrt{(1-u)(u+\eps\frac{2}{p-1})}}u^{-\frac{1}{p}}du.
    \end{split}
  \end{equation*}
  We note that the fraction $\frac{(p-1)u-\frac{p-1}{2}+\eps}{\sqrt{(1-u)(u+\eps\frac{2}{p-1})}}$ is odd around its root $u=\frac{\lambda}{p-1}=\frac{1}{2}-\frac{\eps}{p-1}$ and since it is multiplied by a strictly decreasing function $u^{-\frac{1}{p}}$, we have $I_2+I_3<0$. 
  Moreover,
  \begin{equation*}
    \begin{split}
      I_1+I_4 &= \int_{0}^\eps \frac{(p-1)u-\frac{p-1}{2}+\eps}{\sqrt{(1-u)(u+\eps\frac{2}{p-1})}}u^{-\frac{1}{p}}du + \int_{1-\eps(\frac{2}{p-1}+1)}^1  \frac{(p-1)u-\frac{p-1}{2}+\eps}{\sqrt{(1-u)(u+\eps\frac{2}{p-1})}}u^{-\frac{1}{p}}du \\
      &\leq \eps^{-\frac{1}{p}} \int_{0}^\eps \frac{(p-1)u-\frac{p-1}{2}+\eps}{\sqrt{(1-u)(u+\eps\tfrac{2}{p-1})}}du + (1-\eps(\tfrac{2}{p-1}+1))^{-\frac{1}{p}} \int_{1-\eps(\frac{2}{p-1}+1)}^1  \frac{(p-1)u-\frac{p-1}{2}+\eps}{\sqrt{(1-u)(u+\eps\frac{2}{p-1})}}du \\
      &= -(p-1) \left(\eps^{-\frac{1}{p}} \sqrt{(1-u)(u+\eps\tfrac{2}{p-1})} \bigg|_{0}^\eps  + (1-\eps(\tfrac{2}{p-1}+1))^{-\frac{1}{p}}  \sqrt{(1-u)(u+\eps\tfrac{2}{p-1})}\bigg|_{1-\eps(\frac{2}{p-1}+1)}^1   \right) \\
      &= -(p-1) \left(\eps^{-\frac{1}{p}} \left( \sqrt{(1-\eps)\eps(1+\tfrac{2}{p-1})} - \sqrt{\eps \tfrac{2}{p-1}}  \right)  -  (1-\eps(\tfrac{2}{p-1}+1))^{-\frac{1}{p}} \sqrt{(1-\eps)\eps(\tfrac{2}{p-1}+1)} \right) \\
      &\leq -(p-1) \eps^{\frac{1}{2}-\frac{1}{p}} + O(\eps^{\frac{1}{2}}),
    \end{split}
  \end{equation*}
  as long as $\eps<\frac{2}{p+1}$. 
  Thus $I_1+I_4$ is strictly negative as $\eps$ is small enough and hence $\Delta z(\frac{p-1}{2}-\eps,0)<0$.
  Note also that $\Delta z(0,0)>0$, so by continuity and monotonicity of the function $\lambda \mapsto \Delta z(\lambda,0)$, there exists a unique $\lambda^*_p \in (0,\frac{p-1}{2})$ such that $\Delta z(\lambda^*_p,0)=0$. 
\end{proof}

The value of $\Delta z$ on $\partial S^\rho_2$ and $\partial S^\rho_3$ is known by Lemma~\ref{lemma: positivity on diagonal}, for $\partial S^\rho_4$ we have to estimate $\Delta z$ around its possible singularity.
\begin{lemma}
\label{lemma: Delta z is negative around singularity}
    There exists $\rho_p>0$ such that for $(\lambda,C^2)$ with $|\lambda-\frac{p-1}{2}|<\rho_p$ and  $C^2<\rho_p$ we have $\Delta z(\lambda,C^2)<0$.
\end{lemma}
\begin{proof}
    From \eqref{eq: Delta z in u} it suffices to show that
    \begin{equation*}
    \begin{split}
    I =    \int_{u_m}^1 \frac{(p-1)u-\lambda}{u^{\frac{1}{p}}\sqrt{h(u)}}du =\int_{u_m}^1 \frac{(p-1)u-\lambda}{u^{\frac{1}{p}}\sqrt{\frac{(p-1)^2}{p^2} (1-u)(u-(\frac{2\lambda}{p-1}-1)) +C^2(1-u^{\frac{2}{p}-2}) }} du
    \end{split}
    \end{equation*}
    is strictly negative. We note that the integrand is negative on $(u_m,\frac{\lambda}{p-1})$ and positive on $(\frac{\lambda}{p-1},1)$.
    From Lemma~\ref{lemma: estimate around (p-1)/2},
    \begin{equation*}
        u_m < \max(0,\tfrac{2\lambda}{p-1}-1) + C^{\frac{p}{3p-2}} =: b \ll 1,
    \end{equation*}
    thus we split up the integral into 
    \begin{equation*}
        I =I_1+I_2+I_3+I_4= \left(\int_{u_m}^{b} + \int_{b}^{\frac{\lambda}{p-1}} + \int_{\frac{\lambda}{p-1}}^{\frac{2\lambda}{p-1}-b} + \int_{\frac{2\lambda}{p-1}-b}^1 \right) \frac{(p-1)u-\lambda}{u^{\frac{1}{p}}\sqrt{\frac{(p-1)^2}{p^2} (1-u)(u-(\frac{2\lambda}{p-1}-1)) +C^2(1-u^{\frac{2}{p}-2}) }} du.
    \end{equation*}
    
    We start by estimating $I_4$. Since $u\geq \frac{1}{2}$ and $u-(\frac{2\lambda}{p-1}-1) \geq \frac{1}{2}$ (as long as $C^2<\delta_p$ and $|\lambda-\frac{p-1}{2}|<\delta_p$ for $\delta_p$ sufficiently small) and $C^2(1-u^{\frac{2}{p}-2})>-\frac{1}{4}\frac{(p-1)^2}{p^2}(1-u)$ on the interval of integration,
    \begin{equation*}
        \begin{split}
            I_4 \leq \int_{\frac{2\lambda}{p-1}-b}^1 \frac{p-1-\lambda}{u^{\frac{1}{p}} \sqrt{\frac{(p-1)^2}{p^2} \frac{1}{4}(1-u)}}du \leq \frac{ p 2^{1+\frac{1}{p}}}{p-1} \int_{\frac{2\lambda}{p-1}-b}^1 \frac{1}{\sqrt{1-u}}du = \frac{ p 2^{2+\frac{1}{p}}}{p-1} \sqrt{1-\frac{2\lambda}{p-1}+b} \to 0
        \end{split}
    \end{equation*}
    as $(\lambda, C^2) \to (\frac{p-1}{2},0)$.
    
    Consider now the $I_2+I_3$ term. We note that the numerator is linear and symmetric across $u=\frac{\lambda}{p-1}$. Since the lengths of the integration intervals in $I_2$ and $I_3$ are exactly the same, in order to obtain strict negativity of $I_2+I_3$ it suffices to show that the denominator of the integrand in $I_2$ is strictly less than the denominator in $I_3$.
    In other words, we show for every $t \in [0,\frac{\lambda}{p-1}-b]$,
    \begin{equation}
    \label{eq: total estimate for denominator}
    \begin{split}
        &(\tfrac{\lambda}{p-1}-t)^{\frac{2}{p}} h(\tfrac{\lambda}{p-1}-t)
        = (\tfrac{\lambda}{p-1}-t)^{\frac{2}{p}}\left( \frac{(p-1)^2}{p^2}(1-\tfrac{\lambda}{p-1}-t)(1-\tfrac{\lambda}{p-1}+t) +C^2(1-(\tfrac{\lambda}{p-1}-t)^{\frac{2}{p}-2})\right) \\
        & \leq (\tfrac{\lambda}{p-1}+t)^{\frac{2}{p}}\left( \frac{(p-1)^2}{p^2} (1-\tfrac{\lambda}{p-1}-t)(1-\tfrac{\lambda}{p-1}+t) +C^2(1-(\tfrac{\lambda}{p-1}+t)^{\frac{2}{p}-2})\right) 
        = (\tfrac{\lambda}{p-1}+t)^{\frac{2}{p}} h(\tfrac{\lambda}{p-1}+t).
    \end{split}
    \end{equation}

   \noindent We will split up the domain $[0,\frac{\lambda}{p-1}-b]$ into three parts.
First consider $t \ll 1$. By Taylor, we have
\begin{equation*}
\begin{split}
    LHS &= \left[(\tfrac{\lambda}{p-1})^{\frac{2}{p}} - \frac{2}{p}(\tfrac{\lambda}{p-1})^{\frac{2-p}{p}}t + O(t^2) \right] \left[ \frac{(p-1)^2}{p^2}(1-\tfrac{\lambda}{p-1}-t)(1-\tfrac{\lambda}{p-1}+t) \right] \\
    & \qquad+ C^2\left[(\tfrac{\lambda}{p-1})^{\frac{2}{p}} - \frac{2}{p}(\tfrac{\lambda}{p-1})^{\frac{2-p}{p}}t + O(t^2) \right] \left[ 1 -(\tfrac{\lambda}{p-1})^{\frac{2-2p}{p}} + \frac{2-2p}{p}(\tfrac{\lambda}{p-1})^{\frac{2-3p}{p}}t + O(t^2) \right], \\
    RHS &= \left[(\tfrac{\lambda}{p-1})^{\frac{2}{p}} + \frac{2}{p}(\tfrac{\lambda}{p-1})^{\frac{2-p}{p}}t + O(t^2) \right] \left[ \frac{(p-1)^2}{p^2}(1-\tfrac{\lambda}{p-1}-t)(1-\tfrac{\lambda}{p-1}+t) \right] \\
    & \qquad+ C^2\left[(\tfrac{\lambda}{p-1})^{\frac{2}{p}} + \frac{2}{p}(\tfrac{\lambda}{p-1})^{\frac{2-p}{p}}t + O(t^2) \right] \left[ 1 -(\tfrac{\lambda}{p-1})^{\frac{2-2p}{p}} - \frac{2-2p}{p}(\tfrac{\lambda}{p-1})^{\frac{2-3p}{p}}t + O(t^2) \right]. \\
\end{split}
\end{equation*}
Since the terms multiplying $(\tfrac{\lambda}{p-1})^{\frac{2}{p}}$ agree and thereby cancel in the inequality, after dividing by $t$, we obtain
\begin{equation*}
\begin{split}
    &  - \frac{2}{p}\frac{(p-1)^2}{p^2}(\tfrac{\lambda}{p-1})^{\frac{2-p}{p}} (1-\tfrac{\lambda}{p-1})^2 - C^2\left[ \frac{2}{p}(\tfrac{\lambda}{p-1})^{\frac{2-p}{p}} (1- (\tfrac{\lambda}{p-1})^{\frac{2-2p}{p}}) + \frac{2p-2}{p}(\tfrac{\lambda}{p-1})^{\frac{4-3p}{p}} \right] + O(t)\\
    &\quad \leq   \frac{2}{p}\frac{(p-1)^2}{p^2}(\tfrac{\lambda}{p-1})^{\frac{2-p}{p}} (1-\tfrac{\lambda}{p-1})^2   + C^2\left[ \frac{2}{p}(\tfrac{\lambda}{p-1})^{\frac{2-p}{p}} (1- (\tfrac{\lambda}{p-1})^{\frac{2-2p}{p}}) + \frac{2p-2}{p}(\tfrac{\lambda}{p-1})^{\frac{4-3p}{p}} \right] + O(t),\\
\end{split}
\end{equation*}
which holds for all $t < \bar t_p$ where $0<\bar t_p \ll \tfrac{\lambda}{p-1}$ is taken sufficiently small.
For the $C^2$ terms we note that $b\geq C^\frac{p}{3p-2}$, that is $\frac{\lambda}{p-1}-t \geq b\geq C^{\frac{p}{3p-2}}$ and thus
    \begin{equation*}
        \left|C^2(\tfrac{\lambda}{p-1}-t)^{\frac{2}{p}} (1- (\tfrac{\lambda}{p-1}-t)^{\frac{2-2p}{p}}) \right| = O(C^{\min(2,\frac{4p}{3p-2})}), \qquad
        \left| C^2(\tfrac{\lambda}{p-1}+t)^{\frac{2}{p}} (1- (\tfrac{\lambda}{p-1}+t)^{\frac{2-2p}{p}}) \right| = O(C^2).
    \end{equation*}
    Hence it remains to show for $t \in [\bar t_p, \tfrac{\lambda}{p-1}-b]$ the inequality
    \begin{equation}
    \label{eq: inequality with big O}
        (\tfrac{\lambda}{p-1}-t)^{\frac{2}{p}} (1-\tfrac{\lambda}{p-1}-t)(1-\tfrac{\lambda}{p-1}+t) + O(C^{\min(2,\frac{4p}{3p-2})}) \leq (\tfrac{\lambda}{p-1}+t)^{\frac{2}{p}} (1-\tfrac{\lambda}{p-1}-t)(1-\tfrac{\lambda}{p-1}+t).
    \end{equation}
    Consider first the case $t \in (\tfrac{\lambda}{p-1}-b-\bar t_p, \tfrac{\lambda}{p-1}-b]$. 
    Then, adjusting $\bar t_p$ and $\delta_p>|\lambda -\frac{p-1}{2}|$ if necessary,
    \begin{equation*}
        (1-\tfrac{\lambda}{p-1}-t)(1-\tfrac{\lambda}{p-1}+t) [(\tfrac{\lambda}{p-1}+t)^{\frac{2}{p}}-(\tfrac{\lambda}{p-1}-t)^{\frac{2}{p}}] \geq (1-\tfrac{2\lambda}{p-1}+b)(1-b-\bar t_p)  \frac{1}{2} \geq \frac{1}{4} C^{\frac{p}{3p-2}},
    \end{equation*}
    using the definition $b=\max(0,\frac{2\lambda}{p-1}-1) + C^{\frac{p}{3p-2}}$.
    Thus \eqref{eq: inequality with big O} holds if $C^2<\delta_p$ for some $\delta_p>0$.
    Lastly, for $t\in[\bar t_p, \tfrac{\lambda}{p-1}-b-\bar t_p]$ note that $(1-\tfrac{\lambda}{p-1}-t)$ and $(\tfrac{\lambda}{p-1}+t)^{\frac{2}{p}}-(\tfrac{\lambda}{p-1}-t)^{\frac{2}{p}}$ are bounded below by a constant $c_p>0$ and hence for $C^2<\delta_p$ sufficiently small,
    \begin{equation}
    \begin{split}
    \label{eq: estimate for denominator in special interval}
        &(\tfrac{\lambda}{p-1}-t)^{\frac{2}{p}} h(\tfrac{\lambda}{p-1}-t) = (\tfrac{\lambda}{p-1}-t)^{\frac{2}{p}} \frac{(p-1)^2}{p^2} (1-\tfrac{\lambda}{p-1}-t)(1-\tfrac{\lambda}{p-1}+t) + O(C^{\min(2,\frac{4p}{3p-2})}) \\
        &\qquad \leq (\tfrac{\lambda}{p-1}+t)^{\frac{2}{p}}\frac{(p-1)^2}{p^2} (1-\tfrac{\lambda}{p-1}-t)(1-\tfrac{\lambda}{p-1}+t) + O(C^2) - \frac{c_p}{2} = (\tfrac{\lambda}{p-1}+t)^{\frac{2}{p}} h(\tfrac{\lambda}{p-1}+t) - \frac{c_p}{2}.
    \end{split}
    \end{equation}

\noindent Therefore, by the change of variables $u=\frac{\lambda}{p-1}-t$ (for $I_2$) and $u=\frac{\lambda}{p-1}+t$ (for $I_3$ respectively), 
\begin{equation*}
\begin{split}
    I_2+I_3 &=  \left( \int_0^{\bar t_p} + \int_{\bar t_p}^{\frac{\lambda}{p-1}-b-\bar t_p} +\int_{\frac{\lambda}{p-1}-b-\bar t_p}^{\frac{\lambda}{p-1}-b} \right)\frac{-t}{\sqrt{(\frac{\lambda}{p-1}-t)^{\frac{2}{p}}h(\frac{\lambda}{p-1}-t)}} + \frac{t}{\sqrt{(\frac{\lambda}{p-1}+t)^{\frac{2}{p}}h(\frac{\lambda}{p-1}+t)}}dt \\
    &\leq \int_{\bar t_p}^{\frac{\lambda}{p-1}-b-\bar t_p}  \frac{-t}{\sqrt{(\frac{\lambda}{p-1}+t)^{\frac{2}{p}}h(\frac{\lambda}{p-1}+t) - \frac{c_p}{2}}} + \frac{t}{\sqrt{(\frac{\lambda}{p-1}+t)^{\frac{2}{p}}h(\frac{\lambda}{p-1}+t)}}dt\\
    &= \int_{\bar t_p}^{\frac{\lambda}{p-1}-b-\bar t_p} \frac{t \left(\sqrt{(\frac{\lambda}{p-1}+t)^{\frac{2}{p}}h(\frac{\lambda}{p-1}+t) - \frac{c_p}{2}} - \sqrt{(\frac{\lambda}{p-1}+t)^{\frac{2}{p}}h(\frac{\lambda}{p-1}+t) } \right)}{\sqrt{(\frac{\lambda}{p-1}+t)^{\frac{2}{p}}h(\frac{\lambda}{p-1}+t)} \sqrt{(\frac{\lambda}{p-1}+t)^{\frac{2}{p}}h(\frac{\lambda}{p-1}+t) - \frac{c_p}{2}}} dt \\
    &\leq \int_{\bar t_p}^{\frac{\lambda}{p-1}-b-\bar t_p} \frac{-t c_p}{4 (\frac{\lambda}{p-1}+t)^{\frac{2}{p}}h(\frac{\lambda}{p-1}+t) \sqrt{(\frac{\lambda}{p-1}+t)^{\frac{2}{p}}h(\frac{\lambda}{p-1}+t) - \frac{c_p}{2}}} dt \leq \int_{\bar t_p}^{\frac{\lambda}{p-1}-b-\bar t_p}   \frac{-t c_p}{M_p} dt = -c'_p<0,
\end{split}
\end{equation*}
using \eqref{eq: total estimate for denominator}, \eqref{eq: estimate for denominator in special interval} in the first inequality and the fact that $t\mapsto h(\frac{\lambda}{p-1}+t)$ (and thereby the whole denominator) is uniformly bounded above (independent of $\lambda$, $C^2$) in the last.
Since $I_1\leq 0$ and $I_4 \to 0$, we have $I =I_1+I_2+I_3+I_4 \leq -\frac{1}{2}c'_p$ whenever $|\lambda-\frac{p-1}{2}|, C^2 <\rho_p$ sufficiently small, finishing the proof.
\end{proof}

Let $\lambda^*_p$ be the constant from Lemma~\ref{lemma: Delta z on C=0} and take $\rho \in (0, \frac{1}{2} \left( \frac{p-1}{2} - \lambda^*_p \right))$ sufficiently small. Let $\Delta z: \overline{S^\rho} \to \R$ be the continuous extension of $\Delta z|_{S^\rho}:S^\rho \to \R$ from Proposition~\ref{prop: existence of continuous extension}.
Then, the following properties hold:
\begin{itemize}
  \item $\Delta z(p-1,0)=\Delta z(\lambda_p^*,0)=0$, by Proposition~\ref{prop: limit of Delta z at p-1} and Lemma~\ref{lemma: Delta z on C=0}.
  \item $\Delta z>0$ on $\partial S^\rho_3 \cup \partial S^\rho_2 \cup \left([0,\lambda^*_p) \times \{0\} \right) $, by Lemma~\ref{lemma: positivity on diagonal} and Lemma~\ref{lemma: Delta z on C=0}.
  \item $\Delta z<0$ on $\left((\lambda_p^*,\frac{p-1}{2}-\rho] \cup [\frac{p-1}{2}+\rho, p-1) \right)  \times \{0\}$, by Lemma~\ref{lemma: Delta z on C=0}.
  \item $\Delta z<0$ on $\partial S^\rho_4$ by Lemma~\ref{lemma: Delta z is negative around singularity}, upon taking $\rho=\rho(p)$ smaller if necessary.
\end{itemize}

\noindent This gives now an abstract description of the zero level set.

\begin{proposition}
  \label{prop: Delta z=0 is continuous curve}
  The zero level set of $\Delta z:\overline{S^\rho} \to \R$, denoted by $\{\Delta z=0\}$, contains a connected component $\tilde \Gamma$ connecting $(\lambda_p^*,0)$ and $(p-1,0)$ in $\overline{S^\rho}$, where only the points $(\lambda_p^*,0),(p-1,0) \in \tilde \Gamma$ lie in $\partial S^\rho$.
\end{proposition}

\begin{proof}
  The function $\Delta z$ is continuous up to the boundary on $S^\rho$. Moreover, $\Delta z $ is strictly negative on
  \begin{equation*}
    \partial S^\rho_{neg} := \left((\lambda_p^*, \tfrac{p-1}{2}-\rho] \cup [\tfrac{p-1}{2}+\rho, p-1) \right)  \times \{0\} \cup \partial S^\rho_4,
  \end{equation*} 
  strictly positive on 
  \begin{equation*}
    \partial S^\rho_{pos}:= (\partial S^\rho_2 \setminus\{P_2\}) \cup \partial S^\rho_3 \cup \left([0,\lambda^*_p) \times \{0\} \right),
  \end{equation*}
   and vanishes at $P_1=(\lambda_p^*,0)$ and $P_2=(p-1,0)$. 
   We now show that there exists a connected component $\tilde \Gamma$ of $\{\Delta z=0\}$ such that $P_1, P_2 \in \tilde \Gamma$.

   First, consider a homeomorphism $\varphi: \overline{S^\rho} \to [0,1]^2$ such that $\varphi(P_1) = (\frac{1}{2},0)$, $\varphi(P_2)=(\frac{1}{2},1)$, $\varphi(\partial S^\rho_{neg}) = \partial [0,1]^2 \cap \{x<\frac{1}{2}\}$, $\varphi(\partial S^\rho_{pos}) = \partial [0,1]^2 \cap \{x>\frac{1}{2}\}$ and set $F= \Delta z \circ \varphi^{-1}: [0,1]^2 \to \R$. 
   Then $F$ is continuous and $F|_{\{0\} \times [0,1]}<0$ as well as $F|_{\{1\} \times [0,1]}>0$.
   Thus it suffices to show that there exists a connected component $\Sigma$ of the compact set $\{F=0\}$ containing $\varphi(P_1) = (\frac{1}{2},0)$ and $\varphi(P_2)=(\frac{1}{2},1)$, the original $\tilde \Gamma$ is then given by $\varphi^{-1}(\Sigma)$.
   
   This follows from the ``crossing lemma'' \cite[Lemma~2.9]{crossing_lemma_paper}, (see also \cite[Lemma~1.1.17]{pireddu2009fixedpointschaoticdynamics_phdthesis}), once we verify that every continuous map $\sigma:[0,1] \to [0,1]^2$ with $\sigma(0) \in \{0\} \times [0,1]$ and $\sigma(1) \in \{1\} \times [0,1]$ satisfies $\{F=0\} \cap \sigma([0,1]) \neq \varnothing$.
   This follows from $F(\sigma(0))<0$ and $F(\sigma(1))>0$ as well as continuity of $F$ and $\sigma$ and the intermediate value theorem, i.e.\ there exists $\xi \in (0,1)$ such that $F(\sigma(\xi))=0$ and the intersection is non-empty. 
   We thereby obtain a connected component $\Sigma \subset \{F=0\}$, such that $\Sigma \cap ([0,1]\times \{0\}) \neq \varnothing$ and $\Sigma \cap ([0,1]\times \{1\}) \neq \varnothing$. Since $\Sigma \cap ([0,1]\times \{0\}) \subset \{F=0\} \cap ([0,1]\times \{0\}) =P_1$ it follows that $P_1 \in \Sigma$, and analogously $P_2 \in \Sigma$.
\end{proof}

We also estimate the global shape of $\{\Delta z=0\}$ and so $\tilde \Gamma$ in $S^0$.
\begin{proposition}
  \label{prop: global growth of Delta z=0}
  The curve $r_\mathcal{C}=0$ intersects the set $\{\Delta z=0\}$ only at the right endpoint $(p-1,0)$.
  Moreover, letting $A$ be the domain whose boundary is given by the piecewise smooth curve $([0,p-1]\times \{0\}) \cup \{r_{\mathcal{C}}=0\}$, its closure $\overline A$ in $\overline{S^\rho}$ contains itself a connected compact subset $\Gamma$ of $\{\Delta z=0\}$ containing $(\lambda_p^*,0)$ and $(p-1,0)$, which furthermore intersects the curve $r_\mathcal{C}=0$ only at the point $(p-1,0)$.
\end{proposition}
\begin{proof}
  Suppose that $(\lambda,C^2) \in S^0$ lies on $r_\mathcal{C}=0$. Then $k_m = \left(\frac{p^2C^2}{(p-1-\lambda)^2 + p^2C^2} \right)^{\frac{1}{2p-2}}$ and thus by Lemma~\ref{lemma: condition for rm=0}
  \begin{equation*}
    (p-1)k(s)^p-\lambda \geq (p-1)k_m^p-\lambda = (p-1)\left(\frac{p^2C^2}{(p-1-\lambda)^2 + p^2C^2} \right)^{\frac{p}{2p-2}} - \lambda =0.
  \end{equation*}
  Since $k(s)$ is continuous, but not constant, the inequality is strict for a set of positive measure and it follows that $\Delta z(\lambda,C^2) = \int_0^1 \frac{(p-1)k(s)^p-\lambda}{|J_0|} >0$. 

Let $\Gamma = \overline A \cap \tilde \Gamma$, which is closed and clearly $(\lambda_p^*,0),(p-1,0) \in \Gamma$. 
It remains to show that $\Gamma$ is connected. 
Suppose that $\Gamma$ is not connected, i.e.\ there exist relatively disjoint closed sets $U$ and $V$ such that $U \cup V = \Gamma$. 
We may assume that $(p-1,0) \in V$.
Then $U$ and $W:=V \cup (\tilde \Gamma \cap \overline{A^c})$ (where the complement and closure are taken in $\overline{S^\rho}$) are relatively closed in $\tilde \Gamma$ and 
\begin{equation*}
   U \cup W = U \cup (V \cup (\tilde \Gamma \cap \overline{A^c})) = (U \cup V) \cup (\tilde \Gamma \cap \overline{A^c}) = \Gamma \cup (\tilde \Gamma \cap \overline{A^c}) = (\tilde \Gamma \cap \overline A) \cup (\tilde \Gamma \cap \overline{A^c}) = \tilde \Gamma,
\end{equation*}
and unless $U\cap W \neq \varnothing$, $U$ and $W$ form a separation for $\tilde \Gamma$, which contradicts the connectivity of $\tilde \Gamma$.
Indeed, suppose there exists $x \in U\cap W$. Since $U \cap V =\varnothing$, necessarily $x \in U \cap (\tilde \Gamma \cap \overline{A^c}) = \tilde \Gamma \cap (\overline A \cap \overline{A^c})$.
However, as $\overline A \cap \overline{A^c} = \{r_{\mathcal{C}}=0\}$ and $\tilde \Gamma \subset \{\Delta z=0\}$, by the first part of the proposition, $x=(p-1,0)$. This in turn contradicts $x \in U$, that is $U\cap W = \varnothing$.
Therefore $\Gamma$ is connected.
\end{proof}

From now on we work only with the connected compact subset $\Gamma$ of $\{\Delta z=0\}$, lying below $r_{\mathcal{C}}=0$ (Proposition~\ref{prop: global growth of Delta z=0}).
We furthermore characterize the asymptotic behavior of $\Gamma$ near the $(p-1,0)$ endpoint, which is of importance in the next section, see also Figure~\ref{fig: deltaZ asypmtotic}.

\begin{figure}[ht]
  \includegraphics*[width=0.7\textwidth]{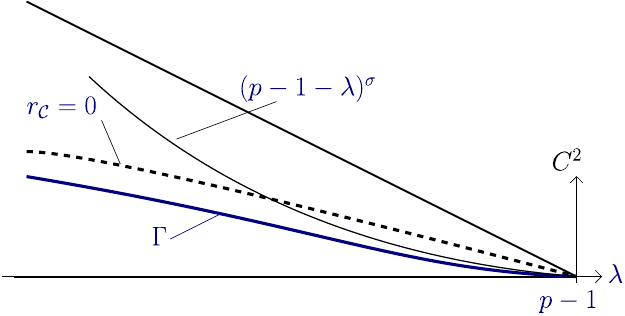}
  \caption{The asymptotic behavior of $\Gamma$}
  \label{fig: deltaZ asypmtotic}
\end{figure}

\begin{corollary}
  \label{cor: asymptotics at p-1 of Gamma}
  Fix $\sigma \in (1,2)$. Then there exists $\delta=\delta(p,\sigma)$ such that for all $(\lambda,C^2) \in \Gamma \subset \{\Delta z =0\}$ with $\lambda \in (p-1-\delta, p-1)$, we have $C^2 < (p-1-\lambda)^\sigma$.
\end{corollary}
\begin{proof}
  By Proposition~\ref{prop: limit of Delta z at p-1}, there exists $\delta=\delta(p,\sigma)$ such that $\Delta z(\lambda,C^2)>0$ for all $C^2 \in [(p-1-\lambda)^\sigma, \frac{p-1-\lambda}{p}]$ and $|p-1-\lambda| < \delta$, that is for $(\lambda,C^2) \in \Gamma \subset \{\Delta z=0\}$ we have $C^2 < (p-1-\lambda)^\sigma$.
\end{proof}

\section{The function \texorpdfstring{$\Delta \theta$}{Delta theta}}
\label{section: Delta theta}
Recall that
\begin{equation}
\label{eq: formula for Delta theta}
\begin{split}
  \Delta \theta(\lambda, C^2) &= 2\int_{-P}^0 \theta_s = 2\int_{-P}^0 \frac{C |J_0| ((p-1)k(s)^p - \lambda)}{k(s)^{2p-2}|J_0|^2-p^2C^2} ds  = 2\int_{-P}^0 \frac{C \sqrt{(p-1-\lambda)^2 + p^2 C^2} ((p-1)k(s)^p - \lambda)}{k(s)^{2p-2}((p-1-\lambda)^2 + p^2 C^2)-p^2C^2} ds.
\end{split}
\end{equation}

We are now interested in the value of $\Delta \theta$ along  $\Gamma \subset \{\Delta z = 0\}$ in the $\lambda$-$C^2$ plane.
To show existence of some non-planar closed elasticae it suffices to show that $\Delta \theta$ is continuous (at least along $\Gamma$) and that it takes at least two different values. 
(Then by continuity it attains all intermediate values, in particular, all intermediate rational multiples of $2\pi$.) 

We should also note that strictly speaking $\Delta \theta$ is only defined for $C^2\neq 0$, the value of $\Delta \theta$ for $(\lambda,0)$ is to be understood as the limit from inside $S$.

\begin{lemma}
  \label{lemma: Delta theta denominator stays positive}
  Let $(\lambda,C^2)\in S^0$ with $\Delta z(\lambda,C^2)=0$. Then 
  \begin{equation*}
    k(s)^{2p-2}((p-1-\lambda)^2 + p^2 C^2)-p^2C^2\geq \eps_{\lambda,C^2}>0.
  \end{equation*}
\end{lemma}

\begin{proof}
  It suffices to show that strict inequality holds for the minimum $k_m$, that is we want to show
  \begin{equation*}
    k_m^{2p-2} > \frac{p^2 C^2}{(p-1-\lambda)^2 + p^2 C^2}.
  \end{equation*}
  Assume that $k_m^{2p-2} \leq \frac{p^2 C^2}{(p-1-\lambda)^2 + p^2 C^2}$. By Corollary~\ref{cor: global lower bound for km} we have equality, and $(\lambda,C^2)$ lies also in $\{r_\mathcal{C}=0\}$, which contradicts Proposition~\ref{prop: global growth of Delta z=0}. 
\end{proof}

\begin{proposition}
  \label{prop: continuiy of Delta theta}
  The function $\Delta \theta: S^0 \to \R$ is continuous along $\Gamma \cap S^0$. 
\end{proposition}
\begin{proof}
  Let $\Gamma \cap S^0 \ni (\lambda_n,C_n^2)\to(\lambda,C^2)\in \Gamma \cap S^0$ and denote their corresponding periods by $P_n$ and $P$ (which are uniformly bounded as $\Gamma$ stays away from the possible singular point $(\frac{p-1}{2},0)$) and curvatures by $k_n$ and $k$. Then
  \begin{equation*}
    \begin{split}
     \frac{|\Delta \theta(\lambda_n,C_n^2) - \Delta \theta(\lambda,C^2)|}{2} &= \left|\int_{-P_n}^0 (\theta_n)_s \, ds - \int_{-P}^0 \theta_s \, ds \right|  = \left| \int_{-P}^0(\theta_n)_s  - \theta_s \, ds + \int_{-P_n}^{-P}(\theta_n)_s \, ds\right| \\
      &\leq P\|(\theta_n)_s-\theta_s\|_{L^\infty(0,P)} + |P_n-P|\|(\theta_n)_s\|_{L^\infty(P,P_n)}.
    \end{split}
  \end{equation*}
  We also have
  \begin{equation*}
    (\theta_n)_s = \frac{C_n |J_{0,n}| ((p-1)k_n(s)^p - \lambda_n)}{k_n(s)^{2p-2}|J_{0,n}|^2-p^2C_n^2},
  \end{equation*}
  where the denominator is bounded below by some $\eps_{\lambda,C^2}>0$ for all $n$ sufficiently large (i.e.\ in a neighborhood of $(\lambda,C^2)$), thanks to Lemma~\ref{lemma: Delta theta denominator stays positive} and Lemma~\ref{lemma: continuity of minimal curvature}.
  Since $k_n\to k$ in $L^\infty(0,P)$ by Remark~\ref{rmk: local uniform convergence of w_n}, we have $(\theta_n)_s \to \theta_s$ and the first term converges to zero. The second term converges also to zero by Proposition~\ref{prop: continuity of period}.
\end{proof}

\begin{remark}
  We remark that $\Delta \theta$ is not continuous in $S^0$ and does not have a well-defined limit at $(p-1,0)$. In fact, $\Delta \theta$ jumps by $2\pi$ when crossing the curve $r_\mathcal{C}=0$, since then the $p$-elastica transverses the cylindrical axis, see also \cite[Figure~2]{langer-singer_classification}.
\end{remark}

We now turn towards the evaluation of the limits of $\Delta \theta$ at $(p-1,0)$ and $(\lambda^*_p,0)$ along $\Gamma$. 
In particular, in the former case, the asymptotic behavior from Corollary~\ref{cor: asymptotics at p-1 of Gamma} comes into play.

\begin{proposition}
  \label{prop: Delta theta at right endpoint}
  Let $(\lambda_n, C^2_n) \to (p-1,0)$ such that $C_n^2 \leq (p-1-\lambda_n)^\sigma$ with $\sigma \in (1,2)$. Then $\lim_{n} \Delta \theta (\lambda_n , C_n^2) =0$.
\end{proposition}
\begin{proof}
  We show that 
  \begin{equation*}
    \frac{|\Delta \theta(\lambda_n,C_n^2)|}{2} = \bigg| \int_{-P_n}^0 \frac{C_n \sqrt{(p-1-\lambda_n)^2 + p^2 C_n^2} \left((p-1)k_{n}(s)^p - \lambda_n \right)}  {k_{n}(s)^{2p-2}((p-1-\lambda_n)^2 + p^2 C_n^2)-p^2C_n^2} ds \bigg| \to 0.
  \end{equation*}
 The denominator is nonnegative by Lemma~\ref{lemma: Delta theta denominator stays positive}. Moreover, from Corollary~\ref{cor: approx of A from estimate at p-1}, taking $\gamma=\sigma$, we have for $n$ large
  \begin{equation}
    \label{eq: quantitative lower bound for denom theta_s}
    \begin{split}
      D_n &:=k_{n}(s)^{2p-2}((p-1-\lambda_n)^2 + p^2 C_n^2)-p^2C_n^2  \\ 
      &\geq k_{m,n}^{2p-2}((p-1-\lambda_n)^2 + p^2 C_n^2)-p^2C_n^2 \\
      &= \left(1-\frac{2}{p}(p-1-\lambda_n) + 2C_n^2 + O(|p-1-\lambda_n|^\sigma)\right)^{2} [(p-1-\lambda_n)^2 + p^2 C_n^2]-p^2C_n^2 \\
      &= \left(1-\frac{4}{p}(p-1-\lambda_n) + 4C_n^2 + O(|p-1-\lambda_n|^\sigma)\right) [(p-1-\lambda_n)^2 + p^2 C_n^2]-p^2C_n^2 \\
      &= (p-1-\lambda_n)^2 + O(|p-1-\lambda_n|^{1+\sigma})
    \end{split}
  \end{equation} 
 as well as
 \begin{equation*}
  \begin{split}
    (p-1)k_{m,n}^p - \lambda_n  &= (p-1)\left( 1-\frac{2}{p-1}(p-1-\lambda_n) + \frac{2p}{p-1}C_n^2 + O(|p-1-\lambda_n|^\sigma)\right) - \lambda_n \\
    &= -(p-1-\lambda_n) + O(|p-1-\lambda_n|^\sigma).
  \end{split}
 \end{equation*}
 Taking the absolute value inside the integral, we have
 \begin{equation*}
  |(p-1) k_n(s)^p - \lambda_n|  \leq \max(|(p-1)k_{m,n}^p - \lambda_n|, p-1-\lambda_n) \leq 2(p-1-\lambda_n)
 \end{equation*}
 and hence 
  \begin{equation*}
    |\Delta \theta| \leq 2 P_n \frac{C_n \sqrt{(p-1-\lambda_n)^2 + p^2 C_n^2} (p-1 - \lambda_n)}  {(p-1-\lambda)^2 + O(|p-1-\lambda|^{1+\sigma})}.
  \end{equation*}
  Since $C_n\leq (p-1-\lambda_n)^{\sigma/2}$, for $n$ large, the numerator is bounded as
  \begin{equation}
  \label{eq: quantitative upper bound for numerator}
    N_n \leq (p-1-\lambda_n)^{\sigma/2} \sqrt{(p-1-\lambda_n)^2 + p^2 (p-1-\lambda_n)^{\sigma}} (p-1 - \lambda_n) \leq O (|p-1-\lambda_n|^{1+\sigma}).
  \end{equation}
  Combining the estimates on numerator \eqref{eq: quantitative upper bound for numerator} and denominator \eqref{eq: quantitative lower bound for denom theta_s} with Proposition~\ref{prop: continuity of period} gives
  \begin{equation*}
    \frac{|\Delta \theta(\lambda_n,C_n^2)|}{2} \leq 2 P_{n} \frac{N_n}{D_n} \leq 4 P_n  \frac{O(|p-1-\lambda_n|^{1+\sigma})}{(p-1-\lambda_n)^{2}} = O(|p-1-\lambda_n|^{\sigma-1}) \to 0 
  \end{equation*}
  as $n\to \infty$, i.e.\ $\lambda_n \to p-1$. This finishes the proof.
\end{proof}

We now calculate $\Delta \theta$ at the second endpoint around $(\lambda^*_p,0)$, where the limiting direction is irrelevant. However, note that now the minimal curvature $k_m$ approaches zero and has to be carefully estimated.

\begin{proposition}
  \label{prop: theta limit at figure 8}
  Let $(\lambda,C^2) \to (\lambda^*_p,0)$. Then $ \Delta \theta (\lambda, C^2) \to -\pi$.
\end{proposition}

\begin{proof}
  Perform again the change of variables $k=k(s)$, $dk = k'(s)ds$  (see Remark~\ref{rmk: k prime}) to obtain 
  \begin{equation*}
    \frac{\Delta \theta}{2} = C |J_0| \int_{k_m}^1 \frac{(p-1)k^p - \lambda}{k^{2p-2}|J_0|^2 - p^2 C^2}\frac{1}{\frac{1}{p-1}\sqrt{E k^{4-2p} -\frac{(p-1)^2}{p^2} k^{4} + 2\lambda \frac{p-1}{p^2} k^{4-p} - C^2 k^{6-4p}  }} dk. 
  \end{equation*}
  Substitute $v=k^{2p-2}$, $dv = (2p-2)k^{2p-3}dk = (2p-2)v^{\frac{2p-3}{2p-2}}dk$ to get for $v_m=k_m^{2p-2}$ (notice the $2$ and the factor $p-1$ cancel),
  \begin{equation*}
    \begin{split}
    \Delta \theta &= C |J_0| \int_{v_m}^1 \frac{(p-1)v^{\frac{p}{2p-2}} - \lambda}{v|J_0|^2 - p^2 C^2}\frac{1}{\sqrt{E v^\frac{4-2p}{2p-2} -\frac{(p-1)^2}{p^2} v^\frac{4}{2p-2} + 2\lambda \frac{p-1}{p^2} v^\frac{4-p}{2p-2} - C^2 v^\frac{6-4p}{2p-2}}} \frac{1}{v^\frac{2p-3} {2p-2}} dv \\
    &= C |J_0| \int_{v_m}^1 \frac{(p-1)v^{\frac{p}{2p-2}} - \lambda}{v|J_0|^2 - p^2 C^2}\frac{1}{\sqrt{E v -\frac{(p-1)^2}{p^2} v^\frac{4p-2}{2p-2} + 2\lambda \frac{p-1}{p^2} v^\frac{3p-2}{2p-2} - C^2 }} dv.\\
    \end{split}
  \end{equation*}
  First, note that since $\lambda^*_p<\frac{p-1}{2}$, there exist $\delta_p$ such that $\lambda < \frac{p-1}{2}-\delta_p$.  
  Hence, from Corollary~\ref{cor: estimates for km} and Remark~\ref{rmk:first order Taylor},
  \begin{equation*}
  \begin{split}
    v_m & = \left(\left(\frac{(p-1)(p-1-2\lambda)}{p^2} \right)^{\frac{1}{2-2p}}C^{\frac{2}{2p-2}} + O(C^{\min(\frac{2p+2}{2p-2}, \frac{4p-2}{2p-2})}) \right)^{2p-2} = \frac{p^2}{(p-1)(p-1-2\lambda)} C^2 + O(C^{\min(\frac{3p-2}{p-1},4)}),
  \end{split}
  \end{equation*}
  thus for $C^2$ sufficiently small, $C^\frac{1}{2}\gg v_m$, and 
  \begin{equation}
    \begin{split}
      \label{eq: approx for km denominator}
      v_m |J_0|^2 - p^2C^2   &\geq \left(\frac{p^2}{(p-1)(p-1-2\lambda)} C^2 + O(C^{\min(\frac{3p-2}{p-1},4)}) \right) ((p-1-\lambda)^2+p^2C^2) -p^2C^2 \\
      & = \frac{p^2 \lambda^2}{(p-1)(p-1-2\lambda)}C^2 + O(C^{\min(\frac{3p-2}{p-1},4)}).
    \end{split}
  \end{equation}
  To conclude, the denominator is strictly positive for small $C^2$, a quantitative version of Lemma~\ref{lemma: Delta theta denominator stays positive}.
  Let
  \begin{equation*}
  \begin{split}
      \psi(v)&:= vg(v^{\frac{1}{2}})= E v -\frac{(p-1)^2}{p^2} v^\frac{4p-2}{2p-2} + 2\lambda \frac{p-1}{p^2} v^\frac{3p-2}{2p-2} - C^2, \\
      \psi'(v) &= E-\frac{4p-2}{2p-2} \frac{(p-1)^2}{p^2} v^{\frac{2p}{2p-2}} + 2\lambda \frac{3p-2}{2p-2} \frac{p-1}{p^2} v^{\frac{p}{2p-2}}.
  \end{split}
  \end{equation*}
  We claim that $\psi$ is also unimodal, i.e.\ it increases until its maximum and then decreases. To see this, consider $A(v) = v^{\frac{-2p}{2p-2}}\psi'(v)$ which has derivative
  \begin{equation*}
      A'(v) = \frac{-2p}{2p-2}Ev^{\frac{-2p}{2p-2}-1} - 2\lambda\frac{3p-2}{2p-2} \frac{p-1}{p^2} \frac{p}{2p-2}v^{\frac{-p}{2p-2}-1} <0,
  \end{equation*}
  since $E = \frac{(p-1)^2}{p^2}-2\lambda\frac{p-1}{p^2}+C^2$ is positive for $\lambda$ close to $\lambda_p^*<\frac{p-1}{2}$.
  Therefore $A$ and also $\psi'$ have only one zero in $(0,1)$ with sign change positive-negative, proving the claim.
  We compute
  \begin{equation*}
    \begin{split}
      \psi(C^{\frac{1}{2}}) &= E C^{\frac{1}{2}} -\frac{(p-1)^2}{p^2}C^{\frac{1}{2}\frac{4p-2}{2p-2}} + 2\lambda \frac{(p-1)}{p^2}C^{\frac{1}{2}\frac{3p-2}{2p-2}} -C^2 = \frac{(p-1)(p-1-2\lambda)}{p^2} C^{\frac{1}{2}} + O(C^{\frac{1}{2} \frac{3p-2}{2p-2}}), \\
      \psi(1-C^{\frac{1}{2}}) &= E(1-C^{\frac{1}{2}}) - \frac{(p-1)^2}{p^2}\left(1-C^{\frac{1}{2}}\right)^{\frac{4p-2}{2p-2}} + 2\lambda\frac{p-1}{p^2} \left(1-C^{\frac{1}{2}}\right)^{\frac{3p-2}{2p-2}} -C^2 \\
      &= E(1-C^{\frac{1}{2}})- \frac{(p-1)^2}{p^2}\left(1- \frac{4p-2}{2p-2} C^{\frac{1}{2}} + O(C)\right) + 2\lambda\frac{p-1}{p^2} \left(1-\frac{3p-2}{2p-2}C^{\frac{1}{2}} + O(C) \right) -C^2 \\
      &= C^{\frac{1}{2}} \left(-\frac{(p-1)(p-1-2\lambda)}{p^2} + \frac{(p-1)^2}{p^2} \frac{4p-2}{2p-2} - 2\lambda \frac{p-1}{p^2}\frac{3p-2}{2p-2}  \right) +O(C) \\
      &= \frac{p-1-\lambda}{p} C^{\frac{1}{2}} + O(C) ,  
    \end{split}
  \end{equation*}
  and since $\psi$ is unimodal,
  \begin{equation}
  \label{eq: estimate for psi in middle}
      \inf_{v \in (C^{\frac{1}{2}}, 1-C^\frac{1}{2})} \psi(v) = \min (\psi(C^{\frac{1}{2}}), \psi(1-C^{\frac{1}{2}}) ) = \min (\tfrac{p-1-\lambda}{p}, \tfrac{(p-1)(p-1-2\lambda)}{p^2} )C^{\frac{1}{2}} +  O(C^{\min(1,\frac{3p-2}{4p-4})})\geq c_p C^{\frac{1}{2}},
  \end{equation}
  for some constant $c_p>0$. Moreover by similar computations,
  \begin{equation}
    \label{eq: estimates for psiprime}
    \begin{split}    
     &\text{in } (v_m,C^\frac{1}{2}):\quad \psi'(v) = E -\frac{4p-2}{2p-2} \frac{(p-1)^2}{p^2}v^{\frac{2p}{2p-2}} + 2\lambda\frac{3p-2}{2p-2} \frac{p-1}{p^2} v^{\frac{p}{2p-2}} = \frac{(p-1)(p-1-2\lambda)}{p^2} + O(C^{\frac{p}{4p-4}}),\\
     &\text{in } (1-C^\frac{1}{2},1):\quad \psi'(v) =  E -\frac{4p-2}{2p-2} \frac{(p-1)^2}{p^2}v^{\frac{2p}{2p-2}}  + 2\lambda\frac{3p-2}{2p-2} \frac{p-1}{p^2}v^{\frac{p}{2p-2}} = \frac{\lambda-(p-1)}{p} + O(C^{\frac{p}{4p-4}}) <0.
    \end{split}
  \end{equation}
  We now split up the integral to deal with the singular endpoints separately,
  \begin{equation*}
    \begin{split}
      \Delta \theta &= I_1+I_2+I_3 = C |J_0| \left( \int_{v_m}^{C^{\frac{1}{2}}} + \int_{C^{\frac{1}{2}}}^{1-C^{\frac{1}{2}}} + \int_{1-C^{\frac{1}{2}}}^1 \right) \frac{(p-1)v^{\frac{p}{2p-2}} - \lambda}{v|J_0|^2 - p^2 C^2}\frac{1}{\sqrt{\psi(v)}} dv.
    \end{split}
    \end{equation*}
    
    \underline{First term:}
  By the weighted mean value theorem (Theorem~\ref{thm: weighted mvt}) there exists $a\in [v_m,C^{\frac{1}{2}}]$ such that
  \begin{equation*}
    I_1 = C |J_0| ((p-1)a^{\frac{p}{2p-2}} - \lambda) \int_{v_m}^{C^{\frac{1}{2}}}  \frac{1}{v|J_0|^2 - p^2 C^2}\frac{1}{\sqrt{\psi(v)}} dv.
  \end{equation*}
  Moreover, for $v\in [v_m, C^{\frac{1}{2}}]$, there exists $\xi_v \in (v_m,v)$ such that
  \begin{equation*}
    \psi(v) = \psi(v_m) + \psi'(\xi_v)(v-v_m) = \psi'(\xi_v)(v-v_m) = \left( \frac{(p-1)(p-1-2\lambda)}{p^2} + O(C^{\frac{p}{4p-4}}) \right) (v-v_m),
  \end{equation*}
  using the uniform bound \eqref{eq: estimates for psiprime} in $(v_m, C^{\frac{1}{2}})$ for $\psi'(\xi_v)$, that is $O(C^{\frac{p}{4p-4}})$ can be taken independent of $v$.
  Hence
  \begin{equation*}
    \begin{split}
      I_1 &= C |J_0| ((p-1)a^{\frac{p}{2p-2}} - \lambda) \int_{v_m}^{C^{\frac{1}{2}}}  \frac{1}{v|J_0|^2 - p^2 C^2}\frac{1}{\sqrt{\left( \frac{(p-1)(p-1-2\lambda)}{p^2} + O(C^{\frac{p}{4p-4}}) \right) (v-v_m)}} dv \\
      &= \frac{C |J_0| ((p-1)a^{\frac{p}{2p-2}} - \lambda)}{\sqrt{\frac{(p-1)(p-1-2\lambda)}{p^2} + O(C^{\frac{p}{4p-4}})}}  \int_{v_m}^{C^{\frac{1}{2}}}  \frac{1}{v|J_0|^2 - p^2 C^2}\frac{1}{\sqrt{v-v_m}} dv  \\
      &= \frac{C |J_0| ((p-1)a^{\frac{p}{2p-2}} - \lambda)}{\sqrt{\frac{(p-1)(p-1-2\lambda)}{p^2} + O(C^{\frac{p}{4p-4}})}}   \frac{2}{\sqrt{|J_0|^2\bigl(|J_0|^2 v_m-p^2C^2\bigr)}} \arctan\!\left(\sqrt{\frac{|J_0|^2\bigl(C^{1/2}-v_m \bigr)}{|J_0|^2 v_m-p^2C^2}}\right),
    \end{split}
  \end{equation*}
  using the substitution $t=u-v_m$ and the standard integral $\int \frac{1}{at^2 +M^2}dt = \frac{1}{M\sqrt{a}}\arctan(\frac{t\sqrt{a}}{M})$. We also use \eqref{eq: approx for km denominator} to ensure positivity of the denominator.
  Passing to the limit as $(\lambda,C^2) \to (\lambda^*_p,0)$, we have $a\to 0$ and thus 
  \begin{equation*}
    \lim_{(\lambda,C^2) \to (\lambda^*_p,0)} I_1 = \frac{-2  \lambda^*_p }{\sqrt{\frac{(p-1)(p-1-2\lambda^*_p)}{p^2}}} \lim_{(\lambda,C^2) \to (\lambda^*_p,0)} \frac{C}{\sqrt{v_m |J_0|^2  - p^2 C^2}} \lim_{(\lambda,C^2) \to (\lambda^*_p,0)} \arctan\left( \frac{|J_0|\sqrt{C^{\frac{1}{2}} - v_m}}{\sqrt{|J_0|^2 v_m - p^2 C^2}} \right).
  \end{equation*}
  Use now the approximation \eqref{eq: approx for km denominator} to calculate
  \begin{equation*}
    \begin{split}
      \frac{C}{\sqrt{v_m|J_0|^2 - p^2 C^2}} &= \frac{C}{\sqrt{\frac{p^2\lambda^2}{(p-1)(p-1-2\lambda)}C^2 + O(C^{\min(\frac{3p-2}{p-1},4})}} \\
       &= \sqrt{\frac{(p-1)(p-1-2\lambda)}{p^2 \lambda^2}} +O(C^{\min(\frac{p}{p-1},2)}) \to \frac{\sqrt{(p-1)(p-1-2\lambda^*_p)}}{p \lambda^*_p},
    \end{split}
  \end{equation*}
as well as
\begin{equation*}
  \frac{|J_0|\sqrt{C^{\frac{1}{2}} - v_m }}{\sqrt{v_m |J_0|^2  - p^2 C^2}} = \frac{|J_0|\sqrt{C^{\frac{1}{2}} - v_m}}{\sqrt{\frac{p^2\lambda^2}{(p-1)(p-1-2\lambda)}C^2 + O(C^{\frac{6p-4}{2p-2}})}} = \frac{|J_0| }{\sqrt{\frac{p^2 \lambda^2}{(p-1)(p-1-2\lambda)}}} \frac{\sqrt{C^{\frac{1}{2}} -v_m}}{C} \left(1+ O(C^{\frac{p}{p-1}}) \right) \to \infty.
\end{equation*}
Hence, we obtain 
\begin{equation*}
  \lim_{(\lambda,C^2) \to (\lambda^*_p,0)} I_1 = \frac{-2  \lambda^*_p }{\sqrt{\frac{(p-1)(p-1-2\lambda^*_p)}{p^2}}} \frac{\sqrt{(p-1)(p-1-2\lambda^*_p)}}{p\lambda^*_p} \frac{\pi}{2} = -\pi.
\end{equation*}

\underline{Second term:} 
From \eqref{eq: estimate for psi in middle}, as $(\lambda,C^2) \to (\lambda^*_p,0)$, we directly have 
\begin{equation*}
\begin{split}
  |I_2| &\leq C |J_0| \frac{p-1-\lambda}{C^{\frac{1}{2}}|J_0|^2-p^2C^2} \int_{C^{\frac{1}{2}}}^{1-C^{\frac{1}{2}}} \frac{1}{\sqrt{\psi(u)}} \leq C |J_0| \frac{p-1-\lambda}{C^{\frac{1}{2}}|J_0|^2-p^2C^2} \frac{1}{\sqrt{c_p C^{\frac{1}{2}}}} = O(C^{\frac{1}{4}}) \to 0.
\end{split}
\end{equation*}

\underline{Third term:}
Similar to the first term, using \eqref{eq: estimates for psiprime}, for any $v \in [1-C^{\frac{1}{2}}, 1]$, there exists $ \xi_v \in [v,1]$ such that
\begin{equation*}
  \psi(v) = \psi(1) - \psi'(\xi_v)(1-v) = -\psi'(\xi_v)(1-v) = \left( \frac{p-1-\lambda}{p} + O(C^{\frac{p}{4p-4}}) \right) (1-v).
\end{equation*}
Thus, as $(\lambda,C^2) \to (\lambda^*_p,0)$, we have
\begin{equation*}
  \begin{split}
    |I_3| &\leq C |J_0| \frac{p-1-\lambda}{(1-C^{\frac{1}{2}})|J_0|^2-p^2C^2} \int_{1-C^{\frac{1}{2}}}^1 \frac{1}{\sqrt{\psi(v)}}dv \\
    &= C |J_0| \frac{p-1-\lambda}{(1-C^{\frac{1}{2}})|J_0|^2-p^2C^2} \int_{1-C^{\frac{1}{2}}}^1 \frac{1}{\sqrt{\left( \frac{p-1-\lambda}{p} + O(C^{\frac{p}{4p-4}}) \right) (1-v)}}dv \\
    &= C |J_0| \frac{p-1-\lambda}{(1-C^{\frac{1}{2}})|J_0|^2-p^2C^2}\frac{1}{\sqrt{\frac{p-1-\lambda}{p} + O(C^{\frac{p}{4p-4}}) }} \int_{1-C^{\frac{1}{2}}}^1 \frac{1}{\sqrt{ 1-v}}dv \\
    &= C |J_0| \frac{p-1-\lambda}{(1-C^{\frac{1}{2}})|J_0|^2-p^2C^2}\frac{2\sqrt{C^{\frac{1}{2}}}}{\sqrt{\frac{p-1-\lambda}{p} + O(C^{\frac{p}{4p-4}})}} = O(C^{\frac{5}{4}}) \to 0.
  \end{split}
\end{equation*}
We conclude that $\lim_{(\lambda,C^2) \to (\lambda^*_p,0)} \Delta \theta(\lambda,C^2) = -\pi $.
\end{proof}

We are now able to prove the main result, Theorem~\ref{thm: main result}.

\begin{proof}[Proof of Theorem~\ref{thm: main result}] The proof is split into several steps.

 \textbf{Step 1:} Proposition~\ref{prop: global growth of Delta z=0} gives a connected compact set $\Gamma$ connecting $(\lambda^*_p,0)$ and $(p-1,0)$ in $S^0$, along which $\Delta z$ vanishes. 
  Moreover, by Corollary~\ref{cor: asymptotics at p-1 of Gamma} with $\sigma= 3/2$, there exists $\delta_p>0$ such that $\Gamma$ lies below the curve
  $$\lambda \mapsto \left(\lambda,(p-1-\lambda)^{3/2}\right)\Big|_{[p-1-\delta_p,p-1]}$$ 
  in $S^0$. 
  Thus, by Propositions~\ref{prop: theta limit at figure 8} and \ref{prop: Delta theta at right endpoint}, along $\Gamma$ we have $\Delta \theta(\lambda, C^2)  \to -\pi$  at the endpoint $(\lambda^*_p,0)$ and $\Delta \theta(\lambda, C^2)  \to 0$  at the endpoint $(p-1,0)$. 
  Since $\Delta \theta$ is continuous along $\Gamma$ inside $S^0$ (Proposition~\ref{prop: continuiy of Delta theta}), any value in $(-\pi,0)$ is attained. 
  Hence, we conclude that for each $\frac{n}{m}=q \in (0,\frac{1}{2}) \cap \Q$, there exists $(\lambda,C^2) \in S$ such that for the corresponding curve $\gamma_{\lambda,C^2}$ solving \eqref{eq:EL for k and tau}, $\Delta z(\lambda,C^2)=0$ and $\Delta \theta(\lambda,C^2) = -2\pi \frac{n}{m}$. 
  In other words, $\gamma_{\lambda,C^2}$ closes in $m$ periods of its curvature while making exactly $n$ turns around the cylindrical axis, i.e.\ $\gamma_{\lambda,C^2}$ has $m$-fold rotational symmetry.

 \textbf{Step 2:} By Proposition~\ref{prop: global growth of Delta z=0}, the minimal radius $r_\mathcal{C}$ is strictly positive and thus $\gamma_{\lambda,C^2}$ does not cross the cylindrical axis.
 We show that the curve $s\mapsto (r(s),z(s))$ is a regular simple closed curve of period $2P(\lambda,C^2)$, i.e.\ $\gamma_{\lambda,C^2}$ lies on a torus of revolution.
 From Proposition~\ref{prop: derivatives for cylindrical coordinates}, $r$ and $z$ are periodic and have only two critical points in each period of the curvature (namely minimum and maximum) and $(r_s(s),z_s(s))\neq(0,0)$. 
 First, since $r(s) = \frac{p}{|J_0|^2}\sqrt{|J_0|^2k^{2p-2}(s)-p^2C^2}$, it is $2P$ periodic and strictly decreasing on $(0,P)$, strictly increasing on $(P,2P)$ and has maximum at $s=0$ and minimum at $s=P$.
 Hence the curve, restricted to $[0,P]$ or $[P,2P]$, is graphical and it remains to show that $z(s_1) \neq z(s_2)$ for any $s_1 \in (0,P)$ and $s_2 \in (P,2P)$. 
 On $(0,P)$, the function $z(t)=\int_0^t \frac{(p-1)k(s)^p-\lambda}{|J_0|}ds$ has a strictly decreasing integrand by Proposition~\ref{prop: existence} and Remark~\ref{rmk: k prime} (as $k(s)$ decreases monotonically from $k_M=1$ to $k_m$) and $z(0)=z(P)=0$. Thereby $z(t)>0$. 
 Analogously on $(P,2P)$, $z(t)<0$ and we conclude that curve is simple and closed.

 \textbf{Step 3:} It remains to show that $\gamma_{\lambda,C^2}$ itself is embedded. Suppose that there exists $s_1,s_2 \in [0,2mP)$ with $s_1<s_2$ and $a \in \Z$ such that 
 \begin{equation*}
  (r(s_1), z(s_1), \theta(s_1)) =  \gamma_{\lambda,C^2}(s_1) = \gamma_{\lambda,C^2}(s_2) = (r(s_2), z(s_2), \theta(s_2)+2\pi a). 
 \end{equation*}
 By periodicity, we may assume $s_1 \in [0,2P(\lambda,C^2))$.
 By the previous step, $s\mapsto (r(s),z(s))$ is injective on $[0,2P(\lambda,C^2))$, i.e.\ $(r(s_1), z(s_1)) = (r(s_2), z(s_2))$ implies that $s_2=s_1 + 2bP(\lambda,C^2)$ for some $\N \ni b<m$.
 That is, $\theta(s_2)-\theta(s_1) = b\Delta\theta = -2\pi \frac{bn}{m} \notin 2\pi \Z$, as $1\leq b<m$ and $\gcd(n,m)=1$.
  We conclude that $\gamma_{\lambda,C^2}$ is embedded and equivalent to the $(n,m)$-torus knot.
\end{proof}

We remark several interesting facts about the result.
\begin{remark}
  When approaching the endpoint $(p-1,0)$ along $\Gamma$, the torus knot $\gamma_{\lambda,C^2}$ winds tighter and tighter and, restricted to its first curvature period, $\gamma_{\lambda,C^2}$ looks almost like a circle orthogonal to $\frac{\partial}{\partial \theta}$, with a small gap at its endpoints. When approaching the endpoint $(\lambda^*_p,0)$ on the other hand, the inner radius of the circle degenerates and the first curvature period of $\gamma_{\lambda,C^2}$ resembles a figure-eight (with a small gap at its endpoints), again orthogonal to $\frac{\partial}{\partial \theta}$. This is depicted in Figure~\ref{fig:torus different knots} and Figure~\ref{fig:cs different knots}.

More precisely, as $(\lambda,C^2)\to(\lambda^*_p,0)$, we have $r_\mathcal{C} \approx \frac{p|J_0| k_m^{p-1}}{|J_0|} \approx \frac{p(p-1-\lambda_p^*)k_m^{p-1}}{(p-1-\lambda_p^*)^2} \to 0$ since $k_m \to 0$. 
On the other hand, $R_\mathcal{C} \to \frac{p}{(p-1-\lambda_p^*)}$, a finite value. 
  For $(\lambda,C^2)\to(p-1,0)$, from Corollary~\ref{prop: wm around S2:2}, 
  \begin{equation*}
      \sqrt{|J_0|^2 k_m^{2p-2} -p^2C^2} \geq \sqrt{|J_0|^2 (1-O(|p-1-\lambda|)) -p^2C^2} = \sqrt{(p-1-\lambda)^2 - O(|p-1-\lambda|^{1+\sigma})} \geq \frac{1}{2} (p-1-\lambda),
  \end{equation*}
 using the fact that $C^2 \leq (p-1-\lambda)^\sigma$ along $\Gamma$.
 Thus from \eqref{eq: r_C definition}, 
  \begin{equation*}
  R_\mathcal{C} \geq r_\mathcal{C} = \frac{p\sqrt{|J_0|^2k_m^{2p-2} -p^2C^2}}{|J_0|^2} 
  \geq \frac{\frac{p}{2} (p-1-\lambda)}{(p-1-\lambda)^2+p^2C^2} \geq \frac{1}{4p} |p-1-\lambda|^{1 -\sigma} \to \infty.
  \end{equation*}
  Moreover, from \eqref{eq: estimate for um},
  \begin{equation*}
    |\theta_s(s)| = \left|\frac{C|J_0|((p-1)k(s)^p-\lambda)}{|J_0|^2k(s)^{2p-2}-p^2C^2} \right| \leq  \frac{C |J_0|O(|p-1-\lambda|)}{\left| |J_0|^2k_m^{2p-2}-p^2C^2 \right|}  \leq \frac{ C |J_0|O(|p-1-\lambda|)}{\frac{1}{2}(p-1-\lambda)^2} = O(|p-1-\lambda|^{\sigma/2 + \sigma/2 + 1 -2}) \to 0,
  \end{equation*} 
  the  angle $\theta$ stays almost constant in one (uniformly bounded, see Proposition~\ref{prop: continuity of period}) curvature period. 
  That is, $\gamma_{\lambda,C^2}$ does a full revolution around the axis of the torus of revolution, while sweeping out only a small angle. 
  This in turn means that $\gamma_{\lambda,C^2}$ has to be  close to a geodesic on the torus of revolution, i.e. an arc orthogonal to $\frac{\partial}{\partial \theta}$. Since $k_m \to 1$, the geodesic and therefore also $\gamma_{\lambda,C^2}$ (restricted to one curvature period) are close to a circle of radius $1$.
\end{remark}

\begin{figure}[ht]
    \centering
    \begin{subfigure}{0.24\textwidth}
        \centering
        \includegraphics[trim={80px 100px 80px 100px}, clip, width=\linewidth]{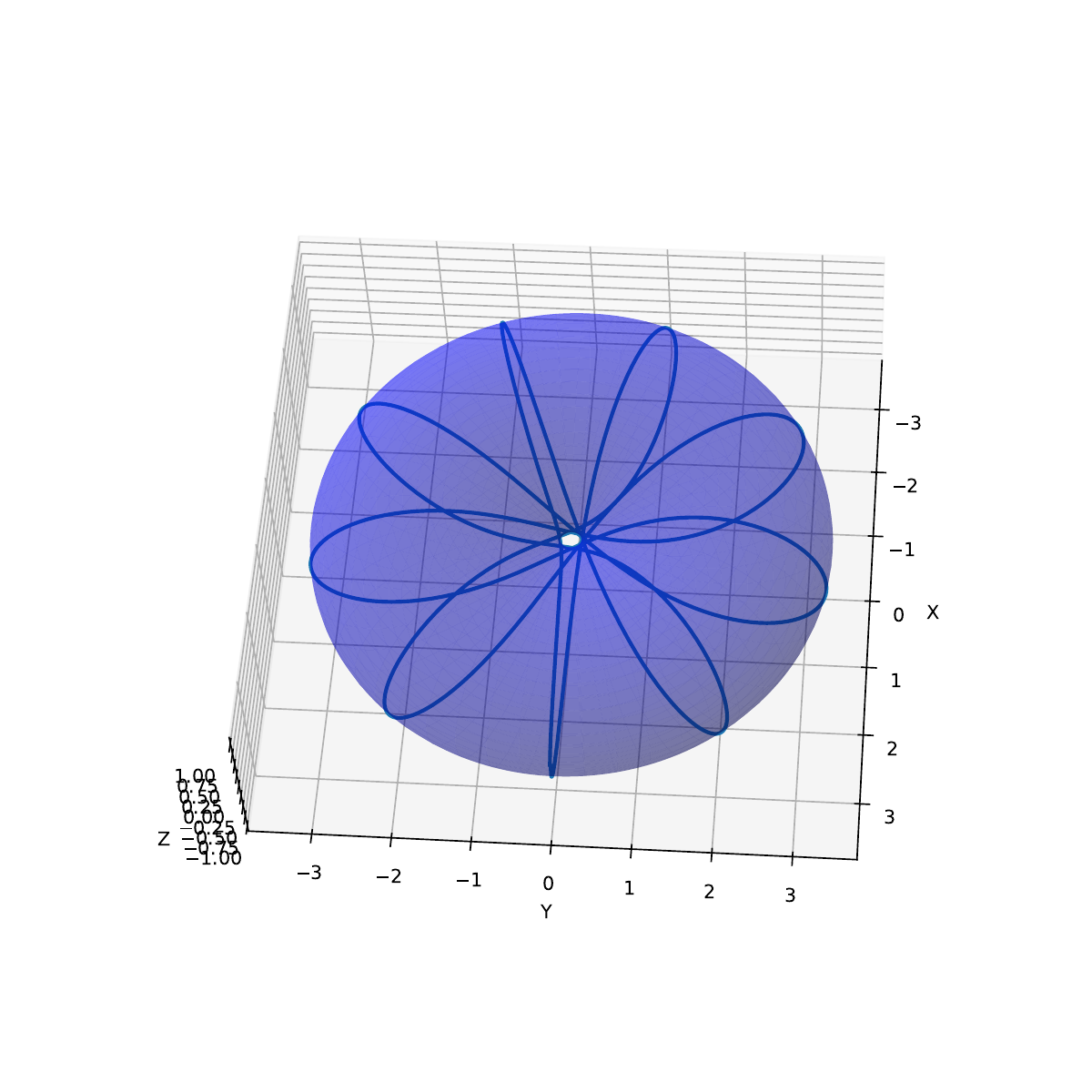}
        \caption{$q= \frac{n}{m}=\frac{4}{9}$}
    \end{subfigure}
    \hfill
    \begin{subfigure}{0.24\textwidth}
        \centering
        \includegraphics[trim={80px 100px 80px 100px}, clip, width=\linewidth]{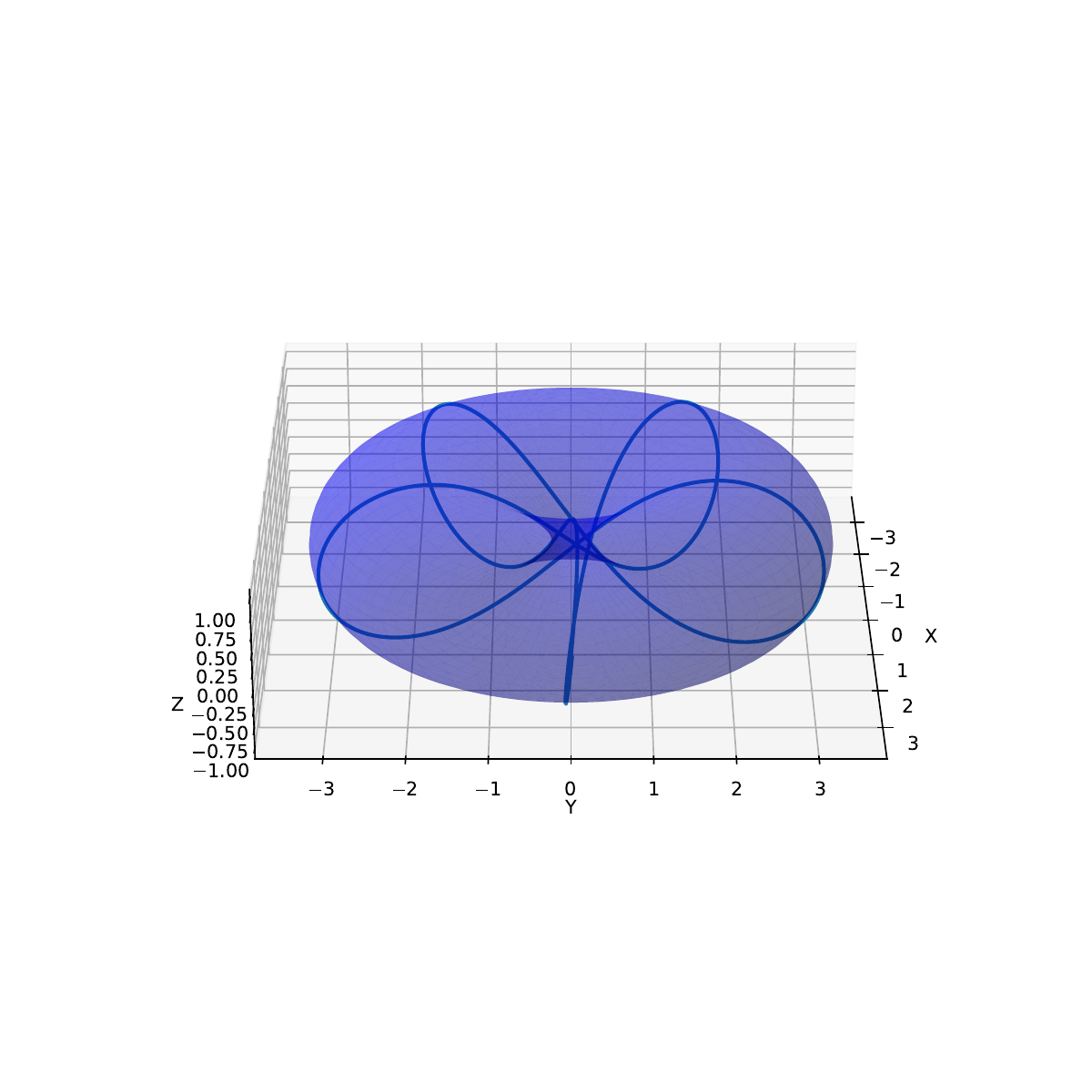}
        \caption{$q= \frac{n}{m}=\frac{2}{5}$}
    \end{subfigure}
    \hfill
    \begin{subfigure}{0.24\textwidth}
        \centering
        \includegraphics[trim={80px 100px 80px 100px}, clip, width=\linewidth]{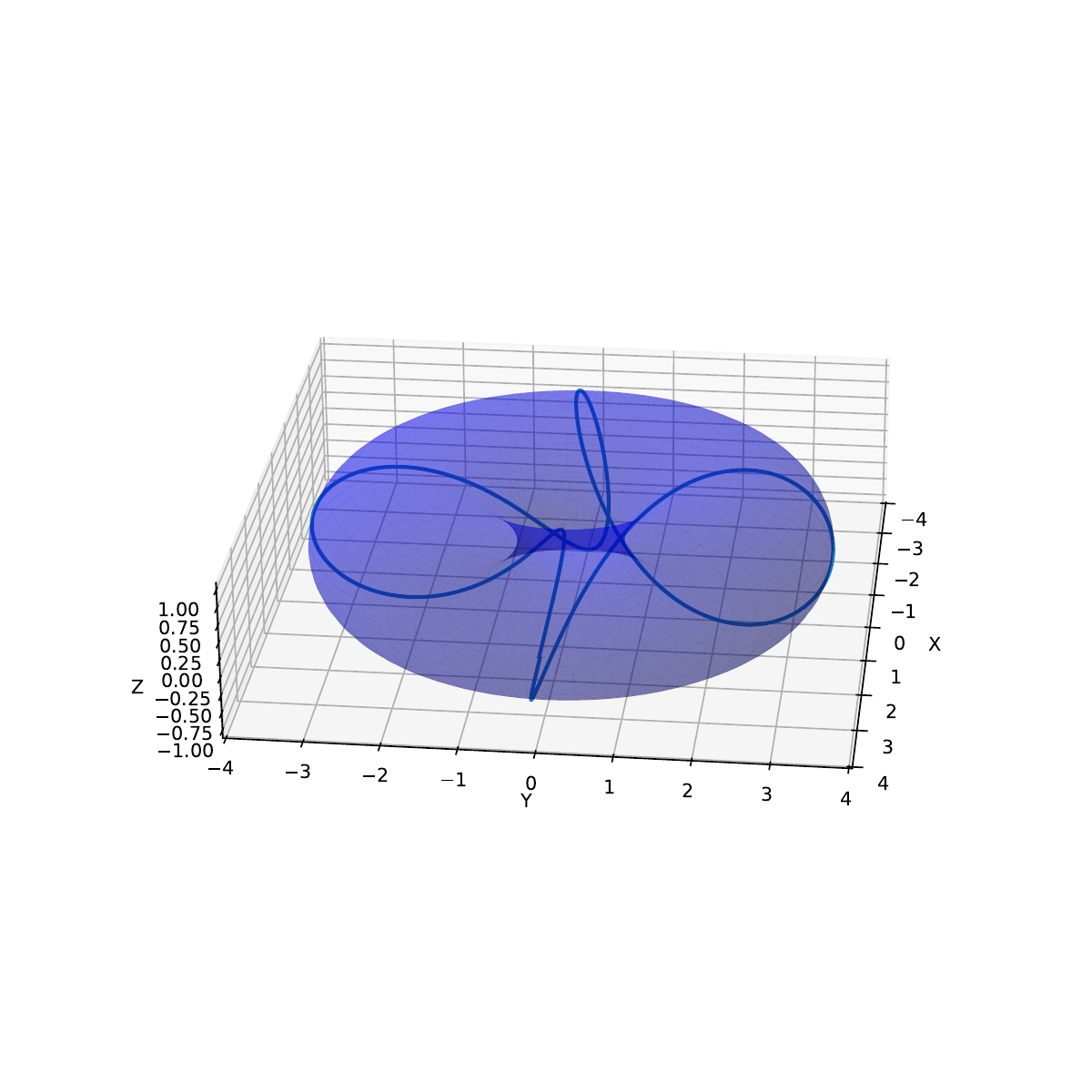}
        \caption{$q= \frac{n}{m}=\frac{1}{4}$}
    \end{subfigure}
        \hfill
    \begin{subfigure}{0.24\textwidth}
        \centering
        \includegraphics[trim={80px 100px 80px 100px}, clip, width=\linewidth]{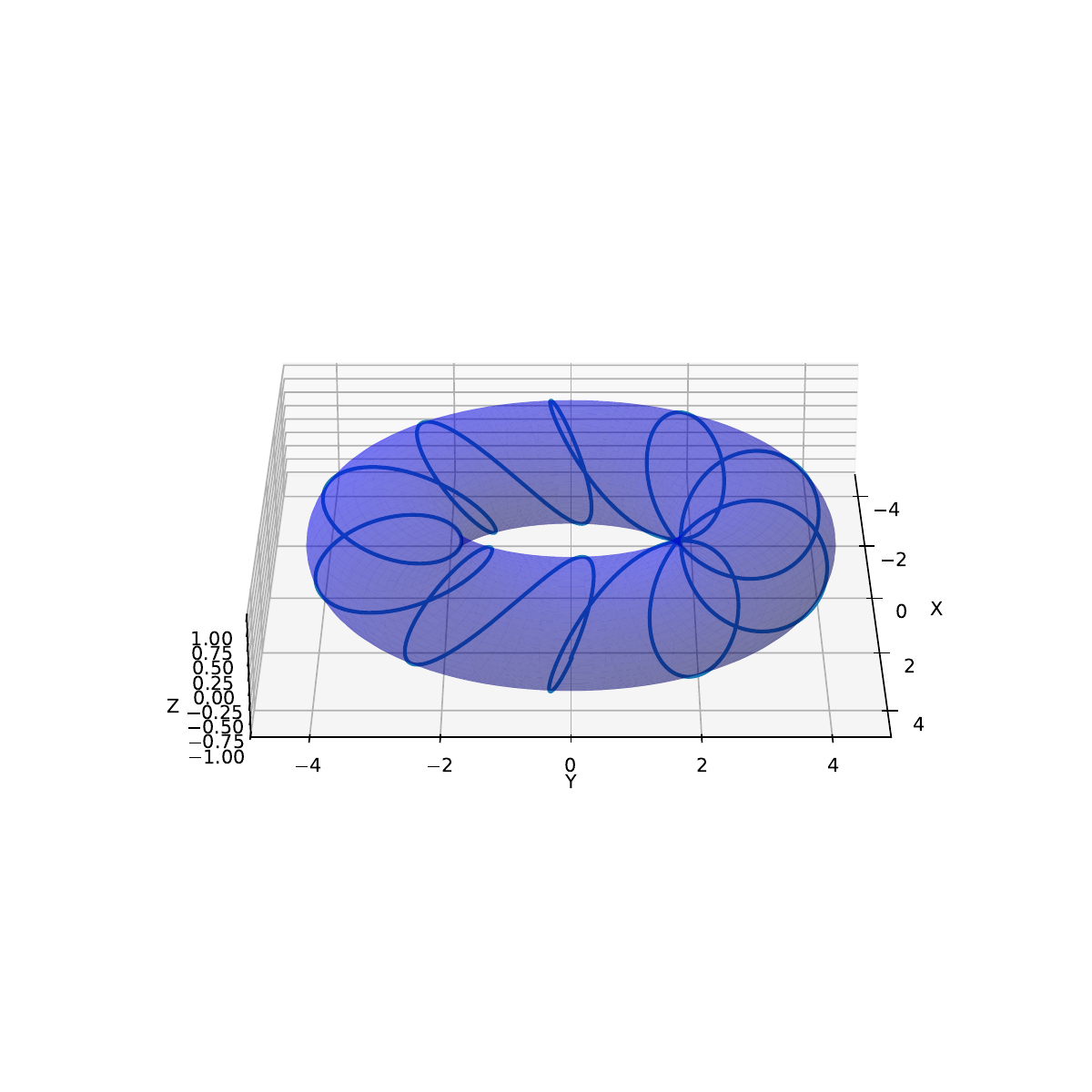}
        \caption{$q= \frac{n}{m}=\frac{1}{10}$}
    \end{subfigure}
    \caption{Different $(n,m)$-torus knot $p$-elasticae for $p=2$: For $q = \frac{n}{m}\to0$, the torus knot winds tightly around the torus, for $q= \frac{n}{m}\to \frac{1}{2}$, it resembles interlaced $p$-figure-eight curves.}
    \label{fig:torus different knots}
\end{figure}

\begin{figure}[ht]
    \centering
    \begin{subfigure}{0.24\textwidth}
        \centering
        \includegraphics[trim={30px 30px 30px 135px}, clip, width=\linewidth]{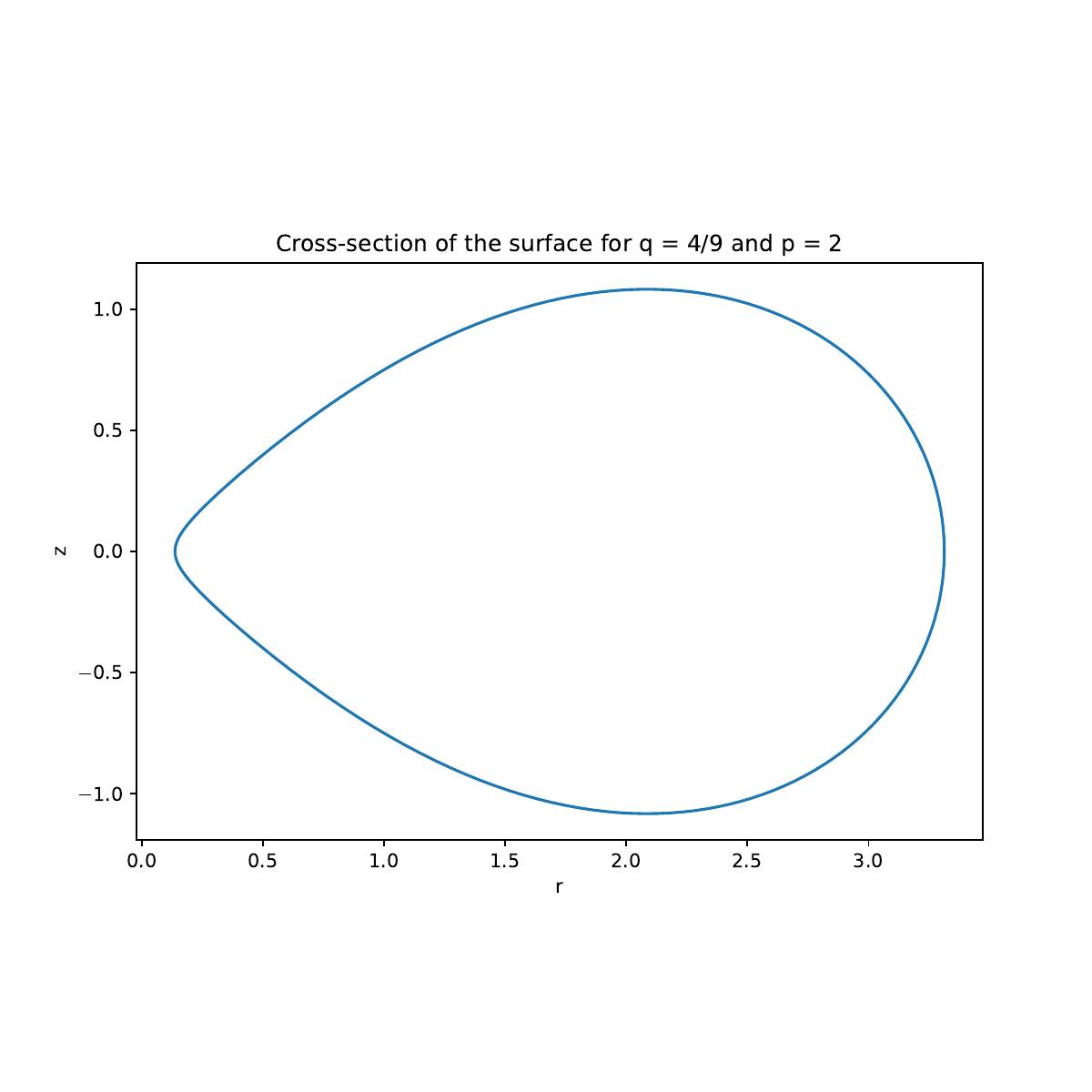}
        \caption{$q= \frac{n}{m}=\frac{4}{9}$}
    \end{subfigure}
    \hfill
    \begin{subfigure}{0.24\textwidth}
        \centering
        \includegraphics[trim={30px 30px 30px 130px}, clip, width=\linewidth]{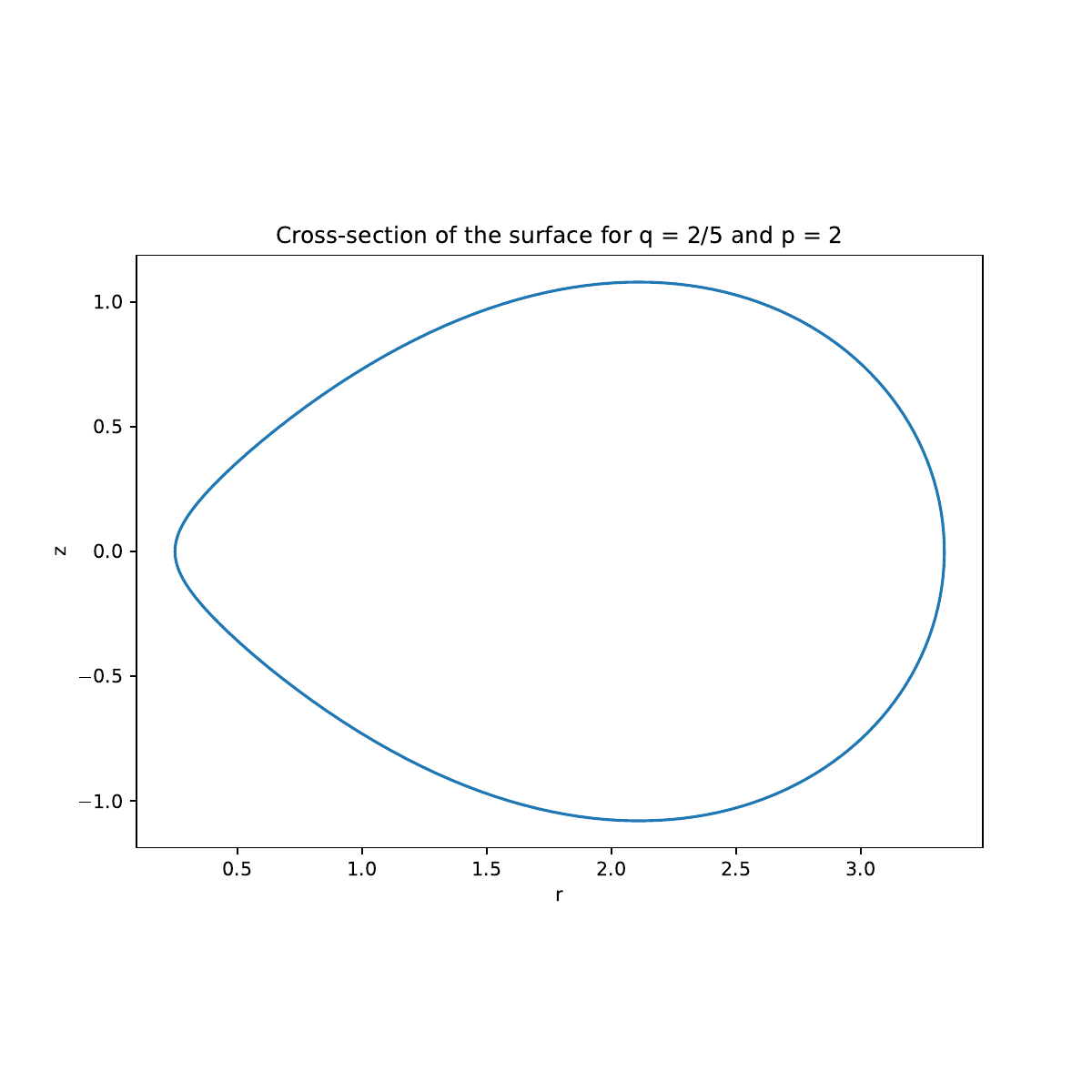}
        \caption{$q= \frac{n}{m}=\frac{2}{5}$}
    \end{subfigure}
    \hfill
    \begin{subfigure}{0.24\textwidth}
        \centering
        \includegraphics[trim={30px 30px 30px 118px}, clip, width=\linewidth]{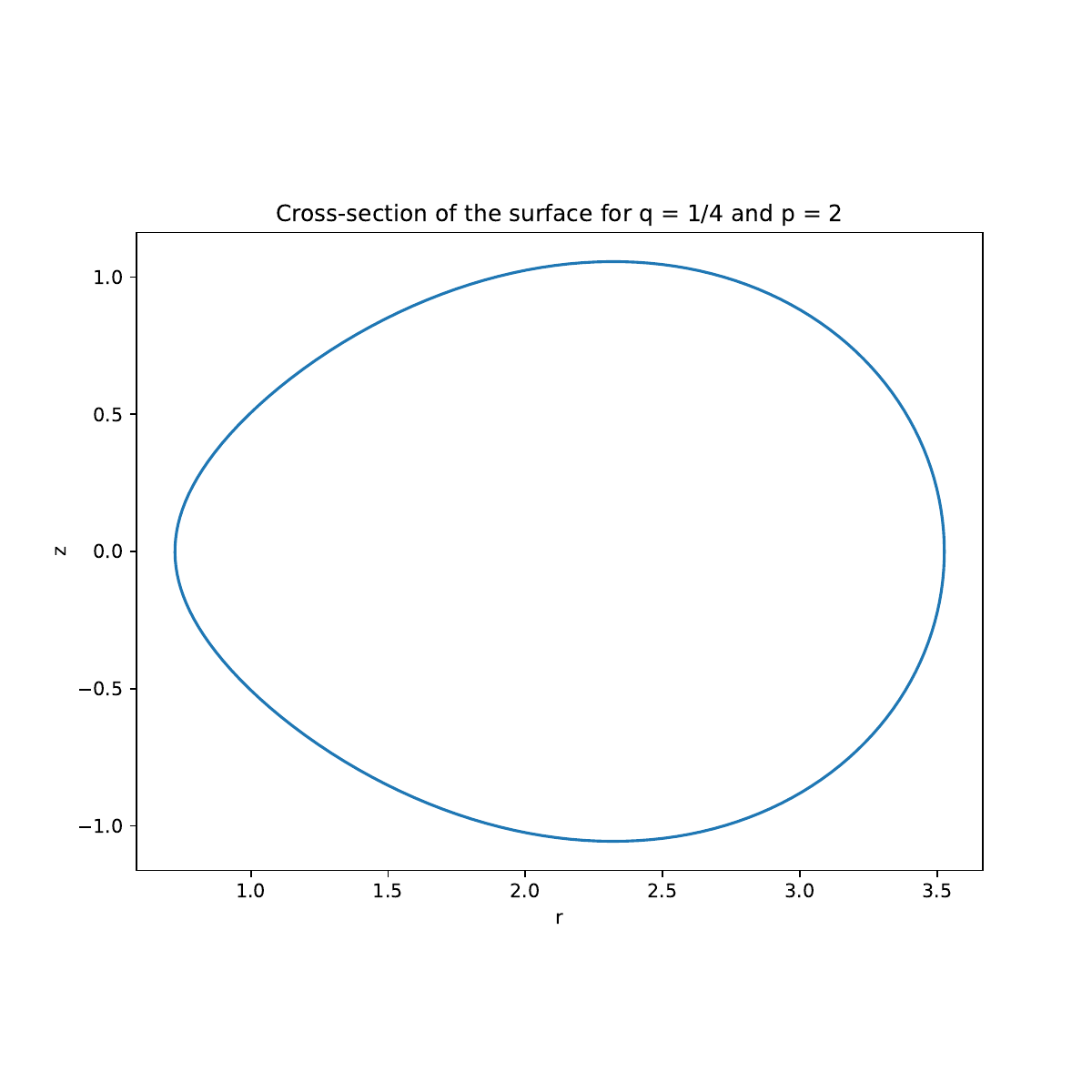}
        \caption{$q= \frac{n}{m}=\frac{1}{4}$}
    \end{subfigure}
        \hfill
    \begin{subfigure}{0.24\textwidth}
        \centering
        \includegraphics[trim={30px 30px 30px 110px}, clip, width=\linewidth]{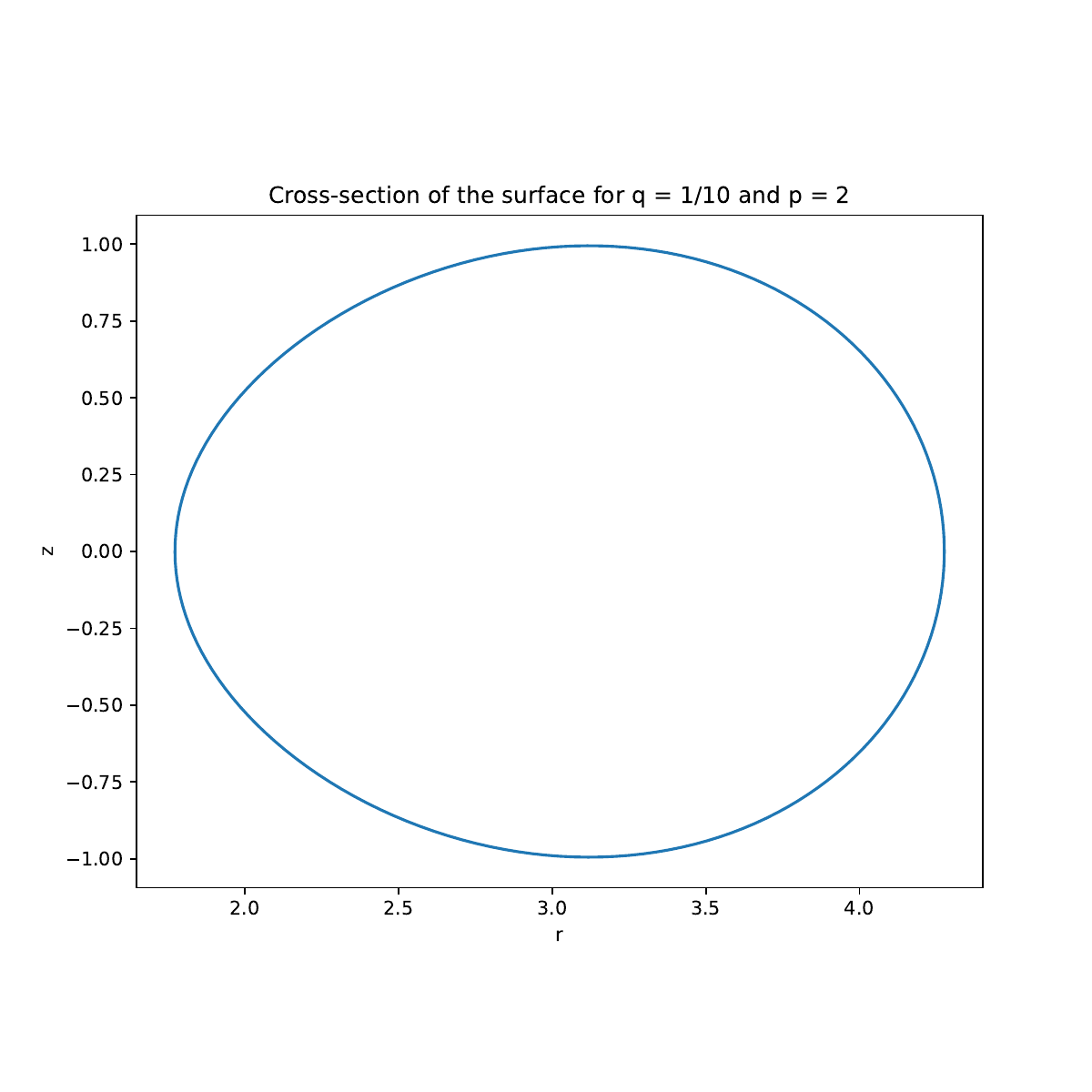}
        \caption{$q= \frac{n}{m}=\frac{1}{10}$}
    \end{subfigure}
    \caption{Cross sections of different torus knot $p$-elasticae for $p=2$: For $q= \frac{n}{m}\to 0$, the cross section becomes circular, for $q= \frac{n}{m}\to \frac{1}{2}$ a leaf of the $p$-figure-eight.}
    \label{fig:cs different knots}
\end{figure}

\begin{remark}
  Concerning the uniqueness, in the case $p=2$, \cite{langer-singer_classification} uses the explicit solution formula to differentiate $\Delta \theta$ along $\{\Delta z=0\}$ by lengthy calculations involving elliptic integrals. Such an approach is not fruitful here, mainly due to two issues:
  \begin{itemize}
    \item The set $\{\Delta z=0\}$ is not completely characterized, Proposition~\ref{prop: Delta z=0 is continuous curve} only gives existence of a connected subset within $\{\Delta z=0\}$. One way to show $\Gamma = \{\Delta z=0\}$ would be to show $\frac{\partial }{\partial C^2} \Delta z >0$ on $\{\Delta z=0\}$.
    This would be even stronger, since it implies that $\{\Delta z=0\}$ is graphical.
    So far only monotonicity of $u_m$ (thus also $k_m$) is known, i.e.\ the function $C^2 \mapsto u_m(C^2)$ is strictly increasing.
  Indeed, since $h$ is smooth in $u$ and its parameters away from the origin, by the implicit function theorem on $0= h(u_m(C^2))$,
  \begin{equation*}
   0= \frac{\partial h}{\partial u} (u_m(C^2)) \frac{\partial u_m}{\partial C^2}(C^2) + \frac{\partial h}{\partial C^2} (u_m(C^2)),
  \end{equation*}
  hence $\frac{\partial u_m}{\partial C^2}(C^2) = - \frac{\frac{\partial h}{\partial C^2} (u_m(C^2))}{\frac{\partial h}{\partial u} (u_m(C^2))}$. 
  As $\frac{\partial h}{\partial C^2} (u_m(C^2)) = 1-\frac{u_m^{\frac{2}{p}}}{u_m^2}<0$ and $\frac{\partial h}{\partial u} (u_m(C^2))>0$ (Lemma~\ref{lemma: number of solutions}), we have $\frac{\partial u_m}{\partial C^2}(C^2)>0$.
    
    \item In \cite{langer-singer_classification}, an explicit formulation of $\{\Delta z=0\}$ and even $\Delta \theta$ (using elliptic integrals and the Heuman $\Lambda_0$ function) is derived, but here we only have the integral formulations \eqref{eq: Delta z original} and \eqref{eq: formula for Delta theta}, which make an analytic calculation for the derivative of $\Delta \theta$  along $\{\Delta z=0\}$ very complicated. It might be possible to show that $\frac{\partial}{\partial \lambda} \Delta \theta >0$, which then combined with the possible graphicality of $\{\Delta z=0\}$ would give uniqueness.
  \end{itemize}
  However, numerical evidence suggests that the closed elasticae from Theorem~\ref{thm: main result} are indeed unique.
  This would imply that each knot class contains a unique closed $p$-elastica (up to similarity transforms), except for the unknot class, which contains infinitely many closed $p$-elasticae. (The fact that the unknot class contains infinitely many closed $p$-elasticae follows already from our existence theorem.) 
  
  Minimizing the elastic energy within a fixed knot class has seen substantial recent progress, see e.g.\ \cites{vonderMosel1, vonderMosel2, Blatt_knot, knots_computation}, usually using a combination of an elastic energy and the celebrated \textit{O'Hara knot energy} to prevent self-intersections.
\end{remark}

\begin{remark}
  For the quadratic case $p=2$, Langer and Singer prove in a subsequent paper \cite{langersinger_minmax} by a minmax argument that all such closed non-planar elastica are unstable (only the once covered circle is stable) in $\R^3$. We conjecture that an analogous result holds for general exponents $p\in (1,\infty)$. 
\end{remark}

\appendix
\section{Appendix}
We have two short elementary lemmas.

\begin{lemma}
\label{lemma: unique root}
  Let $f:[0,\infty)\to \R$ be given as
  \begin{equation*}
    f(x) := \frac{p-1}{p} x^{\frac{p+1}{p-1}} - \frac{\lambda}{p} x^{\frac{1}{p-1}} - C^2 x^{-3}
  \end{equation*}
  with $\lambda \in  \R$ and $C^2>0$. 
  Then $f$ has only one strictly positive root. The sign change at this strictly positive root is negative-positive.
\end{lemma}
\begin{proof}
  We show that only one other positive root exists. Any such root will also be a root to $g(x):=x^{\frac{-1}{p-1}} f(x)$. However, for $x>0$,
  \begin{equation*}
    g'(x) = \left( \frac{p-1}{p} x^{\frac{p}{p-1}} - \frac{\lambda}{p} - C^2 x^{-3-\frac{1}{p-1}} \right)' = x^{\frac{1}{p-1}} + \left(3+\frac{1}{p-1}\right)C^2 x^{-4-\frac{1}{p-1}} > 0,
  \end{equation*}
  and since $\lim_{x\to 0^+} g(x) = -\infty$ and $\lim_{x\to \infty} g(x) = \infty$, the result follows by the intermediate value theorem and the positivity of $g'$.
\end{proof}

\begin{lemma}
  \label{lemma: number of solutions}
  The function $V:[0,\infty) \to \R$ given as
  \begin{equation*}
    V(x) := \frac{(p-1)^2}{p^2}  x^{\frac{2p} {p-1}} - 2\lambda \frac{p-1}{p^2}x^{\frac{p}{p-1}} + C^2 x^{-2},
  \end{equation*}
  with $\lambda\in \R$ and $C^2> 0$ is strictly unimodal, i.e.\ strictly decreasing until it reaches its minimum and then strictly increasing. 
  In particular, $g(x) = E-V(x)$ has zero, one or two real positive solutions for $E<\min V$, $E=\min V$ and $E>\min V$ respectively and is strictly increasing until it reaches the maximum and then strictly decreasing.
\end{lemma}
\begin{proof}
  Since 
  \begin{equation*}
    V'(x) = 2 \left( \frac{p-1}{p} x^{\frac{p+1}{p-1}} - \frac{\lambda}{p} x^{\frac{1}{p-1}} - C^2 x^{-3} \right),
  \end{equation*}
  we apply Lemma~\ref{lemma: unique root}.
\end{proof}

We also have asymptotic estimates on the derivative of $g$ at its zeros.
\begin{lemma}
  \label{lemma: estimates for gprime and gbis at 1}
  We have
  \begin{equation*}  
     g'(1) = -2\left(\tfrac{p-1-\lambda}{p}-C^2 \right), \qquad g''(1) = -2 - \frac{2}{p(p-1)}(p-1-\lambda) -6C^2.
  \end{equation*}
\end{lemma}
\begin{proof}
  This follows from
  \begin{equation}
    \begin{split}
      \label{eq: derivates for g in full form}
      g'(w) &= -2\frac{p-1}{p} w^{\frac{p+1}{p-1}} + 2\frac{\lambda}{p}w^{\frac{1}{p-1}} + 2C^2 w^{-3}, \\
      g''(w) &= -2\frac{p+1}{p} w^{\frac{2}{p-1}} + 2\frac{\lambda}{p(p-1)}w^{\frac{2-p}{p-1}} + 2(-3)C^2 w^{-4}.
    \end{split}
  \end{equation}
\end{proof}

\begin{lemma}
  \label{lemma: estimate for gprime at wm}
  Let $\gamma \in (1,2)$ and $(\lambda,C^2)\in S^0$ such that Proposition~\ref{prop: wm around S2:2} holds. Then 
  \begin{equation*}
    \begin{split}
      g'(w_m) = 2\left(\tfrac{p-1-\lambda}{p}-C^2 \right) + O(|\tfrac{p-1-\lambda}{p}-C^2|^\gamma) = -g'(1) + O(|\tfrac{p-1-\lambda}{p}-C^2|^\gamma).
    \end{split}
  \end{equation*}
\end{lemma}
\begin{proof}
  Using Proposition~\ref{prop: wm around S2:2},
  \begin{equation*}
    w_m = 1- \frac{-4}{g''(1)} \left(\tfrac{p-1-\lambda}{p}-C^2\right) + O(|\tfrac{p-1-\lambda}{p}-C^2|^\gamma),
  \end{equation*}
  in \eqref{eq: derivates for g in full form} we obtain, 
  \begin{equation*}
  \begin{split}
    g'(w_m) &= -2\frac{p-1}{p} \left(1+ \frac{4}{g''(1)} \left(\tfrac{p-1-\lambda}{p}-C^2\right) + O(|\tfrac{p-1-\lambda}{p}-C^2|^\gamma)\right)^{\frac{p+1}{p-1}} \\ 
    &\quad + 2\frac{\lambda}{p} \left( 1+ \frac{4}{g''(1)} \left(\tfrac{p-1-\lambda}{p}-C^2\right) + O(|\tfrac{p-1-\lambda}{p}-C^2|^\gamma) \right)^{\frac{1}{p-1}} \\
    &\quad + 2C^2 \left(1+ \frac{4}{g''(1)} \left(\tfrac{p-1-\lambda}{p}-C^2\right) + O(|\tfrac{p-1-\lambda}{p}-C^2|^\gamma)\right)^{-3} \\
    &= -2\frac{p-1}{p} \left( 1 + \frac{4}{g''(1)} \frac{p+1}{p-1} \left(\tfrac{p-1-\lambda}{p}-C^2\right) + O(|\tfrac{p-1-\lambda}{p}-C^2|^\gamma)\right) \\
    &\quad + 2\frac{\lambda}{p}\left(1 + \frac{4}{g''(1)} \frac{1}{p-1} \left(\tfrac{p-1-\lambda}{p}-C^2\right) + O(|\tfrac{p-1-\lambda}{p}-C^2|^\gamma)\right) \\
    &\quad + 2C^2 \left(1 -3 \frac{4}{g''(1)} \left(\tfrac{p-1-\lambda}{p}-C^2\right) + O(|\tfrac{p-1-\lambda}{p}-C^2|^\gamma)\right) \\
    &= \left(\tfrac{p-1-\lambda}{p}-C^2 \right) \left( -2 -\frac{8}{g''(1)} \frac{p+1}{p} + \frac{8\lambda}{g''(1)p(p-1)} - \frac{24}{g''(1)} C^2 \right)+ O(|\tfrac{p-1-\lambda}{p}-C^2|^\gamma)\\
    &= \left(\tfrac{p-1-\lambda}{p}-C^2 \right) \left( -2 -\frac{8}{g''(1)} \left( \frac{p+1}{p} - \frac{\lambda}{p(p-1)} + 3C^2 \right) \right) + O(|\tfrac{p-1-\lambda}{p}-C^2|^\gamma) \\
    &= 2\left(\tfrac{p-1-\lambda}{p}-C^2 \right) + O(|\tfrac{p-1-\lambda}{p}-C^2|^\gamma).\\
  \end{split}
  \end{equation*}
  We have used Lemma~\ref{lemma: estimates for gprime and gbis at 1}  in the last equality.
\end{proof}

\subsection{Classical Results}
Recall the following classical results from differential geometry and real analysis.
\begin{proposition}
  \label{prop: characterization of Killing field}
  \cite[Proposition~1]{langer-singer_classification}. A vector field $J$ in $\R^3$ along a curve $\gamma \subset \R^3$ extends to a Killing vector field on $\R^3$ if and 
  only if 
  \begin{equation*}
    0=\langle J',T \rangle = \langle J'',N \rangle = \langle J'''-\frac{k'}{k}J''+k^2 J',B \rangle.
  \end{equation*}
\end{proposition}

\begin{theorem}[Weighted Mean Value Theorem for Integrals]
\label{thm: weighted mvt}
Let $f: [a, b] \to \mathbb{R}$ be a continuous function and let $g: [a, b] \to \mathbb{R}$ be an integrable function that does not change sign on the interval $[a, b]$. 
Then there exists at least one point $c \in (a, b)$ such that
\begin{equation*}
    \int_{a}^{b} f(x)g(x) \, dx = f(c) \int_{a}^{b} g(x) \,dx.
\end{equation*}
\end{theorem}

Recall also the following special cases of Taylor's theorem.
\begin{remark}
  \label{rmk:first order Taylor}
  Let $a,b,r \in \R$, $s<t$ and $\eps \ll 1$. Then a first order Taylor expansion gives
  \begin{equation*}
    \begin{split}
      (a\eps^s + b\eps^t)^r &= a^r \eps^{rs} + ra^{r-1} b \eps^{t+(r-1)s} + \frac{r(r-1)}{2} a^{r-2}b^2 \eps^{2t+(r-2)s} \left(1+\xi\frac{b}{a}\eps^{t-s} \right)^{r-2}\\
      &= a^r \eps^{rs} + ra^{r-1} b \eps^{t+(r-1)s} +O(\eps^{2t+(r-2)s}),
    \end{split}
  \end{equation*}
  where $\xi \in [0,1]$, that is the constant in $O(\eps^{2t+(r-2)s})$ depends on $r,s,t$ and on $a$,$b$ in form of a power-law.
\end{remark}

\begin{remark}
  \label{rmk: Taylor for (1+x)to the r}
  Let $\eps\ll 1$, $a,b \in \R$ and $\gamma>1$. Then from
   a second order Taylor expansion around $1$, there exists $\xi$ with $
   |\xi| \leq |a\eps+b\eps^\gamma|$ such that
\begin{equation*}
  \begin{split}
  (1 +a\eps +b \eps^\gamma)^r &= 1 +ra\eps +rb\eps^\gamma + \frac{r(r-1)}{2} (a\eps+b\eps^\gamma)^2 + \frac{r(r-1)(r-2)}{6} (1+\xi)^{r-3} (a\eps+b\eps^\gamma)^3 \\
    & = 1 + \eps[ra] + \eps^\gamma [rb] + \eps^2 \left[\frac{r(r-1)}{2}a^2 \right] + \eps^{1+\gamma} [r(r-1) ab]  + O(\eps^{\min(2\gamma,3)}).
  \end{split}
\end{equation*}
The constant in $O(\eps^{\min(2\gamma,3)})$ depends only polynomially on $r$, $a$, $b$, since for $\eps$ sufficiently small $(1+\xi)^{r-3}$ remains bounded (e.g.\ for $\eps<\frac{1}{2(a+b)}$, the inequality $|\xi|\leq \eps|a + b\eps^{1-\gamma}|\leq \frac{1}{2}$ holds), regardless of $r$ and so can be absorbed into the constant.
\end{remark}

\bibliographystyle{plain}
\bibliography{references}

\end{document}